\documentclass[10pt]{amsart}
\usepackage{amsmath}
\usepackage{amsfonts}
\usepackage{amssymb}
\usepackage{amsthm}
\usepackage{url}
\usepackage{tikz-cd}
\usepackage{dsfont}
\usepackage{graphicx}
\usepackage{caption}
\usepackage{subcaption}
\usepackage{comment}
\usepackage{stmaryrd}
\usepackage{hyperref}
\usepackage{color}
\usepackage{enumerate}
\usepackage{mathrsfs}

\numberwithin{equation}{section}

\newtheorem{proposition}{Proposition}[section]
\newtheorem{lemma}[proposition]{Lemma}
\newtheorem{theorem}[proposition]{Theorem}
\newtheorem{corollary}[proposition]{Corollary}

\theoremstyle{definition}
\newtheorem{remark}[proposition]{Remark}
\newtheorem{definition}[proposition]{Definition}

\DeclareMathOperator{\End}{End}
\DeclareMathOperator{\Id}{Id}
\DeclareMathOperator{\tr}{tr}

\DeclareMathOperator{\GL}{GL}
\DeclareMathOperator{\Aut}{Aut}

\DeclareMathOperator{\Hom}{Hom}

\newcommand{\dbar}{\bar\partial}
\newcommand{\scG}{\mathscr{G}}
\newcommand{\scGC}{\mathscr{G}^{\mathbb{C}}}
\newcommand{\scH}{\mathscr{H}}
\newcommand{\scN}{\mathcal{N}}
\newcommand{\scP}{\mathcal{P}}
\newcommand{\scQ}{\mathscr{Q}}
\newcommand{\scR}{\mathcal{R}}
\newcommand{\scV}{\mathcal{V}}
\newcommand{\rk}{\mathrm{rank}}
\newcommand{\Vol}{\mathrm{Vol}}

\def\R{\mathbb{R}}
\def\Q{\mathbb{Q}}
\def\Z{\mathbb{Z}}
\def\N{\mathbb{N}}

\def\P{\mathbb{P}}

\def\C{\mathbb{C}}

\def\cF{\mathcal{F}}
\def\cE{\mathcal{E}}
\def\cO{\mathcal{O}}

\def\cG{\mathcal{G}}
\def\cH{\mathcal{H}}

\def\cV{\mathcal{V}}

\def\cQ{\mathcal{Q}}
\def\cC{\mathcal{C}}

\def\Gr{\mathrm{Gr}}

\def\delb{\overline{\partial}}

\def\rank{\mathrm{rank}}
\def\trace{\mathrm{trace}}
\def\Lie{\mathrm{Lie}}
\def\k{\mathfrak{k}}

\def\g{\mathfrak{g}}
\def\aut{\mathfrak{aut}}

\def\Om{\Omega}
\def\om{\omega}
\def\ep{\varepsilon}
\def\>{\rangle}

\def\<{\langle}
\def\>{\rangle}

\title[Hermitian Yang--Mills connections on pullback bundles]{Hermitian Yang--Mills connections on pullback bundles}

\author[Lars Martin Sektnan and Carl Tipler]{Lars Martin Sektnan and Carl Tipler}
\address{Lars Martin Sektnan, Institut for Matematik, Aarhus University, 8000, Aarhus C, Denmark}
\address{Department of Mathematical Sciences, University of Gothenburg, 412 96 Gothenburg, Sweden}
\email{sektnan@chalmers.se}

\address{Carl Tipler, Univ Brest, UMR CNRS 6205, Laboratoire de Math\'ematiques de Bretagne Atlantique, France}
\email{Carl.Tipler@univ-brest.fr}

\begin{document}

\begin{abstract}
We investigate hermitian Yang--Mills connections on pullback bundles with respect to adiabatic classes on the total space of holomorphic submersions with connected fibres. Under some technical assumptions on the graded object of a Jordan--H\"older filtration, we obtain a necessary and sufficient criterion for when the pullback of a strictly semistable vector bundle will carry an hermitian Yang-Mills connection, in terms of intersection numbers on the base of the submersion. Together with the classical Donaldson--Uhlenbeck--Yau correspondence, we deduce that the pullback of a stable (resp. unstable) bundle remains stable (resp. unstable) for adiabatic classes, and settle the semi-stable case. 
\end{abstract}

\maketitle

\section{Introduction}
\label{sec:intro}
Established in the 80's, the Hitchin--Kobayashi correspondence builds a bridge between gauge theory and moduli problems for vector bundles \cite{NaraSesha65,skobayashi82,lubke83,donaldson87,uhlenbeckyau86}. The content of the Donaldson--Uhlenbeck--Yau theorem is that the hermitian Yang--Mills equations, that originated in physics, can be solved precisely on polystable vector bundles in the sense of Mumford and Takemoto \cite{Mum62,Tak72}. Together with their uniqueness property, hermitian Yang--Mills connections are canonically attached to holomorphic vector bundles and play a key role in the study of their moduli over K\"ahler manifolds.

It is then natural to try to understand how hermitian Yang--Mills connections, or polystable bundles, relate to natural maps between complex polarised manifolds such as immersions or submersions. A celebrated result of Metha and Ramanathan implies that the restriction of a slope stable bundle $E$ on a projective variety $X$ to a general complete intersection $Z\subset X$ of sufficiently high degree is again slope stable \cite{MeRa}. On the other hand, to the knowledge of the authors, it seems that there is no similar general result for pullbacks of stable bundles along submersions. In this paper, we will give a novel construction of hermitian Yang--Mills connections, giving an answer to the question of stability of pullback bundles on submersions with connected fibres for so-called adiabatic classes in certain circumstances.

More precisely, let $\pi: (X,H) \to (B,L)$ be a holomorphic submersion with connected
fibres between polarised connected compact complex manifolds such that $L$ is an ample line bundle on $B$, and $H$ is a relatively ample line bundle on $X$.  
In particular, all the fibres $X_b = \pi^{-1} (b)$ of $\pi$ are smooth. The bundle  $L_k = H \otimes \pi^* L^k  $ is then ample on $X$ for $k \gg 0$. In this article, we study the existence of hermitian Yang--Mills connections (HYM for short) on bundles $\pi^* E$ pulled back from $B$ with respect to the classes $L_k$, for $k\gg 0$. More generally, the work applies to not necessarily projective K\"ahler manifolds, replacing $c_1(L)$ by a K\"ahler class on $B$ and $c_1(H)$ by a relatively K\"ahler class on $X$. However, we will use the line bundle notation throughout.

We first show that if one can solve the HYM equation for a simple bundle on $B$, one can also solve it for the pulled back bundle on $X$. We also relate the solutions on $X$ to the pullback of the solution on $B$.
\begin{theorem}
\label{thm:stablecase}
Let $\pi: (X,H) \to (B,L)$ be a holomorphic submersion with connected
fibres. Suppose that $E \to B$ is a simple  holomorphic vector bundle admitting an hermitian Yang--Mills connection $A$ with respect to $\omega_B \in c_1(L).$ Then for any $\omega_X \in c_1 (H)$ there are connections $A_k$ on $\pi^* E$ which are hermitian Yang--Mills with respect to $\omega_X + k \pi^* \omega_B$ for all $k \gg 0$. Moreover, the sequence $(A_k)$ converges to $\pi^*A$ in any Sobolev norm.
\end{theorem}
This result is quite natural, as for $k$ large enough, the global geometry of $(X,L_k)$ is governed by that of $(B,L^k)$. From the Hitchin--Kobayashi correspondence it follows that pullback preserves stability for adiabatic classes.
\begin{corollary}
\label{cor:stablepullback}
Suppose $E$ is a slope stable vector bundle on $(B,L)$. Then for all $k \gg 0$, the pullback bundle $\pi^*E \to X$ is slope stable with respect to $L_k$.
\end{corollary}
On the other hand, the leading order of the slope of a pullback bundle on $(X, L_k)$ is its slope on $(B,L)$. The following is then straightforward. 
\begin{proposition}
 \label{prop:unstablepullbackintro}
 Suppose $E$ is a strictly slope unstable vector bundle on $(B,L)$. Then for all $k \gg 0$, the pullback bundle $\pi^*E \to X$ is strictly slope unstable with respect to $L_k$.
\end{proposition}

The main result of this paper deals with the more subtle situation of pulling back a slope semistable vector bundle. Such a bundle $E \to B$ admits a Jordan--H\"older filtration by subsheaves
$0 = \cF_0 \subset \cF_1 \subset \hdots \subset \cF_\ell = E$
with corresponding stable quotients $\cG_i = \frac{\cF_i}{\cF_{i-1}}$ for $i=1, \hdots, \ell$ of slope $\mu_L(\cG_i)=\mu_L(E)$. In particular, the graded object of this filtration $\Gr(E):=\bigoplus_{i=1}^\ell \cG_i$ is polystable. 
For technical reasons, we will assume $\Gr(E)$ to be locally free, and identify it with the associated vector bundle. From Theorem \ref{thm:stablecase}, it follows that its pullback to $(X,L_k)$, for $k$ large, is a direct sum of stable bundles, possibly of different slopes. The next theorem shows that, under some technical assumptions, the various stable components of $\Gr(E)$ and their {\it adiabatic slopes} govern the stability of $\pi^*E$. 

To give a precise statement, we need to introduce some notations. First, for any pair of torsion free sheaves $\cF$ and $\cE$ on $B$, letting $\mu_k(\pi^*\cE)$ be the slope of $\pi^*\cE$ with respect to $L_k$, we will use the notation $\mu_\infty(\cF)< \mu_\infty(\cE)$ (resp. $\mu_\infty(\cE)=\mu_\infty(\cF)$) when the leading order term in the $k$-expansion of $\mu_k(\pi^*\cE)-\mu_k(\pi^*\cF)$ is positive (resp. when these expansions are equal); a non-strict inequality meaning that one of these two conditions occurs. Then, we will introduce the following set of hypothesis:
\begin{enumerate}
 \item[$(H1)$] The Jordan-H\"older filtration $(\cF_i)_{0\leq i \leq \ell}$ of $E$ is unique;
 \item[$(H2)$] the stable components $\cG_i:=\cF_i/\cF_{i-1}$ of $\Gr(E)$ are pairwise non isomorphic;
 \item[$(H3)$] for all $i$, $\cG_i$ is locally free.
\end{enumerate}

We then have :
\begin{theorem}
\label{thm:ssthmintro}
Let $\pi: (X,H) \to (B,L)$ be a holomorphic submersion with connected
fibres.
Suppose $E$ is a slope semistable vector bundle on $(B,L)$ satisfying $(H1), (H2)$ and $(H3)$.
Assume that for all $i\in [\![ 1, \ell-1 ]\!]$, $\mu_\infty (\cF_i) < \mu_\infty(E)$.
Then for any $\omega_B \in c_1(L)$ and $\omega_X \in c_1 (H)$ there are connections $A_k$ on $\pi^* E$ which are hermitian Yang--Mills with respect to $\omega_X + k \pi^* \omega_B$ for all $k \gg 0$. Moreover, there is an hermitian Yang--Mills connection $A$ on $\Gr(E)$ with respect to $\om_B$ such that $(A_k)$ converges to $\pi^*A$ in any Sobolev norm.
\end{theorem}
Note that this result gives more information than the stability of $\pi^*E$ for adiabatic polarisation, as it provides information on the convergence of the associated hermitian Yang--Mills connections.
\begin{remark}
The assumptions $(H1)$ and $(H2)$ imply that $E$ is simple, which is a necessary condition for the conclusion of Theorem \ref{thm:ssthmintro} to hold true (see Lemma \ref{lem:H1H2implysimple}). On the other hand, assuming $E$ to be simple, the hypothesis $(H1)-(H3)$ seem purely technical. Hypothesis $(H3)$ enables us to see the initial bundle as a small deformation of its graded object (see e.g. \cite{BuSchu}), allowing us to use the differential geometric approach to deformation theory \cite[Section 7]{skobayashi87}. This assumption was already made in the related work \cite{leung97}. It seems likely that one could relax this hypothesis, asking only for a reflexive graded object, by using resolutions of singularities and {\it admissible} Hermite--Einstein metrics as introduced by Bando and Siu (see \cite{sibley15} and the references therein).
 Hypothesis $(H2)$ ensures that the kernel of the linearisation of the HYM equation on $\Gr(E)$ is precisely the direct sum of the kernels of each stable piece $\cG_i$. Finally, together with $(H2)$, $(H1)$ imposes a certain shape on the deformation from $\Gr(E)$ to $E$, see Lemma \ref{lem:gammaijnonzero}. This will simplify the control on the various deformation rates that will appear in the argument. It would be interesting to remove these three hypothesis.
 \end{remark}

Still relying on the Hitchin--Kobayashi correspondence, by gathering Proposition \ref{prop:semistable unstable case}, Theorem \ref{thm:semistablemain} and Corollaries \ref{cor:semistablecor} and \ref{cor:JHfiltrationpullback}, we obtain the following.
\begin{corollary}
\label{cor:sscorintro}
Suppose $E$ is a slope semistable vector bundle on $(B,L)$ satisfying $(H1)-(H3)$. Then, on $(X,L_k)$, when $k\gg 0$, $\pi^*E$ is
\begin{itemize}
 \item unstable if and only if there exists $i\in [\![ 1, \ell-1 ]\!]$ with $\mu_\infty(\cF_i)>\mu_\infty(E)$,
  \item semistable if and only if for all $i\in [\![ 1, \ell-1 ]\!]$, $\mu_\infty(\cF_i)\leq \mu_\infty(E)$,
  \item stable if and only if for all $i\in [\![ 1, \ell-1 ]\!]$, $\mu_\infty(\cF_i)< \mu_\infty(E)$.
\end{itemize}
Moreover, in the semistable case, for $k\gg 0$, a Jordan--H\"older filtration of $\pi^*E$ is given by
$$
0\subset \pi^*\cF_{i_1}\subset \ldots \subset \pi^*\cF_{i_\ell}=\pi^*E
$$
where the $i_j$'s are precisely the indices with $\mu_\infty(\cF_{i_j})=\mu_\infty(E)$.
\end{corollary}
\begin{remark}
Note that for a given vector bundle $E$ on $B$, the number of non-zero terms in the expansion of the slope $\mu_k(\pi^*E)$ is bounded by $\mathrm{dim}_\C(B)$, and that these terms can be computed from intersection numbers on $B$ depending on the geometry of the map $\pi: (X,H) \to (B,L)$. For example, up to multiplicative constants, the leading order term of $\mu_k(\pi^*E)$ is the slope $\mu_L(E)$ of $E$ on $(B,L)$ while the second order term is given by the Hodge--Riemann pairing of $\frac{c_1(\cE)}{\rank{E}} $ with the class $\Theta:=\pi_*(c_1(H)^{m+1})$ (see Lemma \ref{lem:slopeexpansion}). This shows that our criterion is quite explicit, and can be used in various situations to produce stable vector bundles. In Section \ref{sec:examples}, we discuss some examples, in particular when $B$ is a complex curve or a surface, or when the space $X$ is the projectivisation of a vector bundle on $B$.
\end{remark}
In order to produce the connections in Theorems \ref{thm:stablecase} and \ref{thm:ssthmintro}, we use a perturbative argument. This was inspired by, and built on techniques from, analogous questions for constructing extremal (or constant scalar curvature) K\"ahler metrics in adiabatic classes on the total spaces of holomorphic submersions with connected fibres, see among others \cite{hong98,fine04,bronnle15,luseyyedali14,dervansektnan20}.

This perturbative technique fits into a vast array of problems of such a nature in geometric analysis (e.g. gluings, deformations, smoothings, adiabatic constructions), where the common feature is to start with a solution of a given geometric PDE on a given geometric object that we perturb, trying then to solve the same PDE on the perturbed object. A specific feature for HYM connections and cscK metrics is that the solutions of the geometric PDE correspond to zeros of a moment map \cite{donaldson87,Fuj92,donaldson97}, and are related to a Geometric Invariant Theory (GIT) stability notion. 

 In all of the previously cited perturbation results, one starts with a (poly)-stable object. An algebro-geometric asymptotic expansion argument shows that one has to start with at least a semistable object to hope to produce a stable object in the perturbed situation, see for example \cite[Theorem 13]{stoppa10} for such a result for the problem of producing cscK metrics on blowups. This is also the case in our setting, see Proposition \ref{prop:unstablepullbackintro}. The main novel achievement of the present work is to deal with a strictly semistable situation, giving a criterion for when a strictly semistable bundle pulls back to a stable one for adiabatic classes. The arguments get significantly more complicated in this setting. The introduction to Section \ref{sec:semistable} contains a discussion of the main new hurdles to overcome, compared to the stable case.
 
 \begin{remark}
  We should point out that the method we use in the semistable situation is similar to the one developed in \cite{leung97}. While in Leung's work, the perturbation is on the equations, in the present work the perturbation is on the geometry of the manifold, leading to extra complications in the arguments, see Remark \ref{rem:differencewithLeung}.
 \end{remark}

Finally, we focus on a very special case where our results apply, namely when the fibres are reduced to a point, that is when $X=B$. In that setting, the perturbation comes from the change of polarisation on $B$ from $L$ to $L^k\otimes H$. Equivalently, we consider how semistable bundles on $(B,L)$ behave with respect to nearby $\Q$-polarisations $L + \frac{1}{k} H$. Note that while in GIT, (semi)stability is an open condition, understanding the variations of GIT quotients under modifications of the polarisation is a much more difficult problem (see \cite{DolHu,Thad}). We refer to \cite[Chapter 4, Section C]{HuLe} and the references therein for results on variations of moduli spaces of stable bundles on surfaces related to wall-crossing phenomena of polarisations. At a much more humble level, when restricting to a single semistable bundle $E$, our result provides an effective criterion on variations of the polarisation that send $E$ to the stable or unstable loci (see Section \ref{sec:XisBcase}). It seems reasonable to expect that our results extend locally around a semistable bundle. However, obtaining more global results, such as descriptions of variations of the moduli space of stable bundles, would require substantial work, and in particular uniform control on our estimates.

\subsection*{Outline:} In Section \ref{sec:setup}, we gather some of the results we need on the hermitian Yang--Mills equation, slope stability, and the structure one obtains on sections and tensors from a holomorphic submersion with connected fibres. In Section \ref{sec:adiabaticslopes}, we introduce the notion of adiabatic slopes and discuss when the pullback of a strictly unstable bundle is strictly unstable for adiabatic polarisations. Section \ref{sec:stable} proves the main result in the stable case via constructing approximate solutions to the HYM equation in Section \ref{sec:stableapprox}, before showing in Section \ref{sec:nonlinear} how this can be perturbed to a genuine solution. Section \ref{sec:semistable} deals with the more technical semistable case. Finally, in Section \ref{sec:examples} we investigate the consequences of the main results in some examples.

\subsection*{Acknowledgments:}  We thank Ruadha\'i Dervan for helpful discussions and for pointing out the reference \cite{leung97}. We also thank the anonymous referee for their advices that improved the exposition and the content of this paper. The authors benefited from visits to LMBA and Aarhus University; they would like to thank these institutions for providing stimulating work environments, as well as LABSTIC for providing financial support during LMS's visit to LMBA. LMS's postdoctoral position at Aarhus University was supported by Villum Fonden, grant 0019098. The article was updated after LMS joined the University of Gothenburg, funded by a Marie Sk\l{}odowska-Curie Individual Fellowship funded from the European Union's Horizon 2020 research and innovation programme under grant agreement No 101028041. CT is partially supported by the grants MARGE ANR-21-CE40-0011 and BRIDGES ANR--FAPESP ANR-21-CE40-0017. 

%%%%
%Section 2
%%%%

\section{The hermitian Yang--Mills equation, and holomorphic submersions}
\label{sec:setup}
In Sections \ref{sec:HYM} and \ref{sec:slopestability} we introduce the notions of HYM connections and slope stability, together with some general results. From Section \ref{sec:decompositions}, we will start to specialise the discussion to holomorphic submersions with connected fibres. In all the paper, the notation $\Gamma$ will be used to denote smooth sections.

\subsection{The hermitian Yang--Mills equation}
\label{sec:HYM}

Let $E \to X$ be a holomorphic vector bundle of rank $r$ over a compact K\"ahler manifold $X$. An hermitian metric on $E$ is \textit{Hermite--Einstein} with respect to a K\"ahler metric with K\"ahler form $\omega$ if the curvature $F_h \in \Omega^2 \left(X, \End E \right)$ of the corresponding Chern connection satisfies
\begin{align}\label{HEeqn} \Lambda_{\om} \left( iF_h \right) = c \Id_E
\end{align}
for a (uniquely determined, real) constant $c$. Here $\Lambda_{\omega}$ is the contraction operator of $\omega$, determined by
$$
 \Lambda_{\om} \left( \alpha \right) \om^m = m \alpha \wedge \om^{m-1},
$$
where $m$ is the dimension of $X$.

A different point of view is to vary the connection, rather than the metric. A connection $A$ on $E$ is said to be \textit{hermitian Yang--Mills} if
\begin{align*} F_A^{0,2} &= 0, \\
\Lambda_{\om} \left( i F_A \right) &= c \Id_E.
\end{align*}
The former says the $(0,1)$-part of $A$ determines a holomorphic structure on $E$.

The \textit{complex gauge group} is 
$$ 
\scGC(E) = \Gamma \left( \GL \left(E, \mathbb{C} \right) \right).
$$ 
Note that if $\dbar$ is the Dolbeault operator defining the holomorphic structure on $E$, then for each $f \in \scGC(E)$, $ f^{-1} \circ \dbar \circ f$ defines a biholomorphic complex structure on $E$. Let $d_A = \partial_A + \dbar_A$ be the Chern connection of some hermitian metric $h$ on $E$ with respect to the original complex structure (so $\dbar_A = \dbar$). Then the Chern connection $A^f$ of $h$ with respect to the new complex structure induced by $ f^{-1} \circ \dbar \circ f$ is given by 
$$
d_{A^f} = f^* \circ \partial_A \circ (f^*)^{-1} + f^{-1} \circ \dbar \circ f.
$$
In this article we will take the point of view of fixing the hermitian metric and moving the complex structure through the complex gauge group.

The action of any $f \in \scGC(E)$ such that $f^* = f^{-1}$ preserves solutions of the HYM equation. It will therefore be important to gauge fix the equation. Thus we will not move the complex structure through the full complex gauge group $\scGC(E)$, but rather through \begin{align} \scH(E,h) = \scGC(E) \cap \Gamma \left( \End_H (E, h) \right).
\end{align}
Here $\End_H (E,h)$ denotes the hermitian endomorphisms of $(E,h)$, where $h$ is some fixed hermitian metric on $E$. This, for example, ensures that the linearisation of the Equation will be the Laplacian -- a Fredholm operator. See Lemma \ref{lem:linop} below.

For a connection $A$ on $E$, define the Laplace operator $\Delta_{A}$ by
\begin{align}\label{laplaceop}  \Delta_A =  i \Lambda_{\omega} \left( \partial_{A} \bar \partial_{A}  - \bar \partial_{A}  \partial_{A} \right).
\end{align}
Since $\omega$ is K\"ahler, it follows from the Nakano identities (see e.g. \cite[Lemma 5.2.3]{huybrechts05}) that $\Delta_{A} = - d_A^* d_A$ when $A$ is the Chern connection.

Letting $\scQ$ denote the tangent space to $\scH(E,h)$ at the identity, we have that solving the hermitian Yang--Mills equation is equivalent to solving $ \Phi (f) = c \Id_E$, where $\Phi : \scH(E,h) \to \scQ$ is given by $$\Phi (f) = i \Lambda_{\omega} \left( F_{A^f} \right).$$ Equivalently, we want to solve $\Psi (s) = c \Id_E$, where $\Psi : \scQ \to \scQ$ is $\Psi = \Phi \circ \exp.$ It will be more convenient to work with this map, as $\scQ$ is a vector space.

The importance of the Laplace operator for us comes from its relation to the linearisation of $\Psi$. Let $A_{\End E}$ denote the connection induced by $A$ on $\End E$.
\begin{lemma}\label{lem:linop} The differential of $\Psi$  at the identity is $$ d \Psi_{\Id_E} = \Delta_{A_{\End E}}.$$
\end{lemma}

It is crucial for the linear theory to understand the (co-)kernel of the linearised operator to the equation. For simple bundles, the kernel of the Laplace operator is well understood
\begin{lemma} If $E$ is a simple bundle over a compact base, then $$\ker \Delta_{A_{\End E}} = \mathbb{C} \cdot \Id_E$$ on $\Gamma \left( X, \End E \right).$
\end{lemma}
\begin{remark} In the context of Lemma \ref{lem:linop}, we are restricting $\Delta_{A_{\End E}}$ to sections of  $ \End_H ( E, h) $. Only real multiples of the identity are hermitian endomorphisms, thus the kernel is then $ \mathbb{R} \cdot \Id_E.$
\end{remark}

\subsection{Slope stability}
\label{sec:slopestability}
The notion of slope stability originated in Mumford's study of the moduli spaces of vector bundles on Riemann surfaces \cite{Mum62} and has been adapted to higher dimensional varieties \cite{Tak72}. This stability notion can be stated for more general coherent sheaves (see \cite{HuLe} and the references therein) and in this paper we will restrict ourselves to the torsion-free ones. Let $(X,L)$ be a compact polarised K\"ahler manifold of dimension $n$.
\begin{definition}
 \label{def:slope}
Let $\cE$ be a torsion-free coherent sheaf on $X$. The slope $\mu_L(\cE)\in\Q$ of $\cE$ with respect to $L$ is given by the intersection formula
 \begin{equation}
  \label{eq:slope}
  \mu_L(\cE)=\frac{\deg_L(\cE)}{\rank(\cE)}.
 \end{equation}
 In this formula, $\rank(\cE)$ denotes the rank of $\cE$ while $\deg_L(\cE)=c_1(\cE)\cdot L^{n-1}$ stands for its degree, and
 the first Chern class of $\cE$ is taken to be the first Chern class of the reflexive hull of its determinant $((\Lambda^{\rank(\cE)}\cE)^\vee)^\vee$.
\end{definition}
\begin{definition}
\label{def:stability}
A torsion-free coherent sheaf $\mathcal{E} $ is said to be \emph{slope semistable} with respect to $L$ if for any coherent subsheaf $\mathcal{F}$ of $\mathcal{E}$ with $0<\rank( \cF)<\rank(\cE)$, one has
\begin{eqnarray*}
\mu_L(\mathcal{F})\leq \mu_L(\mathcal{E}).
\end{eqnarray*}
When strict inequality always holds, we say that $\mathcal{E}$ is \emph{slope stable}. Finally, $\cE$ is said to be \emph{slope polystable} if it is the direct sum of slope stable subsheaves of the same slope. If $\cE$ is slope semistable, but not slope stable, we will say that $\cE$ is \textit{strictly slope semistable}. We say $\mathcal{E}$ is \emph{strictly slope unstable} if $\mathcal{E} $ is not slope semistable.
\end{definition}
As slope stability will be the only stability notion for sheaves discussed in this paper, we will often omit ``slope'', and simply refer to stability of a sheaf. When the situation is clear enough, we will also omit to refer to the polarisation in the stability notions. We will make the standard identification of a holomorphic vector bundle with its sheaf of sections, and thus talk about slope stability notions for vector bundles as well. In that case, not only is slope stability a tool for building moduli, but it also relates nicely to differential geometry via the Hitchin--Kobayashi correspondence.
\begin{theorem}
\label{thm:HKcorrespondence}
Let $E \to X$ be a holomorphic vector bundle, and let $\omega \in c_1(L)$ be a K\"ahler form. Then there exists an Hermite--Einstein metric on $E$ with respect to $\omega$ if and only if $E$ is polystable with respect to $L$.
\end{theorem}
It was shown in \cite[Theorem 2.4]{skobayashi82} and \cite{lubke83} that if an Hermite--Einstein metric exists, then the bundle is polystable. The converse was proved in \cite[Theorem 4.1]{uhlenbeckyau86} and \cite[Proposition 1]{donaldson87}.

Being a Geometric Invariant Theory notion, slope stability enjoys many nice features (see e.g. \cite[Sections 1 and 4]{HuLe}). In particular, any semistable vector bundle $E$ on $X$ admits a degeneration to a unique polystable object:
\begin{definition}
 \label{def:JHfiltration}
 Let $\cE$ be a coherent torsion-free sheaf on $(X,L)$. A \emph{Jordan--H\"older filtration} for $\cE$ is a filtration by saturated coherent subsheaves:
 \begin{align}
\label{eq:JHfiltration}
0 = \cF_0 \subset \cF_1 \subset \hdots \subset \cF_l = \cE,
\end{align}
such that the corresponding quotients,
\begin{align}
\label{eq:JHquots}
\cG_i = \frac{\cF_i}{\cF_{i-1}},
\end{align}
 for $i=1, \hdots, l$, are stable with slope $\mu_L(\cG_i)=\mu_L(\cE)$.
 In particular, the graded object of this filtration
 \begin{equation}
  \label{eq:JHgraded}
  \Gr(\cE):=\bigoplus_{i=1}^l \cG_i
 \end{equation}
is polystable.
 \end{definition}
 From \cite[Section 1]{HuLe}, we have the standard existence and uniqueness result:
\begin{proposition}
 \label{prop:JHfiltration}
 Any semistable coherent torsion-free sheaf $\cE$ on $(X,L)$ admits a Jordan--H\"older filtration. Such a filtration may not be unique, but the reflexive hull $\Gr(\cE)^{**} $ of the graded object $\Gr(\cE)$ of a Jordan--H\"older filtration is unique up to isomorphism.
\end{proposition}
The graded object $\Gr(E)$ of a Jordan--H\"older filtration will be used in Section \ref{sec:semistable}, when studying the lift of a semistable vector bundle $E$ to a holomorphic submersion with connected fibres. In general, the sheaf $\Gr(E)$ is not the sheaf of sections of a vector bundle. For technical reasons, we will assume $\Gr(E)$ to be locally free, and denote by $\Gr(E)$ as well the associated vector bundle. In that context, from \cite{BuSchu}, $E$ is a complex deformation of $\Gr(E)$.

\subsection{Decomposition of tensors and sections on holomorphic submersions with connected fibres}
\label{sec:decompositions}
We now specialise the discussion to the case of a holomorphic submersion $\pi:(X, H) \to (B,L)$ of polarised K\"ahler manifolds, with a vector bundle $E \to B$ inducing the pullback vector bundle $\pi^* E \to X$. For a given $b\in B$, we will denote by $X_b$ the fibre of $\pi$ over $b$. We set $n=\dim_\C(B)$ and $n+m=\dim_\C(X)$.

Given $\omega_X \in c_1 (H)$ a $(1,1)$-form inducing a fibrewise K\"ahler metric, one obtains a corresponding decomposition of tensors and functions on $X$. Indeed, let $\scV = \ker \pi_* \subset TX$ and let $\cH \cong \pi^* TB$ be the bundle whose fibre at $x\in X$ is the $\omega_{X}$-orthogonal complement to $\scV_x$. This gives a \textit{smooth} splitting of the exact sequence of holomorphic vector bundles $$ 0 \to \scV \to TX \to \pi^* TB \to 0.$$ For functions, one splits $C^{\infty} (X) = C^{\infty}_0 (X) \oplus \pi^* C^{\infty} (B)$, where $C_0^{\infty}(X)$ consists of fibrewise average $0$ functions.

Similarly, one decomposes sections of $\pi^* E$. This is just a higher-dimensional version of the above for functions. Locally, in a trivialisation of $E \to B$, a section of $\pi^* E$ is a map $s: \pi^{-1}(U) \to \mathbb{C}^r,$ where $U \subseteq B$. Denote by $F = X_b$ the fibre of $\pi$ at $b$, let $\om_F=(\om_X)_{\vert F}$ and let $\Vol(F) = \int_F \om_F^n$ be the volume of the fibre $F$ (which is independent of $b$). If $s^i$ denotes the $i^{\textnormal{th}}$ coordinate of $s$, we define the base-like component $s_B$ at $b$ by
$$
s_B^i  = \frac{\int_F s^i \omega_F^m}{\Vol(F)}.
$$ 
This gives a well-defined section $s_B$ of $E$. For in another trivialisation, $s$ is represented by $\hat s : \pi^{-1}(U)  \to \mathbb{C}^r$, say, where $$\hat s^i = \sum_j \psi^i_j s^j,$$ for some functions $\psi_j^i : U \to \mathbb{C}$. Since $\psi_j^i$ does not depend on the point in $F$, $$\int_F \hat s^i \omega_F^m = \sum_j \psi^i_j  \int_F  s^j \omega_F^m = \sum_j \psi^i_j s_B^j, $$ showing $s_B$ is well-defined. We let $s_F = s - \pi^* s_B$ and then this gives a smooth splitting which we denote $$\Gamma \left( X, \pi^*E \right) = \Gamma_0 \left( X, \pi^* E \right) \oplus \pi^* \Gamma \left( B, E \right) .$$
We will refer to the elements in $\Gamma_0 \left( X, \pi^* E \right)$ as {\it vertical} and to the elements in $\pi^* \Gamma \left( B, E \right)$ as {\it basic}.

Since the Laplacian is a crucial operator for us, one particular instance where we will use the decomposition $TX = \scV \oplus \cH$ is for the contraction $\Lambda_k = \Lambda_{\omega_k}$ on $2$-forms, where $\omega_k = \om_X + k \pi^* \omega_B$, $k\in \N^*$. The above decomposition implies that we have a similar splitting
\begin{align}
\label{eqn:2formsplitting}
\Lambda^2TX^*=\Lambda^2\cH^*\oplus \Lambda^1\cV^*\wedge\Lambda^1\cH^* \oplus\Lambda^2\cV^*.
\end{align} For a given $2$-form $\alpha$ on $X$, we denote by $\alpha_\cV$ (resp. $\alpha_\cH$) its purely vertical (resp. horizontal) component in this decomposition.

Being non-degenerate on $\scV$, note that $\omega_{\scV}$ defines a \textit{vertical contraction} operator $\Lambda_{\scV} : \Omega^{1,1} (X, \pi^*G ) \to \Gamma (X,\pi^* G)$ for the pullback of any bundle $G$ from $B$. This is only non-zero on the vertical component in the splitting \eqref{eqn:2formsplitting}. Similarly, $\pi^* \omega_B$ defines a metric on the horizontal part, which induces a horizontal contraction operator $$ \Lambda_{\cH} : \Omega^{1,1} (X, \pi^*G ) \to \Gamma (X,\pi^* G).$$ Note that if $\alpha = \pi^* \beta $ is the pullback of $ \beta  \in \Omega^{1,1} ( B,G ) $, then $\Lambda_{\cH} \alpha = \pi^* \Lambda_{\omega_B} (\beta).$ However, even though $\Lambda_{\cH}$ is non-zero only on the horizontal component of $\Omega^{1,1}(X)$, its image is not contained in the horizontal subspace $\pi^* \Gamma ( B, G)$ of $\Gamma \left( X, \pi^* G \right)$. Finally, note also that as $\omega_B$ gives an inner product on $\cH$, it induces an inner product $\langle \cdot , \cdot \rangle_{\omega_B}$ on $\Lambda^2 \cH^*.$ We then have the following lemma:
\begin{lemma}
 \label{lem:expansion contraction}
 The contraction operator $\Lambda_k = \Lambda_{\omega_k}$ admits the following expansion, for $\alpha\in\Om^{1,1}(X)$:
 \begin{equation}
  \label{eq:expansion contraction}
  \Lambda_k \alpha = \Lambda_\cV \alpha_\cV + k^{-1} \Lambda_\cH \alpha_\cH - k^{-2} \langle \alpha_\cH, \om_\cH \rangle_{\omega_B} + \cO(k^{-3}),
 \end{equation}
where $\omega_{\cV} = \left( \omega_X \right)_{\cV}$ and $\omega_{\cH} = \left( \omega_X \right)_{\cH}$ denotes the vertical and horizontal components of $\omega_X$, respectively. 
\end{lemma}
\begin{proof}
Note that $\omega_X$ has no mixed terms with respect to the decomposition induced by itself. The lemma now follows from a direct computation using the orthogonal decomposition of $TX$ induced by $\om_X$, noting that
$$ \Lambda_{k} \left( \alpha \right) \omega_k^{m+n} = \binom{m+n}{n} \Lambda_{\omega_k} \left( \alpha \right)  \omega_{\scV}^{m} \wedge \left( k \omega_B + \omega_{\cH} \right)^n $$
and
\begin{align*} \alpha \wedge \omega_k^{m+n-1} =& \binom{m+n-1}{n} \alpha_{\cV} \wedge \omega_{\scV}^{m-1} \wedge \left( k \omega_B + \omega_{\cH} \right)^n  \\
&+ \binom{m+n-1}{n-1}   \omega_{\scV}^{m} \wedge \alpha_{\cH} \wedge \left( k \omega_B + \omega_{\cH} \right)^{n-1} \\
=& \frac{1}{m} \binom{m+n-1}{n} \Lambda_{\cV} \left( \alpha_{\cV} \right) \omega_{\scV}^{m} \wedge \left( k \omega_B + \omega_{\cH} \right)^n  \\
&+ k^{-1} \frac{1}{n} \binom{m+n-1}{n-1}  \Lambda_{\omega_B + k^{-1} \omega_{\cH} } \left( \alpha_{\cH} \right) \omega_{\scV}^{m} \wedge \left( k \omega_B + \omega_{\cH} \right)^n .
\end{align*}
To obtain the $k^{-2}$ term, one uses the identity
\begin{equation}
 \label{eq:product on basic forms}
\langle \alpha , \beta \rangle_{\om_B} \om_B^n = (\Lambda_{\om_B}\alpha\: \Lambda_{\om_B} \beta)\: \om_B^n -n (n-1) \alpha\wedge\beta\wedge \om_B^{n-2},
\end{equation}
see e.g. \cite[Lemma 4.7]{szekelyhidi14book}.
\end{proof}

\begin{remark}
 \label{rem:mixed terms trace vanish}
Note that this expansion does not see mixed terms. Moreover, there is a vertical term to leading order only, and the horizontal terms only appear to order $k^{-1}$ and below.
\end{remark}

We end this section with an important Lemma for the perturbation problem we are considering. To not obtain new obstructions arising in the linear theory when trying to solve the HYM equation on the total space $X$, it is crucial that $\pi^* E$ remains simple if $E$ is. This is established below.
\begin{lemma}
\label{lem:pullbackissimple}
Let $\pi : X \to B$ be a holomorphic submersion with connected fibres and $E \to B$ a simple bundle. Then $\pi^* E \to X$ is simple.
\end{lemma}
\begin{proof} Suppose $s \in H^0(X, \End \pi^* E).$ Then, for any $b\in B$, $s_b = s_{\vert X_b}$ is an element of $H^0 \left( X_b, \End \pi^* E_b \right).$ But $\End \pi^* E_b$ is trivial bundle, so $s_b$ is a holomorphic map from a compact manifold to a vector space -- hence is constant. Thus in the decomposition $s = s_B + s_F$, the vertical component $s_F$ vanishes, and so $s$ is pulled back from $B$. Since $B$ is simple, $s = s_B$ then has to be a multiple of the identity.
\end{proof}

\begin{remark}
\label{rem:pullbacknonisom}
 A direct adaptation of Lemma \ref{lem:pullbackissimple} shows that if $\cG_1$ and $\cG_2$ are two holomorphic vector bundles on $B$ with $H^0(B,\cG_1^*\otimes\cG_2)=0$, then $H^0(X,\pi^*\cG_1^*\otimes\pi^*\cG_2)=0$. This implies in particular that $\pi^*\cG_1$ is not isomorphic to $\pi^*\cG_2$ if $\cG_1$ and $\cG_2$ are not isomorphic and stable of the same slope, a fact that will be used later on.
\end{remark}

%%%%
%Section 3
%%%%

\section{Slope stability and adiabatic classes}
\label{sec:adiabaticslopes}
In this section, we calculate slope formulae with respect to adiabatic classes on holomorphic submersions with connected fibres. In particular, we obtain criteria implying unstability or strict semistability for the pullback of a non stable bundle for adiabatic classes. As before, $\pi: (X,H) \to (B,L)$ is a holomorphic submersion of polarised K\"ahler manifolds, with connected fibres.
\subsection{Adiabatic slopes}
\label{sec:adiabatic slope stability}
For a given torsion-free coherent sheaf $\cE$ on $X$, we denote by $\mu_k(\cE)$ the slope of $\cE$ with respect to $L_k:=H+k\pi^* L$. That is
$$
\mu_k(\cE)= \frac{c_1(\cE)\cdot c_1(L_k)^{n+m-1}}{\mathrm{rank}(\cE)},
$$
where we recall that $\dim_{\C} B = n$ and $\dim_{\C} X = m+n.$ Note that we are suppressing the pullback in the notation above. For dimensional reasons, $\mu_k(\cE)$ is a polynomial of degree $n$ in $k$. Define, for all $i\in \lbrace 0, \ldots, n\rbrace$, the following rational numbers $\nu_i(\cE)$ by:
\begin{align}
\label{eq:nuslope}
\mu_k(\cE)=\sum_{i= 0}^n k^{n-i}\nu_i(\cE).
\end{align}
When we consider a pullback sheaf $\pi^*\cE$ from $B$, for simplicity, we will still denote by $(\nu_i(\cE))_{i=0\ldots n}$ the rational numbers in the expansion of $\mu_k(\pi^*\cE)$. In that case, $\nu_0(\cE)=0$, and the rational numbers $(\nu_i(\cE))_{i=1, \ldots , n}$ can be computed from intersection numbers on the base. Setting
\begin{equation}
 \label{eq:first chern class fibration}
\Theta :=(m+1)\left[\int_{X_b}  \om_\cH \wedge \om_X^m\right] = \pi_* \left( c_1(H)^{m+1} \right) \in H^{1,1}(B,\C)\cap H^2(B,\R),
\end{equation}
a direct computation shows:
\begin{lemma}
 \label{lem:slopeexpansion}
 Let $\cE$ be a coherent torsion-free sheaf on $B$. Then, for $k\gg 0$, one has
 $$
 \mu_k(\pi^*\cE)= \nu_1 (\cE)  k^{n-1} + \nu_2 (\cE) k^{n-2} + O(k^{n-3})
 $$
where
\begin{align*} 
\nu_1 (\cE) &=\binom{m+n-1}{n-1} c_1(H_b)^m \cdot \mu_L (\cE) \\
\nu_2 (\cE) &= \binom{n+m-1}{n-2} \frac{\Theta\cdot c_1(\cE) \cdot c_1(L)^{n-2}}{\rk(\cE)},
\end{align*}
and $c_1(H_b)^m$ is the volume for any fibre $X_b$ of $\pi$ with respect to $H_b = H_{\vert X_b}$.
\end{lemma}
\begin{proof}
We have that
\begin{align*}\mu_k (\pi^* \cE ) =& \frac{c_1(\cE)\cdot  \left( k c_1(L) + c_1 (H) \right)^{n+m-1}}{\rk (\cE)} \\
=&  \binom{n+m-1}{n-1} \frac{c_1(\cE)  c_1(L)^{n-1} c_1 (H)^m }{\rk (\cE)} k^{n-1} \\
&+  \binom{n+m-1}{n-2}  \frac{c_1(\cE) c_1(L)^{n-2} c_1 (H)^{m+1}}{\rk (\cE) } k^{n-2} + O(k^{n-3}),
\end{align*}
where we have used that the wedge product of more than $n$ pulled back classes from $B$ vanishes. The result now follows by noting that
$$
c_1(\cE) c_1(L)^{n-2} c_1 (H)^{m+1}=\int_B c_1(\cE)\wedge\om_B^{n-2}\wedge (m+1)\int_{X_b} \om_\cH\wedge \om_X^m,
$$
 and that the volume of $X_b$ is independent of $b \in B$ for basic topological reasons.
\end{proof}
We now introduce a comparison notion for {\it adiabatic slopes}:
\begin{definition}
 \label{def:adiabaticcomparison}
  Let $\cE$ and $\cE'$ be two coherent torsion-free sheaves on $X$. We will say that the adiabatic slope of $\cE$ is greater than the adiabatic slope of $\cE'$ (with respect to $L_k$), denoted $\mu_\infty(\cE)>\mu_\infty(\cE')$, if the leading order term in the $k$-expansion of $\mu_k(\cE)-\mu_k(\cE')$ is positive. If $\mu_k(\cE)=\mu_k(\cE')$ for all $k$, which we denote $\mu_\infty(\cE)=\mu_\infty(\cE')$, we will say that the adiabatic slopes are equal. Finally, we say the adiabatic slope of $\cE$ is at least that of $\cE'$ if the adiabatic slope of $\cE$ is either greater than or equal to that of $\cE'$. We denote this $\mu_\infty(\cE)\geq \mu_\infty(\cE')$.
\end{definition}

\begin{remark}
 In Definition \ref{def:adiabaticcomparison}, when $\cE$ is pulled back from $B$, we will use the notation $\mu_\infty(\cE)$ to refer to its adiabatic slope. In this case, following the proof of Lemma \ref{lem:slopeexpansion}, $\mu_\infty(\cE)$ only depends on intersection numbers computed on $B$.
\end{remark}
We finish this section with some results that will turn useful when considering pullbacks of semistable sheaves. The proof of the following follows from the additivity of the first Chern class:
\begin{lemma}
 \label{lem:additivity of nui}
 For any subsheaf $\cF \subseteq \cE$ and for any $k$:
  $$
 \rank(\cF )\mu_k(\cF)+\rank(\cE/\cF)\mu_k( \cE / \cF)=\rank(\cE)\mu_k(\cE).
 $$
 Equivalently, for all $i$, we have
 $$
 \rank(\cF )\nu_i(\cF)+\rank(\cE/\cF)\nu_i( \cE / \cF)=\rank(\cE)\nu_i(\cE).
 $$
\end{lemma}
As a corollary, one obtains
\begin{corollary}
 \label{cor:subadditivity of slopes}
 Consider an extension of torsion-free coherent sheaves on $B$:
 $$
 0 \to \cF \to \cE\to \cG \to 0.
 $$
 Then the following statements are equivalent:
 \begin{enumerate}
 	 \item[i)] $\mu_k(\cF)\leq \mu_k(\cE)$,
 	\item[ii)] $\mu_k(\cF)\leq \mu_k(\cG)$,	  
	\item[iii)] $\mu_k(\cE)\leq \mu_k(\cG)$.
 \end{enumerate}
 The same equivalences hold when replacing $\mu_k$ by any $\nu_i$ or by $\mu_\infty$, and for equalities and strict inequalities as well.
\end{corollary}

\subsection{Adiabatic unstability}
\label{sec:adiabatic unstability}
We immediately get one of the main results of this section.
\begin{proposition}
 \label{prop:unstable case}
 Let $\cE$ be a torsion-free coherent sheaf on $B$. If $\cE$ is strictly unstable with respect to $L$, then $\pi^*\cE$ is strictly unstable on $X$ with respect to adiabatic polarisations $L_k$, for $k \gg 0$.
\end{proposition}
\begin{proof}
Since by Lemma \ref{lem:slopeexpansion} the leading order term of $\mu_k(\pi^*\cE)$ is the slope on the base, the pullback $\pi^* \cF$ of any  strictly destabilising subsheaf $\cF$ of $\cE$ will strictly destabilise $\pi^*\cE$ when $k \gg 0$.
\end{proof}
We turn now to the study of pullbacks of semistable torsion-free coherent sheaves. Let $\cE$ be such a sheaf on $B$ and let $0=\cF_0\subset \cF_1\subset \ldots \subset \cF_l=\cE$ be a Jordan--H\"older filtration for $\cE$. Denote by $\cG_i=\cF_i/\cF_{i-1}$, $i=1\ldots l$, the stable components of its graded object $\Gr(\cE)$. For any $I\subset \lbrace 1, \ldots, l\rbrace$, introduce the torsion-free coherent sheaf on $B$:
\begin{equation}
 \label{eq:cGI}
 \cG_I:= \bigoplus_{i\in I} \cG_i.
\end{equation}
We will be interested in subsheaves of $\pi^*\cE$ coming from the graded object. We thus introduce $\mathbb{I}(\cE)$ to be the collection of non-empty sets of indices $I=\lbrace i_1, \ldots, i_{l'}\rbrace\subsetneq \lbrace 1,\ldots,l\rbrace$ such that there is a nested sequence of coherent sheaves
\begin{equation}
 \label{eq:sequenceEI}
0= \cE_{i_{0}} \subset \cE_{i_1} \subset \ldots \subset \cE_{i_{l'}} \subset\cE
\end{equation}
with $\cE_{i_j}/\cE_{i_{j-1}}=\cG_{i_j}$ for all $j=1\ldots l'$. In that setting, we will denote $\cE_I:=\cE_{i_{l'}}$.
\begin{remark}
\label{rem:muEIandmuGI}
 Note that $\cE_I$ is not uniquely determined by $I$ (for example, two different extensions $0\to\cG_1 \to \cE^i\to \cG_2 \to 0$, $i=1,2$, could be direct summands of $\cE$). However its rank, as well as the quantities $\nu_i(\cE_I)$ and $\mu_k(\pi^*\cE_I)$ only depend on $I$. Indeed, iterating Lemma \ref{lem:additivity of nui}, we see that $\mu_k(\pi^*\cE_I)=\mu_k(\pi^*\cG_I)$ for all $k\geq 1$, while $\rank(\cE_I)=\rank(\cG_I)$ follows by definition. As our arguments will only rely on the rank and the slopes of $\cE_I$, and not $\cE_I$ itself, for any $I\in\mathbb{I}(\cE)$ we can choose any sheaf $\cE_I$ that fits in a sequence as in (\ref{eq:sequenceEI}).
\end{remark}
\begin{proposition}
 \label{prop:semistable unstable case}
 Let $\cE$ be a semistable torsion-free coherent sheaf on $(B,L)$. Assume that there is $I\in \mathbb{I}(\cE)$ such that $\mu_\infty(\cG_I)>\mu_\infty(\cE)$. Then $\pi^*\cE$ is unstable on $X$ with respect to adiabatic polarisations $L_k$, for $k \gg 0$.
\end{proposition}
\begin{proof}
Let $I\in \mathbb{I}(\cE)$ such that $\mu_\infty(\cG_I)>\mu_\infty(\cE)$. By assumption, we can order $I=\lbrace i_1, \ldots, i_{l'}\rbrace$ such that there is a nested sequence of coherent sheaves
$$
0= \cE_{i_{0}} \subset \cE_{i_1} \subset \ldots \subset \cE_{i_{l'}} \subset\cE
$$
with $\cE_{i_j}/\cE_{i_{j-1}}=\cG_{i_j}$ for all $j=1\ldots l'$. By Remark \ref{rem:muEIandmuGI},  $\mu_k(\pi^*\cE_{i_{l'}})=\mu_k(\pi^*\cG_I)$.
But then, for adiabatic classes, $\pi^*\cE_{i_{l'}}$ is a destabilising subsheaf for $\pi^*\cE$, hence the result.
\end{proof}

In Section \ref{sec:stable}, we will prove by an analytical argument that the pullback of a stable vector bundle is stable for adiabatic classes (Theorem \ref{thm:stablecase} and Corollary \ref{cor:stablepullback}). This is independent from the results in this section, and this can be used to settle the case of equality, assuming $\cE$ and $\Gr(\cE)$ to be locally free.
\begin{proposition}
 \label{prop:semistable semistable case}
 Let $\cE$ be a strictly semistable locally free sheaf on $(B,L)$ such that $\Gr(\cE)$ is locally free. Assume that for all $I\in \mathbb{I}(\cE)$,  $\mu_\infty(\cG_I)=\mu_\infty(\cE)$. Then $\pi^*\cE$ is strictly semistable on $X$ with respect to adiabatic polarisations $L_k$, for $k \gg 0$.
\end{proposition}
\begin{proof}
Note that the assumption on the subsheaves of $\cE$ implies that for all $i\in[\![ 1, l]\!]$, the subsheaf $\cF_i$ from the Jordan--H\"older filtration
$$
0\subset \cF_1 \subset \ldots \subset \cF_l= \cE
$$
satisfies $\mu_\infty(\cF_i)=\mu_\infty(\cE)$. But by Lemma \ref{lem:additivity of nui}, by induction on $i$, we deduce that for all $i$, $\mu_\infty(\cG_i)=\mu_\infty(\cE)$. Then, by Corollary \ref{cor:stablepullback}, $\pi^*\cG_i$ is stable for adiabatic polarisations. As the adiabatic slopes of the direct summands of $\pi^*\Gr(\cE)$ are equal, we deduce that it is polystable, and hence semistable, for adiabatic polarisations. As $\pi^*\cE$ is a small deformation of $\pi^*\Gr(\cE)$ (see e.g. \cite{BuSchu}, or Section \ref{sec:basestructurebundle}), by openness of semistability (\cite[Proposition 2.3.1]{HuLe}), the result follows.
\end{proof}
In Section \ref{sec:semistable}, we will consider the last cases for simple semistable locally free sheaves, that is when for all $I\in \mathbb{I}(\cE)$, we have $\mu_\infty(\cG_I)\leq \mu_\infty(\cE)$. In that situation, assuming $(H1)-(H3)$, we will prove that the pullback $\pi^*\cE$ is semistable for adiabatic classes, with stability achieved if and only if for all $I\in \mathbb{I}(\cE)$, we have a strict inequality $\mu_\infty(\cG_I)< \mu_\infty(\cE)$ (note that under hypothesis $(H1)$ the subsheaves $\cE_I$ for $I\in \mathbb{I}(\cE)$ correspond precisely to the subsheaves $\cF_i$ of the Jordan--H\"older filtration).

\begin{remark}
From the algebraic point of view one could suspect that the pullbacks of the sheaves $\cG_I$ are the crucial subsheaves of $\pi^* \cE$ that will determine the slope (un)stability of $\pi^* \cE$ for adiabatic classes. To obtain an algebraic proof of Theorem \ref{thm:ssthmintro} (and also Theorem \ref{thm:semistablemain}), one needs to check that no other subsheaves of $\pi^* \cE$ destabilise. One also needs to establish some uniformity of the expansion of Lemma \ref{lem:slopeexpansion}. In Section \ref{sec:semistable} we show that the pulled back subsheaves $\pi^*\cG_I$ indeed are the crucial ones to check stability on, but through a rather different approach via a differential-geometric argument.
\end{remark}

%%%%
%Section 4
%%%%

\section{Producing HYM connections in the stable case}
\label{sec:stable}

Let $\pi:(X,H)\to (B,L)$ be a holomorphic submersion with connected fibres, and $E\to B$ a holomorphic vector bundle. In this section we prove Theorem \ref{thm:stablecase}. There are two main steps. We first construct approximate solutions to any desired order in Section \ref{sec:stableapprox}, then show that these can be perturbed to genuine solutions when the order is sufficiently good in Section \ref{sec:nonlinear}. From the construction, it will be clear that the sequence of connections that we will produce satisfy the convergence property stated in Theorem \ref{thm:stablecase}. 

We will frequently suppress pullbacks for bundles and endomorphisms from now on, so that e.g. when we speak of the bundle $E$ over $X$, we really mean $\pi^*E \to X$, and the section $\Id_E$ over $X$ is really $\Id_{\pi^* E}$.

\subsection{The approximate solutions}
\label{sec:stableapprox}

When producing the approximate solutions in this section, following the strategy in \cite{fine04}, we will use power series expansions in negative powers of the parameter $k$ related to the adiabatic classes for sections of $\End(\pi^*E)$, and related bundles. In all of Section \ref{sec:stableapprox}, an expression of the form $\sigma(k)=O(k^{-j})$ is to be understood as holding pointwise. Convergence considerations of those expressions with respect to various Sobolev space norms will be addressed in Section \ref{sec:nonlinear}.

Assume that $E$ is stable on $B$, with respect to $c_1(L)$. Then for any $\omega_B \in c_1 \left( L \right)$, there is an hermitian metric $\tilde{h}$ on $E$ which is Hermite--Einstein, i.e. such that $$ \Lambda_{\omega_B} \left( i F_{\tilde{h}} \right) = c \Id_E ,$$ where $F_{\tilde{h}}$ is the curvature of $\tilde{h}$.

We then get a metric $h = \pi^* \tilde{h}$ on $\pi^* E$. Our first step in showing that we can obtain an Hermite--Einstein metric in adiabatic classes on $X$ is to show that $h$ to leading order is a solution to the Hermite--Einstein equation.
\begin{lemma}
\label{lem:pullbackcv}
Let $\omega_k = \omega_X + k \pi^* \omega_B $ be a K\"ahler metric in $c_1 \left( L^{k} \otimes H \right)$ on $X$. Let $\tilde{h}$ be an hermitian metric on $E$ and let $h = \pi^* \tilde{h}$. Then the curvature $F_h$ of the Chern connection of $h$ satisfies
$$ \Lambda_{\omega_k} \left( i F_h \right) = k^{-1} \Lambda_{\omega_B} \left( i \pi^* F_{\tilde{h}}  \right) + O \left( k^{-2} \right).$$
In particular, if $\tilde{h}$ is Hermite--Einstein, so is $h$ to leading order.
\end{lemma}
\begin{proof} From Lemma \ref{lem:expansion contraction}, the contraction operator $\Lambda_{\omega_k}$ satisfies $$ \Lambda_{\omega_k} \left( \alpha \right) = \Lambda_{\scV} \left( \alpha_{\scV} \right) + k^{-1} \Lambda_{\cH} \left( \alpha_{\cH} \right) + O \left( k^{-2} \right),$$
for any $(1,1)$ form $\alpha$. So the result comes from $F_h = \pi^* F_{\tilde{h}}$, which follows by definitions of pullback and Chern connections.
\end{proof}

Next we will show that after a perturbation, this can be improved to arbitrary order. That is, for any $j$ we find $f_{j,k} \in \scGC \left( \pi^* E \right) $ and constants $c_{j,k}$ such that if we let $A_{j,k} = A^{f_{j,k}}$, then for all $k \gg 0$, $$ \Lambda_{\omega_k} \left( i F_{A_{j,k} } \right) = c_{j,k} \Id_E + O \left( k^{-j-1} \right).$$

For this, we will need to understand the linearised operator better, and in particular how it acts on the different components in the splitting $$\Gamma \left( X, \End \pi^*E \right) = \Gamma_0 \left( X, \End \pi^* E \right) \oplus \pi^* \Gamma \left( B, \End E \right) $$ of sections of $\End \pi^* E$. 
The following definition captures the leading order term of this operator.
\begin{definition}\label{def:vertlap} Let $G \to B$ be a vector bundle of rank $r$. The \textit{vertical Laplace operator}, denoted $$\Delta_{\scV, G} : \Gamma \left( X, \pi^*G \right) \to \Gamma \left( X, \pi^* G \right),$$ is the operator 
$$\Delta_{\scV, G} = i \Lambda_{\cV} \big(  \partial_{F} \bar \partial_{F}  - \bar \partial_{F}  \partial_{F} \big),$$ 
where $d_F=\partial_{F} +\bar \partial_{F} $ is the flat connection along the fibres of $\pi$. 
\end{definition}
For $G = \End E$ we will drop the reference to the bundle in the notation, i.e. $$ \Delta_{\scV} = \Delta_{\scV, \End E}.$$ 
Note that $\Delta_{\scV, G}$ vanishes on the $\Gamma \left( B, G \right)$ component of $\Gamma \left( X, \pi^* G \right)$.
Next we define the operator that appears as the subleading order term.
\begin{definition}\label{def:horlap} Let $G \to B$ be a vector bundle of rank $r$ with an hermitian metric $\tilde{h}_G$. The \textit{horizontal Laplace operator}, denoted $$\Delta_{\cH, G} : \Gamma \left( X, \pi^*G \right) \to \Gamma \left( X, \pi^* G \right),$$ is the operator 
$$\Delta_{\cH,G} = i \Lambda_{\cH} \big(  \partial_{A} \bar \partial_{A}  - \bar \partial_{A}  \partial_{A} \big),$$ 
where $A$ is the pullback of the Chern connection of $(G,\tilde{h}_G)$ on $B$.
\end{definition}
As with the vertical operator, when $G = \End E$ we will drop the reference to the bundle in the notation, i.e. $$ \Delta_{\cH} = \Delta_{\cH, \End E}.$$ Note that this then equals $i \Lambda_{\omega_B} \big(  \partial_{A_{\End E}} \bar \partial_{A_{\End E}}  - \bar \partial_{A_{\End E}}  \partial_{A_{\End E}} \big)$.
\begin{remark}
\label{rem:horizontallaplaceandpulback}
If $s$ is pulled back from a section $\tilde{s}$ on $B$, then $\Delta_{\cH, G} (s)$ is the pullback of $\Delta_{\omega_B, G} ( \tilde{s}).$
\end{remark}

The linearised operator then has the following asymptotic behaviour.
\begin{proposition}
\label{prop:expansion linear term}
Let $\Delta_k $ denote the Laplace operator associated to the Chern connection of $h=\pi^*\tilde{h}$ and $\omega_k$. Then
\begin{align*} \Delta_k (s) = \Delta_{\scV} (s) +   k^{-1} \Delta_{\cH} (s)  + O \left( k^{-2}  \right).
\end{align*}
The same expansion also holds at a Chern connection on $\pi^* E \to X$ coming from a complex structure $f_k \cdot \overline \partial_0$ provided $f_k = \Id_E+ s_k$ for some $s_k \in \Gamma \left( X, \End  \pi^* E \right)$ whose base component $s_{B,k}$ satisfies $s_{B,k} = O(k^{-1})$ and whose vertical component $s_{F, k}$ satisfies $s_{F,k} = O(k^{-2})$.
\end{proposition}
\begin{proof} Recall from Lemma \ref{lem:expansion contraction} that the contraction operator $\Lambda_{\omega_k}$ expands as $$ \Lambda_{\omega_k} \left( \alpha \right) = \Lambda_{\scV} \left( \alpha_{\scV} \right) + k^{-1} \Lambda_{\cH} \left( \alpha_{\cH} \right) + O \left( k^{-2} \right).$$
Working in local trivialisations, we see that the Chern connection of $h$ is of the form $d_A = d_F + \pi^* d_{\tilde{A}}$, where $d_F$ is the trivial connection on the fibres of $\pi$ and $\tilde A$ the Chern connection of $(E,\tilde h)$. The operator  $ \partial_{A_{\End E}} \bar \partial_{A_{\End E}}  - \bar \partial_{A_{\End E}}  \partial_{A_{\End E}}$ will decompose accordingly as a sum of a purely vertical part induced by $d_F$, a purely horizontal part induced by $\pi^*d_{\tilde A}$, and a mixed vertical-horizontal part induced by those two covariant derivatives. The vertical component of $ \partial_{A_{\End E}} \bar \partial_{A_{\End E}}  - \bar \partial_{A_{\End E}}  \partial_{A_{\End E}}$ will then be the fibrewise operator $\partial_F \bar \partial_F - \bar \partial_F \partial_F$, and the horizontal component will be the pullback of the corresponding operator from the base. The result then immediately follows as, to the leading two orders, one does not see any mixed terms in the expansion of $\Lambda_{\omega_k}$.

Finally, we consider the statement regarding the perturbed Chern connections. Under those perturbations, the vertical part of $ \partial_{A_{\End E}} \bar \partial_{A_{\End E}}  - \bar \partial_{A_{\End E}}  \partial_{A_{\End E}}$ only changes at order $k^{-2}$. Also, while the the horizontal part of $ \partial_{A_{\End E}} \bar \partial_{A_{\End E}}  - \bar \partial_{A_{\End E}}  \partial_{A_{\End E}}$ only changes at order $k^{-1}$, the change in curvature is at order $k^{-2}$, because of the term $k^{-1} \Lambda_{\cH}$ when contracting this part with $\omega_k$.
\end{proof}

Consider the decomposition
\begin{equation}
 \label{eq:decomposition of gauge group}
 \Gamma(X,\End(\pi^*E))=\Gamma_0(X,\End(\pi^*E))\oplus \pi^* \Gamma(B,\End(E)),
\end{equation}
and notice that $\Delta_\scV$ sends $\Gamma_0(X,\End(\pi^*E))$ to itself. We have :

\begin{lemma}
 \label{lem:vertical laplacian invertible}
 The vertical Laplacian $$\Delta_\scV : \Gamma_0(X,\End(\pi^*E)) \to \Gamma_0(X,\End(\pi^*E)) $$ is invertible as a linear map.
\end{lemma}
\begin{proof}
 By restriction to each fibre $X_b=\pi^{-1}(b)$, $\Delta_\scV$ defines a differentiable family of self-adjoint elliptic operators $(\Delta_{\scV,b} :\Gamma_0(X_b,\End(\pi^* E_b)) \to \Gamma_0(X_b,\End(\pi^* E_b)))_{b\in B}$ over $B$ in the sense of Kodaira and Spencer \cite{kodairaspencer60} (be careful that the notations $X$ and $B$ are switched from \cite{kodairaspencer60} to our conventions). Denote by $G_b$ the Green operator of $\Delta_{\scV,b}$ (as in \cite[page 48]{kodairaspencer60}).  By construction, the kernel of each $\Delta_{\scV,b}$ is trivial, and thus by \cite[Theorem 5]{kodairaspencer60}, the associated family of Green operators $(G_b)_{b\in B}$ is a differentiable family. In particular, for each $\sigma \in \Gamma_0(X,\End(\pi^*E))$ the solutions $s_b=G_b(\sigma_{\vert X_b})$ to $\Delta_{\scV,b} s_b = \sigma_{\vert X_b}$ are differentiable in $b$, meaning that there is a smooth section $s\in \Gamma_0(X,\End(\pi^*E))$ such that $s_{\vert X_b} =s_b$ satisfying $\Delta_\scV s = \sigma$. 
\end{proof}

Recall that the (co)-kernel of the Laplacian are the constant multiples of the identity if $E$ is simple. Denote by $\Gamma_{\Id_E^\perp}(B,\End(E))$ the $L^2$ orthogonal complement of $\mathbb{C} \cdot \Id_E$ in $\Gamma(B,\End(E))$ with respect to the natural inner product on sections induced by $\om_B$ and $\tilde h$ (the same used to define the adjoint operator of $d_A$). Then:
\begin{lemma}
 \label{lem:horizontal laplacian invertible}
 Assume $E \to B$ is simple. Then the horizontal Laplacian $$\Delta_\cH : \Gamma_{\Id_E^\perp}(B,\End(E)) \to \Gamma_{\Id_E^\perp}(B, \End(E))$$ is invertible.
\end{lemma}
To improve the estimate on the approximate solution to arbitrarily high order in the $k$-expansion, we will also need the following lemma:
\begin{lemma}
 \label{lem:quadratic term is O(k-1)}
 Let $s_B\in \Gamma(B,\End(E))$ and denote by $Q_k$ the quadratic term in the Taylor expansion of the Hermite--Einstein operator, for $\ep \ll 1$:
 \begin{equation}
  \label{eq:taylor expansion HE second term}
  \Lambda_{\omega_k}( i F_{A^{\exp(\ep s_B)}}) = \Lambda_{\om_k}( i F_A) + \ep \Delta_k (s_B) + \ep^2 Q_k(s_B,s_B) + O(\ep^3).
 \end{equation}
Then $Q_k(s_B,s_B)=O(k^{-1})$.
\end{lemma}
\begin{proof}
Denote by $e^{\ep s_B}\cdot E$ the holomorphic vector bundle whose underlying complex vector bundle is the same as for $E$, but with holomorphic connection $e^{-\ep s_B}\circ \delb_E \circ e^{\ep s_B}$, for $\delb_E$ the holomorphic connection of $E$.
 Applying Lemma \ref{lem:pullbackcv} to $(e^{\ep s_B}\cdot E, \tilde h)$ and $(\pi^* e^{\ep s_B}\cdot E, \pi^* \tilde h)$, one has 
 $$\Lambda_{\om_k}\left(i F_{A^{\exp(\ep s_B)}}\right)=O(k^{-1}),$$
 hence the result. 
\end{proof}
We then obtain:
\begin{proposition}
 \label{prop:kexpstable}
 Assume that $\tilde h$ is Hermite--Einstein. Then for all $j\in \N^*$, there exist
\begin{itemize} 
\item gauge transformations $f_{j,k}\in\cG^\C(\pi^*E)$,
\item[]
\item constants $c_{j,k}$ ,
\end{itemize}
such that if we let $A_{j,k} = A^{f_{j,k}}$, then for all $k \gg 0$, $$ \Lambda_{\omega_k} \left( i F_{A_{j,k} } \right) = c_{j,k} \Id_E + O \left( k^{-j-1} \right).$$
Moreover, $f_{j,k}$ admits an expansion $f_{j,k}=\Id_E+s_{B,j,k}+s_{F,j,k}$ with base component $s_{B,j,k} = O(k^{-1})$ and vertical component $s_{F,j,k} = O(k^{-2})$.
\end{proposition}

\begin{proof}
The result follows from an inductive argument on $j\in\N^*$. The case $j=1$ follows by taking $f_{1,k}=\Id_E$. Then  $c_{1,k}  \Id_E = k^{-1}\Lambda_{\om_B}(i\pi^*F_{\tilde h})$, which is constant by the assumption on $\tilde h$, from Lemma \ref{lem:pullbackcv}. Note also that the curvature of the connections $A_{1,k}=A$ admit a Taylor expansion in inverse powers of $k$.

Assume now that the result holds up to step $j\geq 1$. That is, there are constants $c_{j,k}$, gauge transformations  $f_{j,k}=\Id_E+s_{B,j,k}+s_{F,j,k}$ with base component $s_{B,j,k} = O(k^{-1})$ and vertical component $s_{F,j,k} = O(k^{-2})$  such that the contracted curvature   $\Lambda_{\omega_k} \left( i F_{A_{j,k} } \right)$ has a Taylor expansion in inverse powers of $k$. Then, there is an element $\sigma^{j+1}\in \Gamma(X,\End_H (\pi^* E,h))$ such that for $k \gg 0$, one has
$$
 \Lambda_{\omega_k} \left( i F_{A_{j,k} } \right) = c_{j,k} \Id_E +\sigma^{j+1} k^{-j-1}+ O \left( k^{-j-2} \right).
$$
Note that by definition of the complex gauge group action, the curvature of any hermitian connection is always skew-hermitian, and so the error term $\sigma^{j+1}$ is an hermitian endomorphism.

Now, by Proposition \ref{prop:expansion linear term}  and Lemma \ref{lem:quadratic term is O(k-1)}, we have that if $s^{j+1}_B\in \pi^* \Gamma_{\Id_E^\perp} (B,\End_H (E, \tilde h) )$, $s_F^{j+1}\in \Gamma_0(X,\End_H (\pi^*E, h) ) $, and if we set
$$
f_{j+1,k}= f_{j,k}\exp( k^{-j} s_B^{j+1})\exp (k^{-j-1} s_F^{j+1} )
$$
then
\begin{align*}
\Lambda_{\om_k} (i F_{A_{j+1,k}}) &= \Lambda_{\om_k} (i F_{A_{j,k}}) + \left( \Delta_\cH s^{j+1}_B + \Delta_\scV s^{j+1}_F \right) k^{-j-1} + O(k^{-j-2}) \\
&=c_{j,k} \Id_E + \left( \sigma^{j+1} +\Delta_\cH s^{j+1}_B + \Delta_\scV s^{j+1}_F \right) k^{-j-1} + O(k^{-j-2}),
\end{align*}
for $k \gg 0$, since $j \geq 1$.
Write
$$
 \sigma^{j+1}=c_{j+1} \Id_E + \sigma^{j+1}_B + \sigma^{j+1}_F
$$
with $c_{j+1}\in \R$, $\sigma^{j+1}_B\in \Gamma_{\Id_E^\perp}(B,\End_H (E, \tilde h))$ and $\sigma_F^{j+1}\in \Gamma_0(X,\End_H (\pi^*E, h) )$. By Lemmas \ref{lem:horizontal laplacian invertible} and \ref{lem:vertical laplacian invertible}, we can find $s_B^{j+1}$ that solves $\Delta_\cH s^{j+1}_B = -\sigma_B^{j+1}$ and $s_F^{j+1}$ that solves $\Delta_\scV s^{j+1}_F=-\sigma^{j+1}_F$. Note also that the changes made to produce $A_{j+1,k}$ from $A_{j,k}$ preserves having a Taylor expansion in inverse powers of $k$ and in particular, we can write
$$
 \Lambda_{\omega_k} \left( i F_{A_{j+1,k} } \right) = c_{j+1,k} \Id_E +\sigma^{j+2} k^{-j-2}+ O \left( k^{-j-3} \right)
$$
for some element $\sigma^{j+2}$ (which depends on the previously chosen $s^i_B$ and $s^i_F$). By construction, $f_{j+1,k}=\Id_E+s_{B,j+1,k}+s_{F,j+1,k}$ with base component $s_{B,j+1,k} = O(k^{-1})$ and vertical component $s_{F,j+1,k} = O(k^{-2})$. The result follows.
\end{proof}

\subsection{Perturbing to a genuine solution}
\label{sec:nonlinear}

We now perturb the approximate solutions constructed above to genuine solutions. The proof relies on a quantitative version of the implicit function theorem, which we now recall.
\begin{theorem}[{\cite[Theorem 4.1]{fine04},\cite[Theorem 25]{bronnle15}}]
\label{thm:quantimpl}
Let $\scR : V \to W$ be a differentiable map of Banach spaces. Suppose the derivative $D \scR$ at $0$ is surjective with right-inverse $\cQ$. Let
\begin{itemize} \item $\delta_1$ be the radius of the closed ball in $V$ where $\scR - D \scR$ is Lipschitz of constant $\frac{1}{2 \| \cQ \|}$;
\item $\delta = \frac{\delta_1}{2 \| \cQ \|}.$
\end{itemize}
Then for all $w \in W$ such that $\| w - \scR (0) \| < \delta$, there exists $v \in V$ such that $\scR(v) = w.$
\end{theorem}

We now apply this to the hermitian Yang--Mills operator $\Psi_{j,k} $ given by $$ (s, \lambda) \mapsto  i \Lambda_{\omega_k} \left( F_{A_{j,k}^{\exp(s)}} \right)-  \lambda \Id_E.$$ The constants are added so that the linearisation is surjective, and so we seek a zero of this map (rather than some constant multiple of the identity). For a Riemannian metric $g$ on $X$, let $L^2_{d} (g)$ denote the Sobolev space $W^{d,2}(X, g)$ of order $2$ and $d$ derivatives, with respect to the metric $g$. If $g$ is K\"ahler with K\"ahler form $\omega$, we may use $L^2_d (\omega)$ to mean $L^2_d (g)$. When $d=0$, we omit the subscript. For Sobolev spaces associated to $\pi^*E$ or $\End \pi^*E$, we use the metric $h$ (and the metric it induces on $\End \pi^*E$) on the bundle. We will use the mean value theorem to establish the required bound on the radius of the closed ball in $V = L^2_{d+2} \left( \End_H (\pi^* E, h) , \omega_k \right) \times \R $ on which the non-linear part $$\scN_{j,k} = \Psi_{j,k} - D \Psi_{j,k}$$ of the HYM operator has the appropriate Lipschitz constant. Thus we will begin by establishing some bounds on $D \Psi_{j,k} = \Delta_{j,k}$ and its right inverse. Note that this inverse is known to exist by Lemma \ref{lem:pullbackissimple}.

\begin{proposition}
\label{prop:inversebound}
For each $d$, there is a $C>0$, independent of $k$ and $j$, such that the right inverse $\scP_{j,k}$ of the map $L^2_{d+2}(\omega_k) \times \R \to L^2_d(\omega_k)$ given by
$$
(s,c) \mapsto \Delta_{j,k}(s) - c \Id_E
$$ 
satisfies
$$ \| \scP_{j,k} \|_{L^2_{d} \to (L^2_{d+2} \times \R)} \leq C k,$$
 where $\| \cdot \|_{L^2_{d} \to (L^2_{d+2} \times \R)}$ is the operator norm induced from the norms on $L^2_d(\omega_k)$ and $L^2_{d+2}(\omega_k) \times \R$.
\end{proposition}

The proof follows closely the strategy of the analogous \cite[Theorem 6.9]{fine04}. A key step is to establish the following Poincar\'e inequality.  Recall that $A_{j,k}$ denotes the Chern connection constructed in Proposition \ref{prop:kexpstable}. We will use the same notation for the induced connection on $\End \pi^*E$.

\begin{lemma}\label{lemma:poincareineq} For each $j$ there exists a $C>0$ such that for all $k \gg 0$ we have that for any $s \in \Gamma \left(X, \End \pi^*E \right)$ whose trace is of average $0$ with respect to $\omega_k$, we have
\begin{align*} \| A_{j,k} s \|^2_{L^2 (\omega_k)} \geq  C k^{-1} \| s \|^2_{L^2 (\omega_k)}.
\end{align*}
\end{lemma}
Note that since $\Delta_{j,k}$ is the self-adjoint operator produced from $A_{j,k}$, the lower bound in Lemma \ref{lemma:poincareineq} is equivalent to the same lower bound for the (absolute value) of the first non-zero eigenvalue of $\Delta_{j,k}$. In particular, it implies the lower bound 
\begin{align*} \| \Delta_{j,k}(s) \|_{L^2 (\omega_k)} \geq  C k^{-1} \| s \|_{L^2 (\omega_k)}
\end{align*}
for all $s$ that are $\omega_k$ orthogonal to $\Id_E$ and, in turn, the upper bound 
\begin{align*} \| \scP_{j,k} (s) \|_{L^2 (\omega_k)} \leq  C k \| s \|_{L^2 (\omega_k)}
\end{align*}
in $L^2$ for any $s$ in the image of $\Delta_{j,k}$, where we have used the same notation for the inverse as an operator $L^2 \to L^2 \cap \langle \Id_E \rangle^{\perp}$.

\begin{proof} In addition to the Riemannian metric $g_k$ whose K\"ahler form is $\omega_k$, we will use the Riemannian metric $\hat g_k$ which is a product of $(g_X)_{\scV}$ and $k g_B $ in the splitting $TX = \scV \oplus \cH.$ The corresponding metrics on any tensor bundles induced by $g_k$ and $\hat g_k$ are then uniformly equivalent (\cite[Lemma 6.2]{fine04}). In the proof, the various $C_i$ appearing are positive constants.

We first note that $A_0 = A_{k,0}$ is independent of $k$, as this is the Chern connection of the pullback of the complex structure on $E$. Moreover, $A_{j,k} = A_0 + O(k^{-1}).$ With respect to the fixed metric $\hat g_1,$ we have the Poincar\'e inequality $$\| A_0 \hat s \|^2_{L^2 (\hat g_1)} \geq C_1 \| \hat s \|^2_{L^2 (\hat g_1)},$$ for any $\hat s$ whose trace has average $0$ with respect to $\hat g_1$. Since $A_{j,k}$ is an $O(k^{-1})$ perturbation of $A_0$, we therefore have a similar inequality
\begin{align}
\label{g1PI}
\| A_{j,k} \hat s \|^2_{L^2 (\hat g_1)} \geq C_2 \| \hat s \|^2_{L^2 (\hat g_1)},
\end{align}
for all $k \gg 0$.

We use this to get the desired inequality. Let $\gamma$ be the constant such that the trace of $s- \gamma \Id_E$ has average $0$ on $X$ with respect to $\hat g_1$. One has
\begin{align*} \| A_{j,k} s \|^2_{L^2(g_k)} & \geq C_3 \|A_{j,k} s \|^2_{L^2 (\hat g_k) } \\
&= C_3 \| A_{j,k} ( s- \gamma  \Id_E ) \|^2_{L^2 (\hat g_k) } \\
&\geq C_3 k^{n-1} \| A_{j,k} ( s- \gamma  \Id_E ) \|^2_{L^2 (\hat g_1) } ,
\end{align*}
using the uniform equivalence of $g_k$ and $\hat g_k$, together with the pointwise estimate  $\vert \sigma \vert_{\hat g_k}^2 \geq k^{-1} \vert \sigma \vert_{\hat g_1}^2$ on sections $\sigma$ of $\End E \otimes \Lambda^{1} (X)$ and the equality $\textnormal{Vol } (\hat g_k ) = k^n \textnormal{Vol } (\hat g_1)$ for the volume forms.

By our choice of $\gamma$, we can combine this with the inequality \eqref{g1PI}, and we get that
\begin{align*}\| A_{j,k} s \|^2_{L^2(g_k)} & \geq C_4 k^{n-1} \| s- \gamma  \Id_E  \|^2_{L^2 (\hat g_1) } .
\end{align*}
Returning to the original metric $g_k$ by reversing the initial argument above (noting that the pointwise metrics now are independent of $k$ as we  have sections of $\End E$, rather than $\End E \otimes \Lambda^1 (X)$), we then get that
\begin{align*}\| A_{j,k} s \|^2_{L^2(g_k)} &\geq C_4 k^{-1} \| s- \gamma \Id_E \|^2_{L^2 (\hat g_k) } \\
&\geq C_5 k^{-1} \| s- \gamma \Id_E \|^2_{L^2 ( g_k) }  \\
&\geq C_5 k^{-1} \| s  \|^2_{L^2 ( g_k) } ,
\end{align*}
where the last line follows because the trace of $s$ has average $0$ on $X$ with respect to $g_k$, and so $\| s- \gamma \Id_E \|^2_{L^2 ( g_k) } = \| s \|^2_{L^2 ( g_k) } + \|  \gamma \Id_E \|^2_{L^2 ( g_k) }.$
\end{proof}

The second key element to proving Proposition \ref{prop:inversebound} is the following Schauder estimate, which is uniform in the parameter $k$. This is analogous to \cite[Lemma 5.9]{fine04}.
\begin{proposition}
\label{prop:schaudernew}
There exists a $C>0$, independent of $k$, such that 
\begin{align}
\label{eq:schaudernew}
\| \sigma \|_{L_{d+2}^2(X, \End \pi^* E)}   \leq C \left( \| \sigma \|_{L^2(X, \End \pi^* E)} + \| \Delta_{j,k} \left( \sigma \right)  \|_{L^2_{d}(X, \End \pi^* E)} \right).
\end{align}
\end{proposition}

We first show an analogue of the above locally on $B$. For an open set $U$ in $B$, we will use the shorthand notation $L_{d}^2(U) = L_{d}^2(\pi^{-1}(U), \End \pi^* (E))$ for the $L^2_{d}$ Sobolev space over $U$, with respect to the Hermitian metric $h$ on $\End \pi^*E$ and the K\"ahler metric $\omega_k$ on $X$. If $\widetilde \omega$ is another K\"ahler metric on $\pi^{-1}(U)$, we will use $L_{d}^2(U, \widetilde \omega)$ to denote the corresponding Sobolev space where we have replaced $\omega_k$ with $\widetilde \omega$.

\begin{lemma}
\label{lem:schauderlocal}
Fix a point $b \in B$ and a coordinate system $U$ centered at $b$. Then for every sufficiently small coordinate disk $D_{2r} \subset U$ of radius $2r$, there exists a constant $C>0$ (that depends on $b$, the coordinates chosen, $r$ and $d$, but \emph{not} $k$), such that for all $\sigma \in L_{d+2}^2(D_{2r})$
\begin{align}
\label{eq:schauderlocal}
\| \sigma \|_{L_{d+2}^2(D_{r})}   \leq C \left( \| \sigma \|_{L^2(D_{2r})} + \| \Delta_{j,k} \left( \sigma \right)  \|_{L^2_{d}(D_{2r})} + k^{-1} \| \sigma \|_{L_{d+2}^2(D_{2r})}   \right)
\end{align}
for all $k \gg 0$.
\end{lemma}
\begin{proof}
We will first establish the bound for the operator 
$$
\widetilde \Delta_k = \Delta_{\scV} (s) +   k^{-1} \Delta_{\cH} (s),
$$ 
which by Proposition \ref{prop:expansion linear term} captures the leading order asymptotic expansion of the linearised operator when $j=0$. More precisely, we begin by showing that there is a $C>0$ which is independent of $k$ such that
\begin{align}
\label{eq:schauderlocal1}
\| \sigma \|_{L_{d+2}^2(D_{r})}   \leq C \left( \| \sigma \|_{L^2(D_{2r})} + \| \widetilde \Delta_k \left( \sigma \right)  \|_{L^2_{d}(D_{2r})} + k^{-1} \| \sigma \|_{L_{d+2}^2(D_{2r})} \right).
\end{align}

 Note that in the smooth category, the fibration is locally a product $F \times D_{2r}$. Moreover, it suffices to establish the estimate in Sobolev spaces defined for a metric that is a product $\widetilde \omega_k = \omega_F + k \omega_B$, since if the radius $r$ is chosen sufficiently small, then the actual metric $\omega_k$ is uniformly equivalent to such a metric, with constant of uniformity independent of $k$ (\cite[Section 5.3]{fine04}). Using this metric allows for a separation of variables argument to establish the estimate.

We will present the argument for $d=0$ -- the argument for higher values of $d$ is exactly the same, only with more notation. Note that with respect to the product metric $\widetilde\omega_k$, the $C^2(D_{r}, \widetilde \omega_k)$ semi-norm $\vert \cdot \vert_{C^2(D_{r}, \widetilde \omega_k)}$ splits \emph{orthogonally} as 
\begin{align*}
\vert \sigma \vert^2_{C^2(D_{r}, \widetilde \omega_k)} = \vert \sigma \vert^2_{C^2_{\cV}(D_{r}, \widetilde \omega_k)} + \vert \sigma \vert^2_{C^2_{\cH}(D_{r}, \widetilde \omega_k)}  + \vert \sigma \vert^2_{C^2_{mixed}(D_{r}, \widetilde \omega_k)},
\end{align*}
where $\vert \sigma \vert_{C^2_{\cV}(D_{r}, \widetilde \omega_k)} $ denotes differentiating just in the fibre direction, $\vert \sigma \vert_{C^2_{\cH}(D_{r}, \widetilde \omega_k)} $ denotes differentiating just in the base direction, and $\vert \sigma \vert_{C^2_{mixed}(D_{r}, \widetilde \omega_k)} $ denotes differentiating exactly once in each of the fibre and base directions. Moreover, since $\widetilde \omega_k$ scales the base direction by $k$,
\begin{align*}
\vert \sigma \vert_{C^2_{\cH}(D_{r}, \widetilde \omega_k)}  =& k^{-2} \vert \sigma \vert_{C^2_{\cH}(D_{r}, \widetilde \omega_1)} \\
 \vert \sigma \vert_{C^2_{mixed}(D_{r}, \widetilde \omega_k)} =& k^{-1} \vert \sigma \vert_{C^2_{mixed}(D_{r}, \widetilde \omega_k)}.
\end{align*}

Now, $\Delta_{\cV}$ restricted to a fibre is elliptic on that fibre. Thus for each $b \in U$, we get the estimate 
\begin{align}
\label{eq:fibreestimate}
\left( \int_{F \times \{b\}} \vert \sigma \vert^2_{C^2_{\cV}(D_{r}, \widetilde \omega_k)}  \omega_F^m \right)^{\frac{1}{2}}  \leq C \left( \| \sigma \|_{L^2(F \times \{b\}, \omega_F)} +  \| \Delta_{\cV} \sigma \|_{L^2(F \times \{b\}, \omega_F)} \right).
\end{align} 
Now, seen as an operator on $F$ for each $b \in U$, $\Delta_{\cV}$ is a smoothly varying elliptic operator with $b$. Moreover, $D_{r}$ is relatively compact in $U$. In particular, the constants of ellipticity of the $\Delta_{\cV}$ can be uniformly bounded from above and below on the whole of $D_{r}$. On a fixed compact manifold with a fixed metric, the Schauder estimates only depend on the constants of ellipticity and the norm of the coefficients of the operator, and so the above implies that we can choose a $C>0$ so that \eqref{eq:fibreestimate} holds uniformly for all $b \in D_r$. Integrating over the base, we therefore get that
\begin{align*}
\left( \int_{\pi^{-1}(D_r)} \vert \sigma \vert^2_{C^2_{\cV}(D_{r}, \widetilde \omega_k)} \widetilde \omega_k^{m+n} \right)^{\frac{1}{2}} &\leq C \left( \| \sigma \|_{L^2(D_r, \widetilde \omega_k)} +  \| \Delta_{\cV} \sigma \|_{L^2(D_r, \widetilde \omega_k)} \right).
\end{align*}
Here we used that since the metric is a product, $\widetilde \omega_k^{m+n} = {m+n \choose n} k^n \omega_F^m \wedge \omega_B^n$. Also, in a completely analogous way, we have a splitting $C^1 = C^1_{\cV} + C^1_{\cH}$ of the $C^1$ semi-norm, and we have the bound 
\begin{align*}
\left( \int_{\pi^{-1}(D_r)} \vert \sigma \vert^2_{C^1_{\cV}(D_{r}, \widetilde \omega_k)} \widetilde \omega_k^{m+n} \right)^{\frac{1}{2}} &\leq C \left( \| \sigma \|_{L^2(D_r, \widetilde \omega_k)} +  \| \Delta_{\cV} \sigma \|_{L^2(D_r, \widetilde \omega_k)} \right).
\end{align*}

Next, we consider the $C^2_{\cH}$ term. First note that because the operator $\Delta_{\cH}$ is elliptic on $U$, we get from the usual local Schauder estimates that there is a $C>0$ such that for each $x \in F$,
\begin{align*}
\| \sigma \|_{L_{2}^2(\{x\} \times D_{r}, \omega_B)}   \leq C \left( \| \sigma \|_{L^2(\{x\} \times D_{2r}, \omega_B)} + \|  \Delta_{\cH} \left( \sigma \right)  \|_{L^2(\{x\} \times D_{2r}, \omega_B)} \right).
\end{align*}
In particular, using the scaling properties of the $C^2_{\cH}$ semi-norm with respect to $\widetilde \omega_k$, we get that 
\begin{align*}
\left( \int_{\pi^{-1}(D_r)} \vert \sigma \vert^2_{C^2_{\cH}(D_{r}, \widetilde \omega_k)} \widetilde \omega_k^{m+n} \right)^{\frac{1}{2}} &\leq C k^{-2} \left( \| \sigma \|_{L^2(D_{2r}, \widetilde \omega_k)} +  \| \Delta_{\cH} \sigma \|_{L^2(D_{2r}, \widetilde \omega_k)} \right) \\
&\leq C k^{-1} \left( \| \sigma \|_{L^2(D_{2r}, \widetilde \omega_k)} +  \| k^{-1} \Delta_{\cH} \sigma \|_{L^2(D_{2r}, \widetilde \omega_k)} \right).
\end{align*}
Similarly, since the $C^1_{\cH}$ semi-norm scales like $k^{-1}$, we get that 
\begin{align*}
\left( \int_{\pi^{-1}(D_r)} \vert \sigma \vert^2_{C^1_{\cH}(D_{r}, \widetilde \omega_k)} \widetilde \omega_k^{m+n} \right)^{\frac{1}{2}} &\leq C k^{-1} \left( \| \sigma \|_{L^2(D_{2r}, \widetilde \omega_k)} +  \| \Delta_{\cH} \sigma \|_{L^2(D_{2r}, \widetilde \omega_k)} \right) \\
&\leq C \left( \| \sigma \|_{L^2(D_{2r}, \widetilde \omega_k)} +  \| k^{-1} \Delta_{\cH} \sigma \|_{L^2(D_{2r}, \widetilde \omega_k)} \right).
\end{align*}

At this point we have established the bound \eqref{eq:schauderlocal1} for the $C^1$-norm as well as the $C^2_{\cV}$ and $C^2_{\cH}$ components of the $C^2$-norm. It remains to establish the bound for the $C^2_{mixed}$ semi-norm. Now, since $\widetilde \Delta_1 = \Delta_{\cV} + \Delta_{\cH}$ is elliptic on $U$, there is a $C>0$ such that 
\begin{align*}
\left( \int_{\pi^{-1}(D_r)} \vert \sigma \vert^2_{C^2_{mixed }(D_{r}, \widetilde \omega_1)} \widetilde \omega_1^{m+n} \right)^{\frac{1}{2}} &\leq C \left( \| \sigma \|_{L^2(D_{2r}, \widetilde \omega_1)} +  \| \widetilde \Delta_1 \sigma \|_{L^2(D_{2r}, \widetilde \omega_1)} \right)
\end{align*}
which from the way $C^2_{mixed}$ scales with $k$ immediately gives that
\begin{align*}
\left( \int_{\pi^{-1}(D_r)} \vert \sigma \vert^2_{C^2_{mixed }(D_{r}, \widetilde \omega_k)} \widetilde \omega_k^{m+n} \right)^{\frac{1}{2}} &\leq C k^{-1} \left( \| \sigma \|_{L^2(D_{2r}, \widetilde \omega_k)} +  \| \widetilde \Delta_1 \sigma \|_{L^2(D_{2r}, \widetilde \omega_k)} \right).
\end{align*}
Now, for $k\geq1$, we then get that
\begin{align*}
&\| \widetilde \Delta_k \sigma \|_{L^2(D_{2r}, \widetilde \omega_k)}^2 - \|  k^{-1} \widetilde \Delta_1 \sigma \|^2_{L^2(D_{2r}, \widetilde \omega_k)} \\
=& (1-\frac{1}{k^2}) \| \Delta_{\cV} (\sigma)\|_{L^2(D_{2r}, \widetilde \omega_k)}^2 - 2 ( \frac{1}{k^2} - \frac{1}{k}) \langle \Delta_{\cV}(\sigma), \Delta_{\cH} (\sigma) \rangle_{L^2(D_{2r}, \widetilde \omega_k)} \\
\geq& 2 ( \frac{1}{k^2} - \frac{1}{k}) \big \vert \langle \Delta_{\cV}(\sigma), \Delta_{\cH} (\sigma) \rangle_{L^2(D_{2r}, \widetilde \omega_k)} \big \vert \\
\geq & - 2 \left( \frac{1}{k} - \frac{1}{k^2} \right)  \| \Delta_{\cV}(\sigma) \|_{L^2(D_{2r}, \widetilde \omega_k)} \| \Delta_{\cH} (\sigma) \|_{L^2(D_{2r}, \widetilde \omega_k)},
\end{align*}
where in the last line, we used the Cauchy--Schwarz inequality. Since $\Delta_{\cV}$ and $\Delta_{\cH}$ are bounded operators, we obtain from this that 
\begin{align*}
\|  k^{-1} \widetilde \Delta_1 \sigma \|_{L^2(D_{2r}, \widetilde \omega_k)} 
\leq \| \widetilde \Delta_k \sigma \|_{L^2(D_{2r}, \widetilde \omega_k)} + C'k^{-1} \|\sigma \|^2_{L^2_2(D_{2r}, \widetilde \omega_k)},
\end{align*}
for some $C'>0$ that is independent of $k$. Thus we have shown that there is a $C>0$ such that
\begin{align*}
&\left( \int_{\pi^{-1}(D_r)} \vert \sigma \vert^2_{C^2_{mixed }(D_{r}, \widetilde \omega_k)} \widetilde \omega_k^{m+n} \right)^{\frac{1}{2}} \\
\leq & C \left( \| \sigma \|_{L^2(D_{2r}, \widetilde \omega_k)} +  \| \widetilde \Delta_k \sigma \|_{L^2(D_{2r}, \widetilde \omega_k)} +  k^{-1} \| \sigma \|_{L_{d+2}^2(D_{2r})}  \right).
\end{align*}
The estimate \eqref{eq:schauderlocal1} then follows from all of the semi-norm estimates above.

Finally, to obtain the inequality \eqref{eq:schauderlocal} from \eqref{eq:schauderlocal1}, we note that 
$$
\| (\Delta_{j,k} - \widetilde \Delta_k) (\sigma)  \|_{L_{d+2}^2(D_{2r})}  \leq k^{-1} \| \sigma \|_{L_{d+2}^2(D_{2r})}.
$$
Indeed, from Proposition \ref{prop:expansion linear term} and the way we have perturbed $\delb$ to $\delb_{j,k}$, the vertical component  and base component of $\Delta_{j,k} - \widetilde \Delta_k$ is $O(k^{-1})$ and $O(k^{-2})$, with respect to fixed fibrewise and base metrics, respectively. Hence this is $O(k^{-1})$ with respect to the metric $\omega_k$. The result follows directly from this and \eqref{eq:schauderlocal1}.
\end{proof}

With this in place, we are ready to prove Proposition \ref{prop:schaudernew}.
\begin{proof}[Proof of Proposition \ref{prop:schaudernew}]
Let $D_{2r_i} (b_i)$ be a finite collection of coordinate disks of radius $2r_i$ about points $b_i \in B$ as in Lemma \ref{lem:schauderlocal} such that the $D_{r_i}(b_i)$ cover $B$ -- here, as in Lemma \ref{lem:schauderlocal}, the radii are computed with respect to the coordinates chosen around $b$.  Let $\tilde \chi_i$ be a partition of unity subordinate to this cover, so that $\chi_i = \pi^* \tilde \chi_i$ is a partition of unity subordinate to the cover $\{ \pi^{-1}(D_{r_i}(b_i)) \}$ of $X$. Then for $\sigma \in L_{d+2}^2 (X, \End \pi^* E)$, we have that 
\begin{align*}
\| \sigma \|_{L_{d+2}^2} =& \| \sum_i \chi_i \sigma \|_{L_{d+2}^2} \\
\leq& \sum_i \| \chi_i \sigma \|_{L_{d+2}^2}  \\
=& \sum_i \| \chi_i \sigma \|_{L_{d+2}^2(D_{r_i}(b_i))} \\
\leq& C \sum_i \left( \| \chi_i \sigma \|_{L^2(D_{2r})} + \| \Delta_{j,k} \left( \chi_i \sigma \right)  \|_{L^2_{d}(D_{2r})} + k^{-1} \| \chi_i \sigma \|_{L_{d+2}^2(D_{2r})}  \right),
\end{align*}
where in the last line, we have used Equation \eqref{eq:schauderlocal}, picking $C$ to be the maximum of the constants associated to disks $D_{r_i}(b_i)$. Note also that by the Leibniz rule, there is a constant $C'>0$ such that $\| \chi_i \sigma \|_{L_{d+2}^2(D_{2r})} \leq C' \| \chi_i \sigma \|_{L_{d+2}^2(X)}$  -- the constant $C'$ depends on the $\chi_i$ and $d$, but not on $\sigma$. Then for all sufficiently large $k$ so that $1-\frac{C C'}{k} \geq \frac{1}{2}$, we therefore get that 
\begin{align*}
\| \sigma \|_{L_{d+2}^2} \leq 2C C' \sum_i \left( \| \chi_i \sigma \|_{L^2(D_{2r})} + \| \Delta_{j,k} \left( \chi_i \sigma \right)  \|_{L^2_{d}(D_{2r})}  \right).
\end{align*}

Now, $ \| \chi_i \sigma \|_{L^2(D_{2r})} \leq  \| \sigma \|_{L^2}$. Moreover, 
\begin{align*}
\| \Delta_{j,k} \left( \chi_i \sigma \right)  \|_{L^2_{d}(D_{2r})} &\leq \| \Delta_{j,k} \left( \chi_i \sigma \right) - \chi_i \Delta_{j,k} \left( \sigma \right)  \|_{L^2_{d}(D_{2r})} + \|  \chi_i \Delta_{j,k} \left( \sigma \right)  \|_{L^2_{d}(D_{2r})} \\
 &\leq \| \Delta_{j,k} \left( \chi_i \sigma \right) - \chi_i \Delta_{j,k} \left( \sigma \right)  \|_{L^2_{d}(D_{2r})} + C' \|  \Delta_{j,k} \left( \sigma \right)  \|_{L^2_{d}},
\end{align*}
where again the second inequality follows from Leibniz rule. The assertion now follows from the bound
\begin{align}
\label{eq:chibound}
\| \Delta_{j,k} \left( \chi_i \sigma \right) - \chi_i \Delta_{j,k} \left( \sigma \right)  \|_{L^2_{d}(D_{2r})}  &\leq C'' k^{-\frac{1}{2}} \| \sigma \|_{L^2_{d+2}}.
\end{align}
Indeed, once such a bound is established, $\| \sigma \|_{L_{d+2}^2} - C'' k^{-\frac{1}{2}} \| \sigma \|_{L^2_{d+2}} \geq \frac{1}{2} \| \sigma \|_{L_{d+2}^2}$ for all sufficiently large $k$ independently of $\sigma$, from which the bound follows immediately.

The key to establishing Equation \eqref{eq:chibound} is that the $\chi_i$ are pulled back from the base and follows as in the work of Fine \cite{fine04}. 
Indeed, by \cite[Lemma 5.6]{fine04}, there exists a constant $C_{a,l}>0$, depending on $a$ and $l$, such that for an order $l$ tensor $\tau \in C^a ( \otimes_{l} T^* X)$ which is pulled from the base, 
\begin{align}
\label{eq:pulledbacktensor}
\| \tau \|_{C^{a}(X, \omega_k)} \leq C_{a,l} k^{-\frac{l}{2}} \| \tau \|_{C^{a}(B, \omega_B)}. 
\end{align}

Now, as only terms where we differentiate the $\chi_i$ contribute to the quantity $\Delta_{j,k} \left( \chi_i \sigma \right) - \chi_i \Delta_{j,k} \left( \sigma \right)$ that we wish to bound, we get similarly to \cite[Lemma 5.5]{fine04} that there is a $C'''>0$ such that
\begin{align*}
\| \Delta_{j,k} \left( \chi_i \sigma \right) - \chi_i \Delta_{j,k} \left( \sigma \right)  \|_{L^2_{d}(D_{2r})}  \leq C''' \sum_{l=1}^{d+2} \| \nabla^l \chi_i \|_{C^0(X, \omega_k)} \| \sigma \|_{L^2_{d+2}(X, \omega_k)}.
\end{align*}
Since the sum above starts at $l=1$, the bound \eqref{eq:chibound} now follows from this by Equation \eqref{eq:pulledbacktensor}. Note that the bound depends on the $\chi_i$, but this is allowed since these are fixed and do not depend on $\sigma$.
\end{proof}

\begin{proof}[Proof of Proposition \ref{prop:inversebound}]
Lemma \ref{lemma:poincareineq} gives the required bound in Proposition \ref{prop:inversebound} as an operator $L^2 \to L^2 \times \R$. We will now show that the required bound as an operator $L^2_d(\omega_k) \to L^2_{d+2}(\omega_k) \times \R$ follows by combining this bound with the Schauder estimate of Proposition \ref{prop:schaudernew}.

Indeed, let
$$
s = c \Id_E + \tilde s \in \Gamma \left( X , \End_H (\pi^*E, h) \right),
$$
where $\tilde s$ is $\omega_k$-orthogonal to $\Id_E$. Then, putting $\sigma = \scP_{j,k} (s)$ (which is possible since by Lemma \ref{lem:pullbackissimple}, $\Delta_{j,k}$ defines an isomorphism between the $\omega_k$-orthogonal complements of $\Id_E$ in $L^2_{d+2}$ and $L^2_d$) in \eqref{eq:schaudernew}, we obtain
\begin{align*}
\| \scP_{j,k} ( s ) \|_{L^2_{d+2} \times \R} &\leq  \vert c  \vert +  \| \scP_{j,k} ( \tilde s ) \|_{L^2_{d+2}} \\
& \leq C_1 \left( \| c \Id_E \|_{L^2_d} + \| \scP_{j,k} ( \tilde s ) \|_{L^2} + \| \tilde s  \|_{L^2_{d}} \right) \\
&\leq C_1 \left( \| c \Id_E \|_{L^2_{d}} + C_2 k \|  \tilde s  \|_{L^2}+ \| \tilde s  \|_{L^2_{d}} \right) \\
&\leq C k  \| s  \|_{L^2_{d}} ,
\end{align*}
which proves Proposition \ref{prop:inversebound}.
\end{proof}

To obtain the required Lipschitz bound, we will also rely on the following, which follows similar arguments as in e.g. \cite[Lemma 7.1]{fine04} or \cite[Lemma 8.18]{szekelyhidi14book}. Recall that $\scN_{j,k}$ denotes the non-linear part of the HYM operator.
\begin{lemma}\label{lem:lipschitz} There exists $c,C > 0$ such that for all $k \gg 0$, we have that if $s_1, s_2 \in L^2_{d+2} \left( \End (\pi^* E, h) , \omega_k \right)$ satisfy $ \| s_i \|_{L^2_{d+2} } \leq c,$ then $$ \| \scN_{j,k} (s_1) - \scN_{j,k}(s_2) \|_{L^2_{d} } \leq C \left( \| s_1 \|_{L^2_{d+2} } + \| s_2 \|_{L^2_{d+2} } \right) \| s_1 - s_2 \|_{L^2_{d+2} }.$$
\end{lemma}
\begin{proof}
By the Mean Value Theorem, there is a $t \in [0,1]$ such that
$$ 
\| \scN_{j,k} (s_1) - \scN_{j,k}(s_2) \|_{L^2_{d} } \leq \| (D \scN_{j,k})_{\exp(t(s_1-s_2))} \| \| s_1 - s_2 \|_{L^2_{d+2} },
$$
where $(D \scN_{j,k})_{\exp(t(s_1-s_2))}$ is the linearisation of $\scN_{j,k}$ at $\exp(t(s_1-s_2)) \cdot A_{j,k}$. But this is nothing but the difference of the two linear operators at $\exp(t(s_1-s_2)) \cdot A_{j,k}$ and $A_{j,k}$. So the estimate boils down to bounding this difference by a positive multiple of $ \| s_1 \|_{L^2_{d+2} } + \| s_2 \|_{L^2_{d+2} }$, which holds if the $s_i$ have sufficiently small norm.
\end{proof}

The final piece we need to invoke Theorem \ref{thm:quantimpl} is that the pointwise estimates we have established also hold in the appropriate Sobolev spaces.
The result below follows as in \cite[Lemma 5.7]{fine04}.
\begin{lemma}
\label{lem:kexpstable}
Let $A_{j,k}$ be the connection of Proposition \ref{prop:kexpstable}. Then for any $d$, one has
$$
\| \Lambda_{\omega_k} \left( i F_{A_{j,k} } \right) - c_{j,k} \Id_E \|_{C^d} = O \left( k^{-j-1} \right)
$$
and
$$
\| \Lambda_{\omega_k} \left( i F_{A_{j,k} } \right) - c_{j,k} \Id_E \|_{L^{2}_{d}} = O \left( k^{-j-\frac{1}{2}} \right).
$$
\end{lemma}

We can now prove Theorem \ref{thm:stablecase}.
\begin{proof}[Proof of Theorem \ref{thm:stablecase}] We wish to show that $\Psi_{j,k}$ has a root for $k \gg 0$, for a suitable choice of $j$. We first note that Lemma \ref{lem:lipschitz} implies that there is a constant $c>0$ such that for all $r>0$ sufficiently small and $k \gg 0$, we have that $\scN_{j,k}$ is Lipschitz of Lipschitz constant $c r$ on the ball of radius $r$. By Proposition \ref{prop:inversebound}, $ \frac{1}{2 \| \scP_{j,k} \|}$ is bounded below by $C k^{-1}$ for some $C > 0$. Combining these two facts we get that there is a $C' > 0$ such that the radius $\delta_1$ on which $\scN_{j,k}$ is Lipschitz of constant $\frac{1}{2 \| \scP_{j,k} \|}$ satisfies $$ \delta_1 \geq C' k^{-1}.$$ Combining this with the bound for $\| \scP_{j,k} \|$, we get that there is a $C''> 0$ such that the corresponding $\delta $ from Theorem \ref{thm:quantimpl} satisfies $$ \delta \geq C'' k^{-2}.$$ Thus we can apply Theorem \ref{thm:quantimpl} to find a root of $\Psi_{j,k}$ provided $ \| \Psi_{j,k} (0) \| \leq C'' k^{-2}.$ By Lemma \ref{lem:kexpstable}, this holds for all $k \gg 0$  if $j \geq 3$, and thus a root can be found in $L^2_{d+2}$. Elliptic regularity theory implies that the solution in fact is smooth if we choose $d$ large enough, and the result follows. Finally, Theorem \ref{thm:quantimpl} implies that the solutions stay sufficiently close to the approximate solutions, so that the convergence result follows from the convergence of the approximate solutions to the pullback of the initial connection.
\end{proof}

%%%%
%Section 5
%%%%

\section{The semistable case}
\label{sec:semistable}
We now consider the case when $E$ is a strictly semistable bundle on $(B,L)$. It then has a degeneration to a direct sum of stable sheaves $\Gr(E)$, via a Jordan--H\"older filtration:
$$
 0= \cF_0\subset \cF_1\subset \ldots \subset \cF_\ell= E.
$$
In this section, we will assume further the following hypothesis :
\begin{enumerate}
 \item[$(H1)$] the Jordan-H\"older filtration of $E$ is unique;
 \item[$(H2)$] the stable components $\cG_i:=\cF_i/\cF_{i-1}$ of $\Gr(E)$ are pairwise non isomorphic;
 \item[$(H3)$] for all $i$, $\cG_i$ is locally free.
\end{enumerate}
The aim of this section is to prove Theorem \ref{thm:ssthmintro}, that is:
\begin{theorem}
\label{thm:semistablemain}
Suppose $E$ is a slope semistable vector bundle on $(B,L)$ satisfying $(H1)-(H3)$.
Assume that for all $i\in [\![ 1, \ell-1 ]\!]$, $\mu_\infty (\cF_i) < \mu_\infty(E)$.
Then for any $\omega_B \in c_1(L)$ and $\omega_X \in c_1 (H)$ there are connections $A_k$ on $\pi^* E$ which are hermitian Yang--Mills with respect to $\omega_X + k \pi^* \omega_B$ for all $k \gg 0$. Moreover, there is an hermitian Yang--Mills connection $A$ on $\Gr(E)$ with respect to $\om_B$ such that $(A_k)$ converges to $\pi^*A$ in any Sobolev norm.
\end{theorem}
As in Section \ref{sec:stable}, we will focus on producing the solutions $A_k$, and the statement on convergence towards $\pi^*A$ will clearly follow from the construction. In the sequel, we will use the terminology \emph{asymptotically stable with respect to subbundles induced from the Jordan--H\"older filtration} to mean the condition on the slope of the $\cF_i$'s appearing in the above theorem.

We will also prove the following corollary:
\begin{corollary}
 \label{cor:semistablecor}
 Let $E$ be a semistable vector bundle on $(B,L)$ satisfying $(H1)-(H3)$. Assume that for all $i\in[\![1, l-1]\!]$, one has $\mu_\infty(\cF_i)\leq \mu_\infty(E)$ with at least one equality. Then $\pi^*E$ is strictly semistable on $X$ with respect to adiabatic polarisations $L_k$, for $k \gg 0$.
\end{corollary}

Before embarking on the proof, which becomes a fair bit more involved than the stable case, we explain a crucial underlying difference. The key additional input is that we will need to work with a \textit{sequence} of Dolbeault operators for even the first approximate solution. Indeed, we know from \cite[Theorem 6.10.13]{skobayashi87} or \cite[Theorem 2]{jacob14} that $E$ is semistable with respect to $L$ if and only if it admits an almost Hermite--Einstein metric, i.e. for all $\varepsilon > 0$ there is an hermitian metric $\tilde h_{\varepsilon}$ on $E$ such that $$ \vert\vert \Lambda_{\omega_B} \left(iF_{\tilde h_{\varepsilon}} \right) - c \Id_E \vert\vert_\infty < \varepsilon.$$

When $E$ is semistable, gauging back to the initial metric $\tilde h$, the above theorem allows us to choose a sequence $\delb_k$ of Dolbeault operators on $E\to B$ such that the curvatures $(F_{\tilde h,\delb_k})$ of the associated Chern connections satisfy $$ \vert\vert \Lambda_{\omega_B} \left( iF_{\tilde h,\delb_k} \right) - c \Id_E \vert\vert_\infty < C k^{-d},$$ where $d>0$ is a parameter that we are free to choose. After establishing some uniformity in this family, this gives that one can produce a sequence of connections on $\pi^* E \to (X, L_k)$ which are hermitian Yang--Mills to order $k^{-1}$, by combining the above estimate with Lemma \ref{lem:pullbackcv}. Thus one can achieve the first approximate solution just as in the stable case.

The next difference comes in when one wants to perturb to achieve a better order approximate solution. For this, the linearised operator, which is the Laplacian, is used. We now have a sequence of linearised operators, corresponding to the sequence of connections on $E$. The difference between the stable case and the current one is that for the former, the subleading order term was the Laplacian pulled back from $E$, while for the latter, one sees the Laplacian of the graded object $\Gr(E)$ that $E$ degenerates to. This is no longer a simple bundle, and so the Laplacian has a larger cokernel.

The consequence of this is that we can no longer kill off the base error just with the complex gauge transformations we used before. To be able to deal with the remaining error, we need to also incorporate the way the complex structure changes and carefully match the rates of the change in complex structure with that of the change in polarisation on $X$. We will see that there is a crucial sign that needs to be correct, and this provides the link with the intersection numbers coming from the chosen Jordan--H\"older filtration of $E$.

Finally, a word on our hypothesis $(H1)-(H3)$. By $(H3)$, $\Gr(E)$ is locally free, which allows us to see $E$ as a complex deformation of $\Gr(E)$, seen as a {\it vector bundle}. We will then work with connections defined on the common underlying smooth vector bundle. Hypothesis $(H1)$, together with $(H2)$, will force this deformation to have a particularly simple form (see Lemma  \ref{lem:gammaijnonzero}), and will become crucial in the perturbation argument to have a good control on the rates of change in the complex structure. Finally, Hypothesis $(H2)$ gives a better control on the kernel of the Laplacian of the graded object.

\begin{remark}
 \label{rem:differencewithLeung}
 The method of perturbing the Hermite--Einstein structure on $\Gr(E)$ was already used in \cite{leung97}. In Leung's work, one starts with a {\it Gieseker stable} bundle $E$. This bundle is semistable in the sense of Mumford and Takemoto, and, assuming $\Gr(E)$ to be locally free, is a complex deformation of a bundle with an Hermite--Einstein metric $h$. Then, a perturbative argument shows that $h$ can be deformed to an {\it almost Hermite--Einstein} metric on $E$. The argument does not use quantitative estimates, but these are crucial in our case. For these reasons, our proofs need a finer analysis and to deal with additional technical issues.
\end{remark}

Having explained the main new issue that we have to deal with, we now start to prove Theorem \ref{thm:semistablemain}. We begin with explaining the additional structure we need to use on $E$ in Section \ref{sec:basestructurebundle}. Once this is done we follow the strategy of the stable case, by first constructing approximate solutions, then perturbing using the implicit function theorem. In Section \ref{sec:ssapproximate}, we produce these approximate solutions in the special case when the graded object only has two stable components. This simpler case is used to expose how to match the rates of variations of the complex structures and polarisations and where the asymptotic stability condition comes in. We then proceed to the construction of approximate solutions in full generality in Section \ref{sec:approximategeneral}, focusing on the new technical issues compared to Section \ref{sec:ssapproximate}. Finally, in Section \ref{sec:linop} we perform the perturbation argument, and in Section \ref{sec:consequences} we give the proof of Corollary \ref{cor:semistablecor}.

\subsection{Structures on the base}
\label{sec:basestructurebundle}
We refer to \cite[Section 7.2]{skobayashi87} and \cite{BuSchu} for the deformation theory techniques that will be used. We consider the Jordan--H\"older filtration for a strictly semistable bundle $E$ (cf Section \ref{sec:slopestability}):
\begin{equation}
 \label{eq:JHgeneral}
 0= \cF_0\subset \cF_1\subset \ldots \subset \cF_\ell= E,
\end{equation}
assuming $(H1)-(H3)$. In particular, the graded object $\Gr(E)$ is locally free. The bundle $E$ is then obtained by a sequence of extensions of vector bundles:
\begin{equation}
 \label{eq:extensiongeneral}
 0\to \cF_{i-1} \to \cF_{i} \to \cG_{i} \to 0,
\end{equation}
$i=1,\ldots, \ell$, with $\cF_0=0$ and $\cF_\ell=E$ and the graded object the direct sum bundle
$$
\Gr(E) = \oplus_{i=1}^{\ell} \cG_i,
$$
which is the same underlying smooth vector bundle as $E$, but with a different holomorphic structure. Thus the Dolbeault operator $\delb_E$ on $E$ is of the form 
$$
\delb_E=\delb_0+\gamma
$$
where $\delb_0$ is the Dolbeault operator on $\Gr(E)$ and $\gamma \in\Om^{0,1}(B,\Gr(E)^*\otimes \Gr(E))$ can be written
$$
\gamma=\sum_{i<j} \gamma_{ij}
$$
with (possibly vanishing) $\gamma_{ij} \in\Om^{0,1}(B,\cG_j^* \otimes \cG_i)$. The integrability condition $\delb_E^2=0$ imposes the Maurer-Cartan equation
\begin{equation}
 \label{eq:MCequation}
\delb_0\gamma+\gamma\wedge\gamma=0,
\end{equation}
where we will use the notation $\delb_0$ to denote the induced operator $\delb_{0,\End(E)}$ when no confusion should arise. Note that in the matrix block decomposition induced by the splitting $\Gr(E)=\cG_1\oplus\ldots\oplus\cG_\ell$, the representation of $\gamma$ is upper-diagonal:
\begin{equation*}
 \gamma=\left[
 \begin{array}{cccc}
  0 & \gamma_{1,2} & \ldots & \gamma_{1,\ell}   \\
 \vdots & \ddots& \ddots &  \vdots\\
 \vdots & &\ddots & \gamma_{l-1,\ell}  \\
  0 &\ldots & \ldots& 0
 \end{array}
\right].
\end{equation*}
\begin{lemma}
 \label{lem:gammaijnonzero}
 Hypothesis $(H1)$ together with $(H2)$ implies that for all $i\in[\![1, \ell-1 ]\!]$, $\gamma_{i,i+1}$ is non-zero.
\end{lemma}
\begin{proof}
 Assume that $\gamma_{i,i+1}=0$. Then the smooth vector bundle $\cG_1\oplus \ldots\oplus\cG_{i-1}\oplus\cG_{i+1}$ endowed with the restriction of $\delb_E$ is a holomorphic subbundle of $\cF_{i+1}$, denoted $\cF_i'$. By hypothesis $(H2)$, $\cF_i'$ is not isomorphic to $\cF_i$. But then 
 $$
  \cF_0\subset \ldots \subset \cF_{i-1} \subset \cF_i'\subset \cF_{i+1}\subset \ldots \subset  \cF_l= E
 $$
 is a different Jordan-H\"older filtration for $E$, which contradicts $(H1)$.
\end{proof}
\begin{remark}
 In the other direction, if the $\gamma_{i,i+1}$ are all non-zero, the only holomorphic subbundles of $E$ build out of extensions of stable components of $\Gr(E)$ are the $\cF_i$'s, and the Jordan-H\"older filtration of $E$ is unique.
\end{remark}
The group $G:=\Aut(\Gr(E))$ and its Lie algebra $\g$ will play a central role.  Another consequence of $(H2)$ is: 
 \begin{lemma}
  \label{lem:generalcokernel}
  The Lie algebra $\g$ is given by
  $$
  \g=\bigoplus_{1\leq j\leq  \ell} \C\cdot\Id_{j}
  $$
  where $\Id_{j}: \cG_j \to \cG_{j}$ is the identity.
 \end{lemma}
 Note that by Remark \ref{rem:pullbacknonisom}, the Lie algebra of $\Aut(\pi^*\Gr(E))$ is also isomorphic to $\g=\bigoplus_{1\leq j\leq  \ell} \C\cdot\Id_{j}$, where, abusing notation, $\Id_{j}$ now denotes the identity on $\pi^* \cG_j$.
 \begin{proof}
  The space of holomorphic endomorphisms of $\Gr(E)$ satisfies 
 $$
 H^0(X,\End(\Gr(E)))  =  \bigoplus_{1\leq i,j \leq \ell }H^0(X,\Hom(\cG_i,\cG_j)) .
$$
The $\cG_i$ are slope stable on $(B,L)$, of the same slope $\mu_L(\cG_i)=\mu_L(E)$. When $i\neq j$, we therefore get that
 $$
 H^0(X,\Hom(\cG_i,\cG_j))=H^0(X,\Hom(\cG_j,\cG_i))=0
 $$ 
 since $\cG_i$ is not isomorphic to $\cG_j$ by $(H2)$. On the other hand, when $i=j$ then
 $$
 H^0(X,\End(\cG_i))\simeq  \C,
 $$  
 see e.g. \cite[Proposition 5.7.11, Corollary 5.7.14]{skobayashi87}. 
 \end{proof}

By polystability, there is a product Hermite--Einstein metric $\tilde h=\tilde h_1\oplus \ldots \oplus \tilde h_\ell$ on $\Gr(E)$. We denote by $A_0$ the associated hermitian Yang--Mills connection, so that 
$$
d_{A_0}=\partial_0+\delb_0
$$
with curvature form satisfying
$$
\Lambda_{\om_B}iF_{A_0} = c_0 \cdot \Id.
$$
The following classical result will have interesting consequences. We refer to \cite[Lemma 4.1 and Corollary 4.2]{BuSchu} for a short proof relying on the K\"ahler-type identities \cite[Section 3.2]{skobayashi87}:
\begin{equation}
 \label{eq:Nakanotype}
\begin{array}{ccc}
[\Lambda_{\om_B},\partial_0] & = & i\delb_0^*,\\

[\Lambda_{\om_B},\delb_0] & = & -i\partial_0^*,
\end{array}
\end{equation}
where the adjoint operator $*$ is induced by $\tilde h$ on $\Gr(E)$ and $\om_B$ on $B$. 
\begin{proposition}
\label{prop:holomorphicisparallel}
 Any $\sigma\in \Gamma(B,\End(\Gr(E)))$ that is holomorphic is parallel. That is, the equation $\delb_0 \sigma=0$ implies $d_{A_0}\sigma =0$.
\end{proposition}
We will gauge fix $\gamma$ by further imposing that 
\begin{equation}
\label{eq:gaugefixgamma}
\delb_0^*\gamma=0.
\end{equation}
Note that this gauge fixing will only be used in the first step in the upcoming induction process to produce approximate solutions, and not at higher order in the $k$-expansions. We can now describe the automorphism group of $E$.
\begin{lemma}
 \label{lem:H1H2implysimple}
 Under assumptions $(H1)$ and $(H2)$, $E$ is simple.
\end{lemma}
\begin{proof}
 Let $\psi\in\End(E)$ and consider the composition $\cF_{\ell-1}\to E \overset{\psi}{\to} E \to \cG_{\ell}$. If it is not zero, as $\cG_\ell$ is stable, from \cite[Proposition 1.2.7]{HuLe}, this composition must be a surjective morphism $\cF_{\ell-1} \to \cG_\ell$. But this is excluded by $(H2)$, so this composition vanishes. Hence $\psi$  induces morphisms $\cF_{\ell-1}\to \cF_{\ell-1}$ and $\cG_\ell\to \cG_\ell$. By induction, as $\cF_{\ell-1}$ satisfies $(H1)$ and $(H2)$, the first morphism is a constant multiple of the identity. So is the latter by simplicity of $\cG_\ell$. Hence, $\psi$ reads
 $$
 \psi = a\Id_{\cF_{\ell-1}}+b\Id_{\cG_\ell} + \sigma,
 $$
 for $(a,b)\in\C^2$ and $\sigma\in\Gamma(\cG_\ell^*\otimes\cF_{\ell-1})$. Then, using 
 $$\delb_0\psi +[\gamma,\psi]=0,$$
 we obtain
 $$
 \delb_0\sigma +(a-b)\sum_{i<\ell}\gamma_{i\ell}=0.
 $$
 Applying $\delb_0^*$ to this last equality, using (\ref{eq:gaugefixgamma}), we have $\sigma\in \Hom(\cG_\ell,\cF_{\ell-1})$, which by $(H2)$ again implies $\sigma=0$. As by Lemma \ref{lem:gammaijnonzero} the term $\gamma_{\ell-1\ell}$ is not zero, we conclude $a-b=0$, and $\psi=a\Id_E$.
\end{proof}

From Proposition \ref{prop:holomorphicisparallel}, together with the identities \eqref{eq:Nakanotype}, we deduce that there is a natural action of the subgroup of gauge transformations
\begin{equation}
 \label{eq:definitiongroupG}
  G=\Aut(\cG_1)\times \ldots\times\Aut(\cG_\ell)\subset \scG^\C(\Gr(E))
\end{equation}
on elements in
$$
\mathfrak{Def}(\Gr(E)):=\lbrace \beta\in\Om^{0,1}(B,\Gr(E)^*\otimes \Gr(E)),\: \delb_0\beta + \beta\wedge\beta=0,\: \delb_0^*\beta=0 \rbrace
$$
parametrising small complex deformations of $\Gr(E)$. On $\gamma$, the action of an element $g\in G\simeq (\C^*)^\ell$ of the form $g=g_1 \Id_{\cG_1}\times \ldots \times g_\ell \Id_{\cG_\ell}$ is given by
\begin{equation}
 \label{eq:Autactionongamma}
 g^*\gamma= g^{-1}\cdot\gamma \cdot g = \sum_{i<j} g_i^{-1}g_j\gamma_{ij}.
\end{equation}
Then, $\gamma$ is gauge-conjugated to all elements of this form. In order to parameterise a family of Dolbeault operators from $\delb_0$ to $\delb_E$, we can make a change of variables. For any $\lambda=(\lambda_1,\ldots,\lambda_{\ell-1})\in (\C^*)^{\ell-1}$ we can find $g_\lambda\in G$ such that for all $i=1\ldots \ell-1$, $\lambda_i=(g_\lambda)_i^{-1}(g_\lambda)_{i+1}$. Setting
$$
\gamma_\lambda:=g_\lambda^{-1} \cdot \gamma \cdot g_{\lambda},
$$
the family of Dolbeault operators
$$
\delb_\lambda:=\delb_0+\gamma_\lambda
$$
can be extended across $\lambda=0$ and gives a complex family of holomorphic vector bundles $E_\lambda$ with $E_\lambda$ isomorphic to $E$ for $(\lambda_i)\neq 0$ and isomorphic to $\Gr(E)$ for $\lambda=0$. Thus we see that $E$ can be obtained as a complex deformation of $\Gr(E)$.

For {\it any} such family of Dolbeault operators $\lambda \mapsto \delb_\lambda$ parametrising a complex deformation from $\Gr(E)$ to $E$, we have a path of Chern connections $A_\lambda$ associated to the structures $(\tilde h,\delb_\lambda)$. We will be interested in the curvature $F_{A_\lambda}$ and its variations. Consider the unique family of (skew-hermitian) connection $1$-forms $a_\lambda$ such that $A_\lambda=A_0+a_\lambda$, 
given explicitly by
$$
 a_\lambda =  \gamma_\lambda - \gamma_\lambda^*.
$$
Recall, e.g. from \cite[Section 4]{huybrechts05}, that
\begin{equation}
 \label{eq:curvatureintermsofconnection}
 F_{A_\lambda}=F_0+d_{A_0} a_\lambda + a_\lambda\wedge a_\lambda.
\end{equation}
As in Section \ref{sec:stable}, we will need to control the linearisation of the operator $\Phi_\lambda$ given by $f\mapsto \Lambda_{\om_B} iF_{A_\lambda^f}$, and in particular its kernel. Thus we introduce a compact form $K$ of $G$:
\begin{equation}
 \label{eq:introductionK}
 K:=\Aut(\cG_1,\tilde h_1)\times \ldots\times\Aut(\cG_\ell,\tilde h_\ell)\subset \scG^\C(\Gr(E)),
\end{equation}
where $\Aut(\cG_i,\tilde h_i)$ stands for the group of automorphisms of the holomorphic vector bundle $\cG_i$ that preserve $\tilde h_i$. We also introduce its Lie algebra $$\k:=\Lie(K).$$
The hermitian endomorphisms in 
$$
i\k=\bigoplus_{i=1}^\ell \R\cdot \Id_{\cG_i}\subset \Gamma(B, \End_H(E,\tilde h) )
$$
will appear in the kernel of $\Phi_\lambda$ at $\lambda=0$. We will see that this space provides the potential obstruction to solving the HYM equation in the semistable case. We thus introduce the $L^2$ projection 
\begin{equation}
 \label{eq:L2projection}
 \begin{array}{cccc}
  \Pi_{i\k} :& \Gamma(B, \End_H (E, \tilde h) ) & \to & i\k\\
             &      s_B                  & \mapsto &\displaystyle \frac{1}{\Vol(B)}\sum_{i=1}^\ell \frac{1}{\rank(\cG_i)} ( \int_B \trace_{\cG_i}(s_B)\:\om_B^n ) \: \Id_{\cG_i}
 \end{array}
\end{equation}
and the induced orthogonal decomposition:
\begin{equation*}
 \label{eq:L2splittingonB}
 \Gamma(B,\End_H(E,\tilde h)) =i\k \oplus \Gamma_{i\k^\perp}(B,\End_H(E,\tilde h))
\end{equation*}
with respect to the pairing 
$$
(s_1,s_2)\mapsto \int_B \trace(s_1\cdot s_2)\,\om_B^n.
$$

We now gather in the next lemma some straightforward results that will be used in the following sections to control the projection of $\Lambda_{\om_B}iF_{A_\lambda}$ onto $i\k$ and to remove the errors orthogonal to $i\k$ (last item follows from ellipticity and self-adjointness).
\begin{lemma}
 \label{lem:tracesandda}
 We have the following:
 \begin{enumerate}
  \item[(i)] The term $d_{A_0}a_\lambda$ is off-diagonal.
\item[(ii)] The trace of $a_\lambda\wedge a_\lambda$ vanishes.
\item[(iii)] For any non-zero component $\gamma_{jj'}$ of $\gamma$, the following constant is positive:
\begin{equation}
\label{eq:positiveC}
-\int_B \trace_{\cG_j}(\Lambda_{\om_B}i\gamma_{jj'}\wedge\gamma_{jj'}^*)\:\om_B^n >0.
\end{equation}
\item[(iv)] The term $\Lambda_{\om_B}\left(d_{A_0}a_\lambda\right)$ vanishes:
\begin{equation}
 \label{eq:vanishingLambdada}
 \Lambda_{\om_B}\left(d_{A_0}a_\lambda\right)=0.
\end{equation}
\item[(v)] The following operator is invertible : 
$$
\Delta_{\om_B,0} :\Gamma_{i\k^\perp}(B,\End_H(E,\tilde h)) \to \Gamma_{i\k^\perp}(B,\End_H(E,\tilde h))
$$
where $\Delta_{\om_B,0} := i \Lambda_{\om_B} \big(  \partial_{A_0} \bar \partial_{A_0}  - \bar \partial_{A_0}  \partial_{A_0} \big)$.
 \end{enumerate}
\end{lemma}

We return now to the holomorphic submersion with connected fibres $\pi:(X,H)\to (B,L)$.
Since (\ref{eq:JHgeneral}) is a Jordan--H\"older filtration, $\Gr(E)$ is polystable.
By Corollary \ref{cor:stablepullback}, each of the pullbacks $\pi^*\cG_i$ of the stable components of the graded object is stable for the polarisation $L_k=H\otimes L^{\otimes k}$, for $k\gg 0$. However, the expansions of the slopes $k^{-n}\mu_k(\pi^*\cG_i)=\sum_{j=1}^n \nu_j(\cG_i) k^{-j}$ (recall Section \ref{sec:adiabaticslopes}) may disagree, and $\pi^*\Gr(E)$ may be unstable. In particular, the results from \cite{BuSchu} describing stable deformations of polystable bundles do not apply, and we need a refined argument to understand stability of $\pi^*E$. The following will play a crucial role in our arguments:
\begin{definition}
 \label{def:discrepancy}
 For $\cF$ and $\cE$ two torsion-free sheaves on $X$, the {\it order of discrepancy} of the adiabatic slopes of $\cF$ and $\cE$ is the leading order of the expansion $\mu_k(\cE)-\mu_k(\cF)$. Given $\cF_i\subset E$, the order of discrepancy of $\cF_i$  will refer to the order of discrepancy of the adiabatic slopes of $\pi^*E$ and $\pi^*\cF_i$. 

The \emph{maximal order of discrepancy} of $E$ is the maximal order of discrepancy among $\cF_i$ for $1\leq i\leq \ell-1$. When $E$ is itself stable on $X$, we use the convention that the maximal order of discrepancy is $1$. 
\end{definition}
Using the forms $\om_X$ and $\pi^*\om_B$, we define a pairing on the space of sections of $\Gamma(X, \End_H (\pi^*E, h))$, where $h=\pi^*\tilde h$ :
$$
(s_1,s_2)\mapsto \int_X \trace(s_1\cdot s_2) \,(\pi^*\om_B)^n\wedge\om_X^m
$$
and we  extend the $L^2$ projection (\ref{eq:L2projection}) as well as the orthogonal decomposition to $\Gamma(X, \End_H (\pi^*E, h))$:
\begin{equation}
 \label{eq:L2splitting}
 \Gamma(X, \End_H (\pi^*E, h))=i\k \oplus \pi^*\Gamma_{i\k^\perp}(B,\End_H(E,\tilde h))  \oplus \Gamma_0(X,\End_H(\pi^*E,h)),
\end{equation}
where we recall that $\Gamma_0(X,\End_H(\pi^*E,h))$ is the space of sections of average zero on each fibre of $\pi$.
Using this $L^2$ inner product on sections, we will consider projections onto various components of $\g$. For $\psi\in\g$, we will denote by  $\Pi_{\langle \psi\rangle}$ the projection onto the subspace spanned by $\psi$. 
By Chern--Weil theory, the constants that one obtains by expanding the $\Pi_{\langle \Id_{\cG_i} \rangle}\Lambda_k(iF_{A_0})$ are (up to multiplicative constants) the $\nu_j(\cG_i)$'s. Thus, these terms give the topological defect of $(\pi^*\Gr(E),h)$ from being Hermite--Einstein. 
\begin{remark}
\label{rem:dAnoncontribution}
 A corollary of Lemma \ref{lem:tracesandda} is that the terms involving $d_{A_0}a_\lambda$ (e.g. $\Lambda_{\omega_B} d_{A_0}a_\lambda$ and $\langle d_{A_0}a_\lambda, \om_\cH \rangle$) will not contribute when computing the projection on $i\mathfrak{k}$ of the perturbed curvature
$$
\Lambda_k i F_{h,f_k^*\delb_\lambda}.
$$
On the other hand, the terms $\Lambda_{\om_B}i a_\lambda\wedge a_\lambda$ will be used to kill off the discrepancies between the $\nu_j(\cG_i)$'s. This will be done by using the $e^{i\k}$-action on $\gamma$. Note that by \cite[Corollary 5.3]{BuSchu}, the $L^2$ projection of  $\Lambda_{\om_B} a_\lambda\wedge a_\lambda$ on $\mathfrak{k}$ can be interpreted as a moment map for the action of $K$ on $\mathfrak{Def}(\Gr(E))$. Thus, we will be essentially looking for zeros of (a $k$-expansion of) this moment map in a $G$-orbit. One should be careful though, the discrepancy orders that are needed to understand the stability of a deformation of $\pi^*\Gr(E)$ {\it depend} on the corresponding element in $\mathfrak{Def}(\Gr(E))$. Thus, one cannot use this moment map interpretation at once on the whole space of small deformations $\mathfrak{Def}(\Gr(E))$.
\end{remark}

%%%%%%%%%%%%%%%%%%%
%                    2-COMPONENTS                %
%%%%%%%%%%%%%%%%%%%
\subsection{Approximate solutions: two components case.}
\label{sec:ssapproximate}

The main goal of this section is to construct the approximate solutions to the HYM equation on $X$, to any desired order. As in Section \ref{sec:stableapprox}, we use here (and in the next section) power series expansions in $k^{-1}$. An expression $\sigma(k)=O(k^{-j})$ for $\sigma$ a section of a vector bundle is to be interpreted as holding pointwise. We will also use such expressions for operators, in which case they should be interpreted as holding after the operator acts on a section. Convergence issues will be addressed in Section \ref{sec:linop}.

We will assume here that the graded object associated to $E$ has only two components. We single out this case as its presentation is simpler, and its proof already gathers some of the main ingredients needed in the general construction. Compared to the stable setting of Section \ref{sec:stableapprox}, the added complication is that we will have to relate the rates of convergence of the two parameters $k^{-1}$ and $\lambda$ involved in the semistable case. 

From now on the Jordan--H\"older filtration takes the simple form 
$$ 
0 \subset \cF_1 \subset E,
$$
and the graded object is
$$
\Gr(E)=\cG_1 \oplus \cG_2,
$$
where $\cG_1 = \cF_1$ and $\cG_2 = \frac{E}{\cF_1}.$ Denoting by $\delb_{\cG_i}$ the Dolbeault operators on $\cG_i$, the Dolbeault operator on $E$ is explicitly given by
$$
\delb_E = \delb_0+\gamma= \delb_{\cG_1} + \delb_{\cG_2} + \gamma
$$
for some $\gamma \in\Om^{0,1}(B,\cG_2^* \otimes \cG_1)$. Note that $\gamma$ satisfies by the Maurer--Cartan equation $\delb_0\gamma=0$ and $\delb_0^*\gamma=0$, since when $\ell=2$, $\gamma\wedge \gamma=0$.

For any $t\in\R_{>0}$, $\delb_E$ can be gauge conjugated to
\begin{equation*}
 \label{eq:Dolbeault on E t}
 \delb_t:=g_t^{-1}\cdot \delb_E \cdot g_t=\left[
 \begin{array}{cc}
  \delb_{\cG_1} & t\gamma \\
  0          &   \delb_{\cG_2}
 \end{array}
\right]
\end{equation*}
via the element $g_t$ in $G=\Aut(\cG_1)\times\Aut(\cG_2)\subset \scG^\C(E)$ given by
$$
 g_t=\left[
 \begin{array}{cc}
  1 & 0 \\
  0          &  t
 \end{array}
\right].
$$
As in Section \ref{sec:basestructurebundle}, we will let $\tilde h$ be the product Hermite--Einstein metric on $\Gr(E)$, while we set $A_t=A_0+a_t$ which is the Chern connection associated to the structure $(\tilde h,\delb_t)$. In matrix representation,
$$
a_t = t
 \left[
 \begin{array}{cc}
  0 & \gamma \\
  -\gamma^*          &  0
 \end{array}
\right]
$$
and we let
\begin{equation}
\label{eq:matrixaa}
a\wedge a=- \left[
 \begin{array}{cc}
  \gamma\wedge \gamma^* & 0 \\
  0        &  \gamma^*\wedge\gamma
 \end{array}
\right],
\end{equation}
such that $a_t \wedge a_t = t^2 a \wedge a.$

The following lemma shows that when $\ell=2$, hypothesis $(H1)$ and $(H2)$ can be replaced by simplicity.
\begin{lemma}
\label{lem:Hlemma}
Assume $E$ simple and that $\mu_\infty(\cF_1)<\mu_\infty(E)$. Then
\begin{itemize}
 \item[(i)] The term $\gamma$ is non-zero.
 \item[(ii)] The stable components $\cG_1$ and $\cG_2$ are not isomorphic.
\end{itemize}
\end{lemma}
 \begin{proof}
  If $\gamma=0$, then $E=\Gr(E)$ which contradicts simplicity of $E$. For $(ii)$, if $\cG_1\simeq\cG_2$, we have $\mu_\infty(\cG_1)=\mu_\infty(\cG_2)=\mu_\infty(\cF_1)$, which leads to a contradiction by Corollary \ref{cor:subadditivity of slopes}.
 \end{proof}
We now proceed to the construction of approximate solutions. We will assume that the discrepancy order of $\cG_1$ is $q$, so that $\nu_j (\cG_1 ) = \nu_j (\cG_2 ) = \nu_j (E)$ for $ j< q$, and $\nu_q (\cG_1) < \nu_q (E).$ 
We first show that we can construct an approximate solution up to order $q-\frac{1}{2}$, simply by working on the two pieces $\cG_1$ and $\cG_2$ separately. The fractional order appearing from now on comes about because we need to choose the speed rate $t$ of the deformation of the complex structure of $E$  to be of the form $t=\lambda \cdot k^{-\frac{q-1}{2}}$ in order to affect the expansion of the curvature at order $k^{-q}$. This involves a fractional power if $q$ is even. 

The existence of an approximate solution up to order $q-\frac{1}{2}$ follows a very similar strategy as in the stable case. For this we need to understand the leading order term of the expansion of the linearised operator.
\begin{proposition}
 \label{prop:initialexpansion}
Let $\Delta_k $ denote the Laplace operator associated to $\omega_k$ and the Chern connection of $(h, \pi^*\delb_{\lambda k^{-\frac{q-1}{2}}}),$ for some real constant $\lambda >0$ and integer $q  \geq 2$. Then
\begin{align}
\label{eqn:laplaceexpansion}
\Delta_k (s) = \Delta_{\scV} (s) +   k^{-1} \Delta_{\cH, 0} (s) + O \left( k^{-\frac{3}{2}} \right),
\end{align}
where $\Delta_{\cH,0}$ is the horizontal Laplace operator associated to the complex structure $\delb_0$ for $E$ on $B$. The same expansion also holds at a Chern connection on $\pi^* E \to X$ coming from a complex structure $f_k \cdot \pi^*\delb_{\lambda k^{-\frac{q-1}{2}}}$ provided $f_k = \Id_E + s_k$ for some $s_k \in \Gamma \left( X, \End  \pi^* E \right)$  whose base component $s_{B,k}$ satisfies $s_{B,k} = O(k^{-\frac{1}{2}})$ and whose vertical component $s_{F, k}$ satisfies $s_{F,k} = O(k^{-\frac{3}{2}})$.
\end{proposition}
\begin{proof}
Let $A_k$ denote the Chern connection of $(h, \pi^*\delb_{\lambda k^{-\frac{q-1}{2}}})$. Then, the induced connection on $\End(E)$ reads 
\begin{equation}
 \label{eq:inducedconnectiongamma}
A_{k,\End(E)} = A_{0,\End(E)}+ \lambda k^{-\frac{q-1}{2}} \left[ \gamma - \gamma^*,\cdot \right],
\end{equation}
where $A_0$ is the Chern connection of $(h, \pi^*\delb_0)$, $[\cdot,\cdot]$ denotes the Lie bracket on sections of $\End(E)$ and where we omit pullbacks. Using Lemma \ref{lem:expansion contraction}, and in particular the fact that the contraction by $\om_k$ of terms pulled back from $B$ is $\cO(k^{-1})$, we obtain
\begin{equation*}
\begin{array}{ccc}
 \Delta_k & = & i\Lambda_k\left( \partial_{A_{0,\End E}} \bar \partial_{A_{0,\End E}}  - \bar \partial_{A_{0,\End E}}  \partial_{A_{0,\End E}}\right) \\
          & + &i\lambda k^{-\frac{q+1}{2}} \:\Lambda_{\om_B} \left(\delb_{A_{0,\End E}} ([\gamma^*,\cdot ])-[\gamma^*,\delb_{A_{0,\End E}}  ]\right)\\
          & + & i\lambda k^{-\frac{q+1}{2}} \: \Lambda_{\om_B} \left(\partial_{A_{0,\End E}}([\gamma,\cdot])-[\gamma,\partial_{A_{0,\End E}} ] \right) + \cO(k^{-\min(q,\frac{q+3}{2})}).
 \end{array}
\end{equation*}
Thus, as $q\geq 2$,
$$
\Delta_k  = \Delta_{0,k} + O \left( k^{-\frac{3}{2}} \right),
$$
where $\Delta_{0,k}$ is the Laplacian of $A_0$ with respect to $\omega_k$. By Proposition \ref{prop:expansion linear term}, $\Delta_{0,k}$, and then $\Delta_k$, satisfy \eqref{eqn:laplaceexpansion}. 

The statement for the perturbed complex structures $f_k \cdot \pi^*\delb_{\lambda k^{-\frac{q-1}{2}}}$, follows by taking an expansion of $\partial_{A^{f_k}_{k, \End E}} \bar \partial_{A^{f_k}_{k, \End E}}  - \bar \partial_{A^{f_k}_{k, \End E}}  \partial_{A^{f_k}_{k,\End E}}$, as in \cite[Lemma 5.3]{dervansektnan19}. With the bounds on $s_{B,k}$ and $s_{F,k}$ given, the base component of 
$$
\left( \partial_{A_{k, \End E}} \bar \partial_{A_{k, \End E}}  - \bar \partial_{A_{k, \End E}}  \partial_{A_{k,\End E}} \right) - \left( \partial_{A^{f_k}_{k, \End E}} \bar \partial_{A^{f_k}_{k, \End E}}  - \bar \partial_{A^{f_k}_{k, \End E}}  \partial_{A^{f_k}_{k,\End E}} \right)
$$
 is $O(k^{-\frac{1}{2}})$ and the vertical component is $O(k^{-\frac{3}{2}})$. Using the asymptotics of $\Lambda_{\omega_k}$, it follows that the Laplacian operators differ by a term that is $O(k^{-\frac{3}{2}})$.
\end{proof}

With this in place, we can construct the approximate solutions at low order. Note that in the case when $q$ is even, the speed rate $t= \lambda \cdot k^{-\frac{q-1}{2}}$ involves a fractional power of $k$. The expansions of the curvature will then involve powers $k^{-j}$ with $j \in \frac{1}{2} \mathbb{Z}_{>0}$. We therefore use $\frac{1}{2} \mathbb{Z}_{>0}$ instead of $ \mathbb{Z}_{>0}$ for the parameter $j$ below.
\begin{lemma} 
\label{lem:loworder}
Let $E$ be a simple semistable bundle with graded object $\cG_1 \oplus \cG_2$ such that $\cG_1$ has discrepancy order $q$. Pick $t= \lambda \cdot k^{-\frac{q-1}{2}},$ for some real constant $\lambda >0$. Then there exist
\begin{itemize} 
\item gauge transformations $f_{q-\frac{1}{2},k}$,
\item[]
\item constants $c_0 , \ldots , c_{q-\frac{1}{2}}$ ,
\end{itemize}
independent of $\lambda,$ such that if we let $A_{q-\frac{1}{2},k} = A_t^{f_{q-\frac{1}{2},k}}$, then for all $k \gg 0$, $$ \Lambda_{\omega_k} \left( i F_{A_{q-\frac{1}{2},k} } \right) = c_{q-\frac{1}{2},k} \Id_E + O \left( k^{-q} \right),$$
where $c_{q-\frac{1}{2}, k} = \sum_{i=0}^{q-\frac{1}{2}} c_i k^{-i}.$ Moreover, $f_{q-\frac{1}{2},k} =\Id_E+ s_{q-\frac{1}{2},k}$ for some $s_{q-\frac{1}{2},k} \in \Gamma \left( X, \End  \pi^* E\right)$ such that 
\begin{itemize} 
\item its base component $s_{B,q-\frac{1}{2},k}$ is $O(k^{-\frac{1}{2}})$ and is diagonal up to, and including, order $k^{-\frac{q-1}{2}}$;
\item its vertical component $s_{F,q-\frac{1}{2},k}$ is $O(k^{-\frac{3}{2}})$ and is diagonal up to, and including, order $k^{-\frac{q+1}{2}}$.
\end{itemize}
\end{lemma}
\begin{proof} By the choice of the speed rate $t$, we have from Equation \eqref{eq:curvatureintermsofconnection} that
$$
 \Lambda_{\omega_k}(i F_{A_t})= \Lambda_{\omega_k}( i F_0)+ \lambda k^{-\frac{q-1}{2}} \Lambda_{\omega_k} (i d_{A_0} a ) + \lambda^2 k^{1-q}  \Lambda_{\omega_k } (i a \wedge a ).
$$
Since $a \wedge a$ is a pulled back term, $\lambda^2 k^{1-q}  \Lambda_{\omega_k } (i a \wedge a )$ is $O(k^{-q})$ and so will not enter the argument at this stage, as we are only interested in terms up to order $k^{-q+\frac{1}{2}}$.

By considering each piece $\cG_1$ and $\cG_2$, both of which are stable, separately, it follows  from the stable case, Proposition \ref{prop:kexpstable}, that we can find
\begin{itemize} 
\item gauge transformations $\hat f_{q-1,k}^1\in\cG^\C(\pi^*\cG_1)$,
\item gauge transformations $\hat f_{q-1,k}^2\in\cG^\C(\pi^*\cG_2)$,
\item constants $c^1_{0}, \ldots, c^1_{q-1} , c^2_{0}, \ldots, c^2_{q-1}$ ,
\end{itemize}
such that  if we let $\hat f_{q-1,k} =\hat f_{q-1,k}^1\times\hat f_{q-1,k}^2 $
and put $\hat A_{q-1,k} = A_0^{\hat f_{q-1,k}}$, then for all $k \gg 0$, 
$$ 
\Lambda_{\omega_k} \left( i F_{\hat A_{q-1,k} } \right) = c^1_{q-1,k} \Id_{\cG_1} + c^2_{q-1,k} \Id_{\cG_2} + O \left( k^{-q} \right),
$$
with $c^j_{q-1,k} = \sum_{i=0}^{q-1} c^j_i k^{-i}$. Note that there are no fractional coefficients here, regardless of the value of $q$. Note also that 
$\hat f_{q-1,k}$ is a diagonal gauge transformation, so $\hat A_{q-1,k}$ is a product connection.

We claim that $c^1_{q-1,k} = c^2_{q-1,k}.$ This follows because we have used a product connection on $\Gr (E)$ above, and so these constants are determined by the order $k^{n- q}$ expansion of the projection
$$
\frac{1}{\rank (\cG_i) } \int_X \Lambda_{\omega_k} \left( i \trace_{\cG_i} F_{\hat A_{q-1,k} } \right) \omega_k^{n+m} .
$$
Since this is a constant multiple of $ \frac{ c_1 ( \cG_i ) c_1(L_k)^{n+m-1}}{\rank (\cG_i)}$, the coefficient to order $k^{n-1-j}$ is a constant multiple of $\nu_j (\cG_i).$ Since these are equal to $\nu_{j} (E)$ for all $j< q$ for both $i=1$ and $2$ by Lemma \ref{lem:additivity of nui}, we get that  $c^1_{q-1,k} = c^2_{q-1,k}.$ 

Next, we consider the actual connection $A_t=A_{\lambda k^{-\frac{q-1}{2}}}$ instead of $A_0$. We set $\hat f_{q-1,k}  A_t:=A_t^{ \hat f_{q-1,k}} $ and compute 
$$
\begin{array}{ccc}
 \Lambda_{\omega_k}(i F_{\hat f_{q-1,k}  A_t}) & = &  \Lambda_{\omega_k}( i F_{\hat A_{q-1,k}})\\
 & +& \lambda k^{-\frac{q-1}{2}} \Lambda_{\omega_k} (i d_{\hat A_{q-1,k} } (\hat f_{q-1,k}\cdot a ))\\
  & +&  \lambda^2 k^{1-q}  \Lambda_{\omega_k } (i (\hat f_{q-1,k}\cdot a) \wedge (\hat f_{q-1,k}\cdot a )).
\end{array}
$$
By the previous discussion, the first term on the right hand side of this equality satisfies
$$
\Lambda_{\omega_k}( i F_{\hat A_{q-1,k}})=c_{q-1,k}  \Id_E + O(k^{-q}),
$$
while the last term is $O(k^{-q})$. As $\hat f_{q-1,k}$ was made purely from diagonal elements, and as $a$ is off-diagonal,  $(\hat f_{q-1,k}\cdot a)$ is off-diagonal. Since $\hat A_{q-1,k}$ is a product connection, we deduce that $\Lambda_{\omega_k} (i d_{\hat A_{q-1,k} } (\hat f_{q-1,k}\cdot a ))$ and the second term are off-diagonal. The lower order contribution in the $k$-expansion of this second term will be given by the first non-zero term in 
$$
\Lambda_{\omega_k} (i d_{A_0} a ) = k^{-1} \Lambda_{\omega_B} (i d_{A_0} a) + \ldots.
$$
By $(iv)$ of Lemma \ref{lem:tracesandda}, $\Lambda_{\omega_B} (i d_{A_0} a)$ vanishes. Thus, the lower order contribution in the $k$-expansion of the second term is basic, meaning is pulled back from $B$, and of order $ k^{-\frac{q+3}{2}}$. That is, we have an expansion
$$
 \Lambda_{\omega_k}(i F_{\hat f_{q-1,k}  A_t})= c_{q-1,k}  \Id_E + k^{-\frac{q+3}{2}} \sum_{r=0}^{\frac{q}{2} - 1} k^{-r} \tilde \sigma_r + O(k^{-q}),
$$
where all the terms $\tilde \sigma_r$ are off-diagonal, and where $\tilde \sigma_0$ is basic. We now want to remove these errors using the linearisation, ensuring that we are not changing the constants $c_{q-1,k}$. Recall that as $\cG_1\neq \cG_2$, the kernel of our linear operator at first order $\Delta_{\cH,0}$ is $i\k$, with 
$$
i\k=\langle \Id_{\cG_1} \rangle \oplus \langle \Id_{\cG_2} \rangle.
$$
In particular, every off-diagonal basic term is in the image of $\Delta_{\cH,0}$, while off-diagonal vertical elements are in the image of $\Delta_{\cV}$.

We first consider the term $k^{-\frac{q+3}{2}} \tilde \sigma_0$. Since this is off-diagonal and basic, there is an off-diagonal $\tilde s_B^0$ such that $\Delta_{\cH,0} (\tilde s_B^0) = - \tilde \sigma_0$. Note that, by Proposition \ref{prop:kexpstable},  $\hat f_{q-1,k}=\Id_E+\hat s_{B,q-1,k}+\hat s_{F,q-1,k}$ with $\hat s_{B,q-1,k}=O(k^{-1})$ and $\hat s_{F,q-1,k}=O(k^{-2})$. Thus the expansion \eqref{eqn:laplaceexpansion} applies to the Laplace operator of $(h,\hat f_{q-1, k} \pi^*\delb_{\lambda k^{-\frac{q-1}{2}}} )$. By Proposition \ref{prop:initialexpansion}, we then have that 
$$
 \Lambda_{\omega_k}(i F_{\exp (k^{- \frac{q+1}{2}} \tilde s_B^0 ) (\hat f_{q-1,k}  A_t)})= c_{q-1,k}  \Id_E + k^{-\frac{q+3}{2}} \sum_{r=\frac{1}{2}}^{\frac{q}{2} - 1} k^{-r} \tilde \sigma'_r + O(k^{-q}),
$$
for some new error terms $\tilde \sigma'_r $.

We now claim that the $\tilde \sigma'_r$ remain off-diagonal. This implies in particular that the projection to $i \k$ of the terms in the curvature of $\exp (k^{- \frac{q+1}{2}} \tilde s_B^0 ) (\hat f_{q-1,k}  A_t)$  up to $O(k^{-q})$ remains $ c_{q-1,k} \Id_E $. To see this, we consider the full change in the curvature
\begin{align*}
F_{A^f} =& f \circ A^{1,0} \circ A^{1,0} \circ f^{-1} + f^{-1} \circ A^{0,1} \circ A^{0,1} \circ f \\
&+ f \circ A^{1,0} \circ f^{-1} \circ f^{-1} \circ A^{0,1} \circ f + f^{-1} \circ A^{0,1} \circ f \circ f \circ A^{1,0} \circ f^{-1},
\end{align*}
when $f$ is an hermitian automorphism, see \cite[Theorem 7.4.20]{skobayashi87}, that we apply to $f=\exp (k^{- \frac{q+1}{2}} \tilde s_B^0)$ and $A=\hat f_{q-1,k}  A_t$. Since $\hat f_{q-1,k}  A_t$ is diagonal to order $ k^{-\frac{q-1}{2}}$, and $f=\exp (k^{- \frac{q+1}{2}} \tilde s_B^0)$ is $\Id_E$ plus an off-diagonal term that occurs at order $k^{- \frac{q+1}{2}}$, we see that 
$$
F_{\exp (k^{- \frac{q+1}{2}} \tilde s_B^0 )( \hat f_{q-1,k}  A_t)}- F_{\hat f_{q-1,k}  A_t}
$$
is off-diagonal to order $k^{- \frac{q+1}{2}-\frac{q-1}{2}} = k^{-q}$. Thus upon contracting, the new diagonal elements in $ \Lambda_{\omega_k}(i F_{\exp (k^{- \frac{q+1}{2}} \tilde s_B^0 ) (\hat f_{q-1,k}  A_t)})$ can only come beginning at order $k^{-q-1}$, since these are base terms. 

We can now proceed in the same manner to find 
\begin{itemize} 
\item $\tilde s_B^{\frac{1}{2}}, \ldots, \tilde s_B^{q-\frac{1}{2}} \in \pi^* \Gamma_{i\k^\perp} ( B, \End_H (E, \tilde h ) )$; 
\item $\tilde s_F^{\frac{1}{2}}, \ldots, \tilde s_F^{q-\frac{1}{2}} \in \Gamma_0 (X, \End_H (\pi^* E, h) ) $,
\end{itemize}
such that if we set 
$$
f_{q-\frac{1}{2},k}=\displaystyle \underset{i}{\Pi}\left( \exp(  k^{-\frac{q+1}{2}}( k^{-i-1} \tilde s_F^i))\circ \exp( k^{-\frac{q+1}{2}} ( k^{-i} \tilde s_B^i  ))\right)\circ \hat f_{q-1,k} 
$$
and put $A_{q-\frac{1}{2},k} = A_t^{f_{q-\frac{1}{2},k}}$,
then
$$
 \Lambda_{\omega_k}(i F_{A_{q-\frac{1}{2},k}})= c_{q-1,k} \Id_E + O(k^{-q}),
$$
 the key being that the new error introduced at the lower orders will remain off-diagonal to $O(k^{-q-1})$, so that the projection to $i \k$ is unchanged, to this order. 
Finally, the statement about the orders of the basic and vertical components of $s_{q-1,k}=f_{q-1,k}-\Id_E$ follows easily by construction.
\end{proof}

 For the further stages of the approximation procedure, we need a more detailed understanding of the expansion of the linearised operator. We will let 
$$
\Id_{\pm} = \frac{1}{\rank (\cG_1)} \Id_{\cG_1} - \frac{1}{\rank (\cG_2)} \Id_{\cG_2}\in i\k,
$$
which is in the kernel of the Laplacian for the complex structure $\delb_0$, but orthogonal to $\Id_E$, that is to the kernel of the Laplacian for the complex structure of $E$.
\begin{proposition}
 \label{prop:fullexpansion}
Let $\Delta_k $ denote the Laplace operator associated to $\omega_k$ and the Chern connection of $(h, \pi^*\delb_{\lambda k^{-\frac{q-1}{2}}})$. Then,
\begin{equation}
 \label{eq:linearDeltaIdpm}
\Delta_k ( \Id_{\pm} )  = \lambda  k^{-\frac{q+3}{2}}\: \sigma_{o.d.}+k^{-q}ic\lambda^2 \:\Lambda_{\om_B}(\gamma\wedge\gamma^*+\gamma^*\wedge\gamma)  +\cO(k^{-q-\frac{1}{2}}) 
\end{equation}
for a constant $c\neq 0$ and where $\sigma_{o.d.}\in \pi^*\Gamma(B,\End_H(E,\tilde h))$ is off-diagonal and depends on $k$, but satisfies $\sigma_{o.d.}=\cO(1)$. In particular,
there is a constant $c' \neq 0$ such that
\begin{equation}
\label{eq:projectionDeltaIdpm}
\Pi_{i\k} \left( \Delta_k ( \Id_{\pm} ) \right) =c' k^{-q}  \lambda^2  \Id_{\pm} + O \left( k^{-q-\frac{1}{2}}  \right).
\end{equation}

The expansion \eqref{eq:linearDeltaIdpm} also holds at a Chern connection on $\pi^* E \to X$ coming from a complex structure $f_k \cdot \pi^*\delb_{\lambda k^{-\frac{q-1}{2}}}$, provided $f_k =\Id_E+ s_k$ for some $s_k \in \Gamma \left( X, \End  \pi^* E\right)$ such that 
\begin{itemize} 
\item the base component $s_{B,k}$ is $O(k^{-\frac{1}{2}})$ and is diagonal up to, and including, order $k^{-\frac{q-1}{2}}$;
\item the vertical component $s_{F,k}$ is $O(k^{-\frac{3}{2}})$ and is diagonal up to, and including, order $k^{-\frac{q+1}{2}}$.
\end{itemize}
\end{proposition}
\begin{proof}
Following the proof of Proposition \ref{prop:initialexpansion}, using \eqref{eq:inducedconnectiongamma}, we consider the expansion in $\lambda$ of the Laplace operator $\Delta_k$ associated to $\omega_k$ and the Chern connection of $(h, \pi^*\delb_{\lambda k^{-\frac{q-1}{2}}})$
to prove \eqref{eq:linearDeltaIdpm}. We specialise to $s = \Id_\pm$, and use the full expression of $\Lambda_k$ together with the full expansion in $\lambda$ to see that 
\begin{equation*}
\begin{array}{ccc}
 \Delta_k\left(\Id_\pm\right) & = & i\lambda k^{-\frac{q-1}{2}} \:\Lambda_k \left([\delb_{A_{0,\End E}} \gamma^*,\Id_\pm ]+[\partial_{A_{0,\End E}}\gamma,\Id_\pm]\right)\\
          & + &  i\lambda^2  k^{-(q-1)} \:\Lambda_k\left( [\gamma,[\gamma^*,\Id_\pm]]-[\gamma^*,[\gamma,\Id_\pm]]\right),
 \end{array}
\end{equation*}
since $\Id_{\pm}$ is in the kernel of $\Delta_{0,k}$ as $A_0$ is a product connection. Consider first the term
\begin{equation}
\label{eq:firsttermlinearisation}
i\lambda k^{-\frac{q-1}{2}} \:\Lambda_k \left([\delb_{A_{0,\End E}} \gamma^*,\Id_\pm ]+[\partial_{A_{0,\End E}}\gamma,\Id_\pm]\right).
\end{equation}
As $A_0$ is a product connection, $\delb_{A_{0,\End E}} \gamma^*$ and $\partial_{A_{0,\End E}}\gamma$ are off-diagonal, and thus (\ref{eq:firsttermlinearisation}) is off-diagonal. Then, using the fact that $\left([\delb_{A_{0,\End E}} \gamma^*,\Id_\pm ]+[\partial_{A_{0,\End E}}\gamma,\Id_\pm]\right)$ is pulled back from $B$, the lower order contribution in the $k$-expansion of (\ref{eq:firsttermlinearisation}) will come from $\Lambda_{\om_B}\left([\delb_{A_{0,\End E}} \gamma^*,\Id_\pm ]+[\partial_{A_{0,\End E}}\gamma,\Id_\pm]\right)$, wich vanishes by $\Lambda_{\om_B}\left(\delb_{A_{0,\End E}} \gamma^*\right)=\Lambda_{\om_B}\left(\partial_{A_{0,\End E}}\gamma\right)=0$. Thus, we see that the term (\ref{eq:firsttermlinearisation})
is off-diagonal, and of order $k^{-\frac{q-1}{2}}k^{-2}=k^{-\frac{q+3}{2}}$. It gives the term $\sigma_{o.d.}$ in \eqref{eq:linearDeltaIdpm}. Finally, the term 
$$
i\lambda^2  k^{-(q-1)} \:\Lambda_k\left( [\gamma,[\gamma^*,\Id_\pm]]-[\gamma^*,[\gamma,\Id_\pm]]\right)
$$
is pulled back from $B$, and a direct computation using Lemma \ref{lem:expansion contraction} provides the expansion \eqref{eq:linearDeltaIdpm}. Formula \eqref{eq:projectionDeltaIdpm} then follows from Lemma \ref{lem:tracesandda}, item $(ii)$, and \eqref{eq:matrixaa}.

The statement for the perturbed complex structure again follows by taking an expansion of the form $\partial_{A^{f_k}_{k, \End E}} \bar \partial_{A^{f_k}_{k, \End E}}  - \bar \partial_{A^{f_k}_{k, \End E}}  \partial_{A^{f_k}_{k,\End E}}.$
Note first that $\Id_{\pm}$ remains in the kernel of the Laplacian of any product connection, and so we are left with analysing the change to the terms 
$$
i\lambda k^{-\frac{q-1}{2}} \:\Lambda_k \left([\delb_{A_{0,\End E}} \gamma^*,\Id_\pm ]+[\partial_{A_{0,\End E}}\gamma,\Id_\pm]\right)
$$
and
$$
i\lambda^2  k^{-(q-1)} \:\Lambda_k\left( [\gamma,[\gamma^*,\Id_\pm]]-[\gamma^*,[\gamma,\Id_\pm]]\right).
$$
 Since we only care about terms up to order $k^{-q}$, the latter is unchanged, as it is already at this critical order. On the other hand, the former changes, but only by diagonal elements acting on off-diagonal elements to leading orders. Thus, as in the proof of Lemma \ref{lem:loworder}, the first potentially diagonal changes happen at the product of the orders where the first off-diagonal changes occur. For the change in the complex structure, this is at order $k^{-\frac{q-1}{2}}$,  and for $f_k$, this is at order $k^{-\frac{q}{2}}$ for basic terms and $k^{-\frac{q}{2}-1}$ for the vertical ones, by assumption. Thus the first potentially diagonal change to $\partial_{A^{f_k}_{k, \End E}} \bar \partial_{A^{f_k}_{k, \End E}}  - \bar \partial_{A^{f_k}_{k, \End E}}  \partial_{A^{f_k}_{k,\End E}}$ applied to $\Id_{\pm}$ occurs at order $k^{-q+\frac{1}{2}}$ for basic terms and order $k^{-q-\frac{1}{2}}$ for vertical terms, and so upon contracting the first potentially diagonal change is at order $k^{-q-\frac{1}{2}}$. Hence the required expansion \eqref{eq:linearDeltaIdpm} holds.
\end{proof}
Next, we deal with the crucial $q^{\textnormal{th}}$ stage. Note that Proposition \ref{prop:fullexpansion} implies that it is only \emph{after} this stage that we can remove all errors in $i\k$ via the linearisation. Thus we rely on the sign condition on $\nu_q$ and the freedom to choose $\lambda$ via the $e^{i\k}$-action to construct the approximate solution. 
\begin{lemma} 
\label{lem:discrepancyorder}
Let $E$ be a simple semistable bundle with graded object $\cG_1 \oplus \cG_2$ such that $\cG_1$ has discrepancy $q$, and such that $\nu_q (\cG_1 ) < \nu_q (E)$. Then there exists $\lambda$ such that if $t= \lambda k^{-\frac{q-1}{2}},$  there are 
\begin{itemize} 
\item gauge transformations $f_{q,k}$,
\item[]
\item constants $c_0 , \ldots , c_{q}$ ,
\end{itemize}
such that if we let 
$A_{q,k} = A_t^{f_{q,k}}$, then for all $k \gg 0$, $$ \Lambda_{\omega_k} \left( i F_{A_{q,k} } \right) = c_{q,k} \Id_E + O \left( k^{-q - \frac{1}{2}} \right),$$
where $c_{q, k} = \sum_{i=0}^{q} c_i k^{-i}.$  Moreover, the same conclusion as in Lemma \ref{lem:loworder} for $s_{q-\frac{1}{2},k}$ hold for $s_{q,k}=f_{q,k}-\Id_E$.
\end{lemma}
\begin{proof}  Let $f_{q-\frac{1}{2}, k}$ be the gauge transformation as in Lemma \ref{lem:loworder}. The curvature $F_k$ of $(h,f_{q-\frac{1}{2}, k} \pi^*\delb_{\lambda k^{-\frac{q-1}{2}}} )$ satisfies
 \begin{align*}
 \Lambda_{\omega_k} \left(iF_k\right) &= c_{q-1,k} \Id_E + k^{-q} \left(\sigma_{i\k} + \sigma_{B} +  \sigma_F   \right) + \cO(k^{-q-\frac{1}{2}}) ,
 \end{align*}
where its components $\sigma_{i\k} \in \langle \Id_{\cG_1} \rangle \oplus \langle \Id_{\cG_2} \rangle$, $\sigma_{B} \in \pi^* \Gamma_{i\k^\perp} ( B, \End_H (E, \tilde h ) )$ and $\sigma_F \in  \Gamma_0 (X, \End_H (\pi^* E, h) ) $ depend on $\lambda$. 

Write $\sigma_{i\k}  = \hat \sigma_{i\k} + \lambda^2\Pi_{i\k} \Lambda_{\om_B} ( i a\wedge a )$. We will first show that we can choose $\lambda$ such that 
\begin{align}
\label{eqn:solveV}
\hat \sigma_{i\k} + \lambda^2 \Pi_{i\k}\Lambda_{\om_B} ( i a\wedge a )  \in \langle \Id_E \rangle.
\end{align}
In the proof of Lemma \ref{lem:loworder}, we saw that the only potential change in the diagonal direction was $\cO(k^{-q-\frac{1}{2}})$ after producing the change $ f_{q-1,k} A_t$. 
Thus there is no change to the projection to $i \k$ at order $k^{-q}$, and so $\sigma_{i\k}$ is the sum of the following two terms. The first is the term $\lambda^2 \Pi_{i\k} \Lambda_{\om_B} ( i a\wedge a ) $ coming as the first non-zero term in the expansion of $\lambda^2 k^{1-q} \Pi_{i\k} \Lambda_{\omega_k } (i a \wedge a )$. The second term $\hat \sigma_{i\k}$ is coming from the product connection $\hat f_{q-1,k} A_0$. Therefore $\hat \sigma_{i \k}$ is given by the $k^{n-q}$ coefficient in the projection of $ \Lambda_{\omega_k} ( iF_{A_{0,k}} )$, which, as in the proof of Lemma \ref{lem:loworder}, is a positive multiple of
$$
\sum_{i=1}^2 \nu_q (\cG_i)  \cdot \Id_{\cG_i}.
$$
On the other hand, the projection of $\lambda^2 \Lambda_{\om_B} ( i a\wedge a ) $ to $i\k$ is a positive multiple of
$$
\sum_{i=1}^2  \frac{1}{\rank ( \cG_i) } \int_X  \left(  \trace_{\cG_i} \lambda^2 \Lambda_{\om_B} ( i a\wedge a )  \right) \omega_X^m \wedge \omega_B^n \cdot \Id_{\cG_i} .
$$
Using item $(ii)$ of Lemma \ref{lem:tracesandda}, we have that $\trace_{\cG_1} a\wedge a = - \trace_{\cG_2}  a\wedge a$, and so the above equals 
$$
C \lambda^2 \left( \frac{ 1 }{\rank{(\cG_1)}}  \cdot \Id_{\cG_1} - \frac{ 1}{\rank{(\cG_2})}  \cdot \Id_{\cG_2} \right)  ,
$$
where $C = \int_X \Lambda_{\omega_B} \trace_{\cG_1} i(a \wedge a)\: \omega_X^m \wedge \omega_B^n$. 

So solving Equation \eqref{eqn:solveV}, boils down to solving 
\begin{align*}
c \nu_q (\cG_1) + C \frac{\lambda^2}{\rank (\cG_1)} =& \tau \\
c \nu_q (\cG_2) -  C \frac{\lambda^2}{\rank (\cG_2)} =& \tau
\end{align*}
for $\lambda$ and $\tau$, where $c$ is a positive constant. Thus what we need is to be able to pick $\lambda$ such that 
\begin{align*}
\lambda^2 =& \frac{c}{C} \frac{\rank (\cG_1) \rank (\cG_2) }{\rank (\cG_1) + \rank (\cG_2) } \left( \nu_q (\cG_2 )-  \nu_q ( \cG_1 ) \right),
\end{align*}
which we can do as $C>0$ by Lemma \ref{lem:tracesandda}, Equation (\ref{eq:positiveC}), and as $ \nu_q (\cG_2 ) > \nu_q ( \cG_1 )$ by hypothesis and Corollary \ref{cor:subadditivity of slopes}.  Thus for this choice of $\lambda$ there is a constant $c_q$ such that 
 \begin{align*}
 \Lambda_{\omega_k}( iF_k) &= (c_{q-1,k} + k^{-q} c_q) \Id_E + k^{-q} \left( \sigma_F + \sigma_{B} \right) + \cO(k^{-q-\frac{1}{2}}) .
 \end{align*}

The errors $\sigma_F$ and $\sigma_{B}$ will be removed via the linearisation using Proposition \ref{prop:initialexpansion}. Since by Remark \ref{rem:horizontallaplaceandpulback} and item $(v)$ of Lemma \ref{lem:tracesandda} the image of $\Delta_{\cH, 0} $ contains $\pi^*\Gamma_{i\k^\perp} ( B, \End_H (E, \tilde h ) )$ (recall decomposition (\ref{eq:L2splitting})), there exists $s_B^{q}$ such that $\Delta_{\cH, 0} ( s_B^q ) = - \sigma_{B}$, and similarly, using Lemma \ref{lem:vertical laplacian invertible}, there exists $s_F^q$ such that $\Delta_{\scV} (s_F^q ) = - \sigma_F.$ Thus the curvature of the connection
\begin{align*}
A_{q,k} &= \exp \left( s_F^q k^{-q} + s_B^i k^{1-q} \right) \cdot A_{q-1,k} \\
&= A_t^{f_{q,k}},
\end{align*}
where 
$f_{q,k} = \exp \left( s_F^q k^{-q} + s_B^i k^{1-q} \right)\circ f_{q-\frac{1}{2}, k} $, satisfies
$$
\Lambda_{\omega_k} \left( i F_{A_{q,k} } \right) = c_{q,k} \Id_E + O \left( k^{-q- \frac{1}{2}} \right),
$$
as required.
\end{proof}

We are now ready to prove that we can obtain approximate solutions to arbitrary higher order, via an induction argument. The key is that from now on we can remove any term orthogonal to $\Id_E$, using Proposition \ref{prop:fullexpansion}. Note that we are again summing over $\frac{1}{2} \mathbb{Z}_{>0}$ rather than the integers.
\begin{proposition} 
\label{prop:highorder}
Let $E$ be a simple semistable bundle with graded object $\cG_1 \oplus \cG_2$ such that $\cG_1$ has discrepancy $q$, and such that $\nu_q (\cG_1 ) < \nu_q (E)$. Then there exists $\lambda$ such that if $t= \lambda k^{-\frac{q}{2}},$ then for each $j \in \frac{1}{2} \mathbb{Z}_{>0}$, there are 
\begin{itemize}
\item gauge transformations $f_{j,k}$,
\item[]
\item constants $\lambda_{q+\frac{1}{2}} , \ldots , \lambda_{j}$,
\item[]
\item constants $c_0 , \ldots , c_{j}$ ,
\end{itemize}
such that if we let $A_{j,k} = A_t^{f_{j,k}}$, then for all $k \gg 0$, $$ \Lambda_{\omega_k} \left( i F_{A_{j,k} } \right) = c_{j,k} \Id_E + O \left( k^{-j-\frac{1}{2}} \right).$$
\end{proposition}
\begin{proof}
The proof is via induction on $j$, noting that Lemma \ref{lem:loworder} and \ref{lem:discrepancyorder} settle the case $j \leq q$. We now assume that $j > q $. We will use the asymptotics of the linearisation given by Propositions \ref{prop:initialexpansion} and \ref{prop:fullexpansion}.

Write the expansion as 
$$
\Lambda_{\omega_k} \left( i F_{A_{j-\frac{1}{2}, k} } \right) = c_{j-\frac{1}{2},k} \Id_E + k^{-j} \left( \sigma^j_{i\k} + \sigma^j_{B} + \sigma^j_F \right) + O(k^{-j-\frac{1}{2}}),
$$
where $\sigma^j_{i \k } = c_j \Id_E + d_j \Id_{\pm}$ for some constants $c_j, d_j$. From Equation \eqref{eq:projectionDeltaIdpm}, we have that for a suitable choice of $\lambda_j$
\begin{align*}
\Lambda_{\omega_k} \left( i F_{\exp ( k^{q-j} \lambda_j \Id_{\pm} ) A_{j-\frac{1}{2}, k} } \right) =& c_{j-\frac{1}{2},k} \Id_E + k^{-j} \left( \sigma_F + \sigma_{B} + c_j \Id_E \right) \\
&+ k^{q-j-\frac{q+3}{2}} \sum_{r=0}^{q-2} k^{-\frac{r}{2}} \tilde \sigma_{r} +  O(k^{-j-\frac{1}{2}}),
\end{align*}
for some off-diagonal terms $\tilde \sigma_{r}$ that depend on $\lambda_j$ and $\lambda$. Note that the exponent occurring in the second line equals $\frac{q-3}{2} - j.$

We now have to remove this error re-introduced to the previous steps. The sections $\tilde \sigma_r$ have a basic and a vertical contribution, and we'll deal with the basic ones, the argument for vertical ones being similar (using $\Delta_\cV$ instead). Note that basic (and vertical) components of an off-diagonal endomorphism are off-diagonal, as can be seen in local trivialisations as in Section \ref{sec:decompositions}.  Then, the basic contributions of the $\tilde \sigma_r$ being off-diagonal, they are in the image of $\Delta_{\cH,0}$. Thus if we perturb by $k^{1 + \frac{q-3}{2} - j} s_B$ for some $s_B$, we can remove the error at the $k^{ \frac{q-3}{2} - j} $ stage. While this changes the lower order terms, note that $ \Delta_{k, j-\frac{1}{2}}(s_B)$  is orthogonal to $\Id_{\pm}$ up to the power $k^{-\frac{q+1}{2}}$, as the connection is a product connection up to this order. Thus $\Delta_{k, j-\frac{1}{2} } (k^{1 + \frac{q-3}{2} - j} s_B)$ is orthogonal to $\Id_{\pm}$ up to order $k^{1 + \frac{q-3}{2} - j -\frac{q+1}{2}} = k^{- j - 1}.$ Moreover, the higher order terms in the Taylor expansion of the new curvature only act on terms of at most order $2(1+ \frac{q-3}{2} -j) = q-2j -1 < -j-1$, as $j>q$. Hence only the linear terms contribute to the change in the $\tilde \sigma_r$.

From this, we see that can remove the error $k^{q-j-\frac{q+3}{2}}  \tilde \sigma_{0}$ (for $r=0$) so that the new lower order error terms are still in the image of $\Delta_{\cH,0}$ (or $\Delta_\cV$) up to order $k^{-j- \frac{1}{2}}$. In other words, the $\tilde \sigma_r$ are perturbed to sections that still are orthogonal to $\Id_{\pm}$. We can therefore continue like this, to remove all the newly introduced errors, until we end up with an error $k^{-j} \left( \hat \sigma^j_F + \hat \sigma^j_{B} \right)$, say, which can now be removed via an element of the form $k^{-j} s_F + k^{1-j} s_{B}.$
\end{proof}

%%%%%%%%%%%%%%%%%%%
%              END 2-COMPONENTS               %
%%%%%%%%%%%%%%%%%%%

\subsection{Approximate solutions: general case}
\label{sec:approximategeneral}
We turn back now to the general setting of Section \ref{sec:basestructurebundle}, with a graded object that has $\ell$ stable components. Relying on Section \ref{sec:ssapproximate}, we will focus on the main differences that appear in the construction of the approximate solutions when more components are involved. The strategy is to first build the approximate solutions to order $q$ by considering the projections of the contracted curvatures onto each $\End(\cF_j)$, inductively on $j$, and then improve to any order by using the linearisation on the full space $\End(E)$. The main new features come from the fact that the various subbundles $\cF_j\subset E$ may come with different discrepancy orders, so that one has to make sure that the perturbations used to improve the $\End(\cF_j)$ component of the solutions don't introduce bad terms on its $\End(\cF_{j+1})$ projection. 

Denote by $q_i$ the discrepancy order of $\cF_i$ for each $i\in[\![ 1, \ell-1]\!]$. Assuming $E$ to be asymptotically stable with respect to subbundles coming from the Jordan--H\"older filtration, we have $\nu_p(\cF_i)<\nu_p(E)$ for $p\leq q_i-1$ and $\nu_{q_i}(\cF_i)<\nu_{q_i}(E)$, for each $i\in[\![ 1, \ell-1]\!]$. As in the two components case, we will introduce a gauge transformation in $\Aut(\Gr(E))$ to control the orders at which the various $\gamma_{ij}$'s will appear in the perturbation from $\delb_0$ to $\delb_E$. We define ${\bf g}_{\underline{\lambda}, \underline{m}}\in G$ by 
\begin{equation}
 \label{eq:initialgaugetransformation}
{\bf g}_{\underline{\lambda}, \underline{m}} = 
\Id_{\cG_1} + \lambda_1 k^{-m_1} \Id_{\cG_2} + \lambda_1\lambda_2k^{-m_1-m_2}\Id_{\cG_3} +\ldots + (\Pi_{i=1}^{\ell-1}\lambda_i) k^{-m_1 -\ldots-m_{\ell-1}} \Id_{\cG_\ell},
\end{equation}
where $\underline{m}=(m_i)_{1\leq i \leq \ell-1}$ is chosen so that for all $i$,
$$
2m_i + 1 = q_i,
$$ 
and the constants $\underline{\lambda}=(\lambda_i)_{1\leq i \leq \ell-1}\in(\R^*)^{\ell-1}$ will be determined later on. Denote by 
$$\gamma_k:={\bf g}_{\underline{\lambda}, \underline{m}}\cdot \gamma. 
$$
Then, the operators $\delb_E$ and $\delb_0+\gamma_k$ are gauge equivalent for any $\underline{\lambda}$, and we will use the latter on (the smooth bundle underlying) $E$ as our starting point in the construction of approximate solutions. We also denote by $A_k$ the Chern connection on $(\pi^*E, \pi^*(\delb_0+\gamma_k), h)$ and by $\Delta_k$ its associated Laplacian with respect to $\omega_k$. We will first derive some estimates for the action of $\Delta_k$ before constructing the approximate solutions.

\subsubsection{The linear operator and its action}
Proposition \ref{prop:initialexpansion} admits the following straightforward generalisation:
\begin{proposition}
\label{prop:linleadingorder}
Under the above setup, there is an expansion
\begin{equation}\label{eqn:laplaceexpansiongeneral}
\Delta_{k} = \Delta_{\Gr(E),k} + O(k^{-\frac{3}{2}}) =  \Delta_{\scV} +   k^{-1} \Delta_{\cH,0} + O(k^{-\frac{3}{2}}) ,
\end{equation}
where $\Delta_{\Gr(E),k}$ is the Laplacian of the graded object with respect to the initial complex structure $\pi^* \dbar_0$, and $ \Delta_{\cH,0} $ is the horizontal Laplacian with respect to that complex structure. The same expansion also holds at a Chern connection on $\pi^* E \to X$ coming from a complex structure $f_k \cdot \pi^*(\delb_0+\gamma_k)$ provided $f_k = \Id_E+ s_k$ for some $s_k \in \Gamma \left( X, \End  \pi^* E \right)$  whose base component $s_{B,k}$ satisfies $s_{B,k} = O(k^{-\frac{1}{2}})$ and whose vertical component $s_{F, k}$ satisfies $s_{F,k} = O(k^{-\frac{3}{2}})$.
\end{proposition}
Hence, the cokernel of the truncation of $\Delta_k$ to order $k^{-1}$ is given by $i\k\subset \g$. We will then need to know precisely the orders of the projections onto $i\k$ of the various terms involved in $\Lambda_{\om_k}(iF_{A_k})$ and its variations. Recall that
\begin{equation}
 \label{eq:expansioncurvaturegeneral}
\Lambda_{\om_k}(iF_{A_k})=\Lambda_{\om_k}(iF_0)+\Lambda_{\om_k}(d_{A_0} a_k) + \Lambda_{\om_k}(ia_k\wedge a_k)
\end{equation}
where $a_k$ is defined by $A_k=A_0+a_k$ or equivalently by $a_k=\gamma_k-\gamma_k^*$. Note then that the Laplacian of $\pi^*(\delb_0+\gamma_k)$ takes the following form:
\begin{equation}
\label{eq:fullexpansionlaplacian}
\begin{array}{ccc}
 \Delta_k & = & \Delta_{\Gr(E),k}\\
          & + &i k^{-1} \:\Lambda_{\om_B} \left(\delb_{A_{0}^{\End E}} ([\gamma_{k}^*,\cdot ])-[\gamma_{k}^*,\delb_{A_{0}^{\End E}}  ]\right)\\
          & + & i k^{-1} \: \Lambda_{\om_B} \left(\partial_{A_{0}^{\End E}}([\gamma_{k},\cdot])-[\gamma_{k},\partial_{A_{0}^{\End E}} ] \right) \\
& + & ik^{-1} \Lambda_{\omega_B} ( [\gamma_{k}, [\gamma_{k}^*, \cdot ] ] - [\gamma_{k}^*, [\gamma_{k}, \cdot ] ] ) \\
& + &  \cO(k^{-\nu} ),
 \end{array}
\end{equation}
where 
$$
\begin{array}{ccc}
 \cO(k^{-\nu})  & = &i(\Lambda_{\omega_k} - k^{-1} \:\Lambda_{\om_B} )\left(\delb_{A_{0}^{\End E}} ([\gamma_{k}^*,\cdot ])-[\gamma_{k}^*,\delb_{A_{0}^{\End E}}  ]\right)\\
          & + & i(\Lambda_{\omega_k}- k^{-1} \: \Lambda_{\om_B} )\left(\partial_{A_{0}^{\End E}}([\gamma_{k},\cdot])-[\gamma_{k},\partial_{A_{0}^{\End E}} ] \right) \\
& + & i(\Lambda_{\omega_k}- k^{-1} \: \Lambda_{\om_B} ) ( [\gamma_{k}, [\gamma_{k}^*, \cdot ] ] - [\gamma_{k}^*, [\gamma_{k}, \cdot ] ] ) 
 \end{array}
$$
stands for higher order terms.
 The following straightforward lemma gives the orders of the various terms involving $\gamma_k$ (recall that $2m_i+1=q_i$, and that applying $\Lambda_{\om_k}$ to a pulled back form is $\cO(k^{-1})$):
\begin{lemma}
 \label{lem:ordergammas}
 Let $(i,p)\in [\![ 1, \ell ]\!]^2$, $i< p$. Then
 $$
 \displaystyle (\gamma_k)_{ip}=(\Pi_{l=i}^{p-1}\lambda_l) k^{-m_i-m_{i+1}\ldots-m_{p-1}}\gamma_{ip}
 $$
 and
 $$
 (\gamma_k)_{ip}\wedge (\gamma_k^*)_{pi}=(\Pi_{l=i}^{p-1}\vert \lambda_l \vert^2) k^{-(q_i-1)-(q_{i+1}-1)\ldots-(q_{p-1}-1)}\gamma_{ip}\wedge \gamma^*_{pi}.
 $$
 Hence,
  $$
\Lambda_{\omega_k}( (\gamma_k)_{ip}\wedge (\gamma_k^*)_{pi})=(\Pi_{l=i}^{p-1}\vert \lambda_l \vert^2)\:\cO( k^{-(\max(q_i,\ldots,q_{p-1})}).
 $$
\end{lemma}
From these estimates, we will deduce the action of the perturbed Laplacian on sections. First, for sections on the diagonal we have (recall that $\Pi_{\langle\Id_{\cF_j}\rangle}$ denotes the $\langle\Id_{\cF_j}\rangle$ component using the projection (\ref{eq:L2splitting})):
\begin{lemma}
 \label{lem:linearisationdiagonals}
 Let $s\in \End(\cG_i)$ and $j\in[\![ 1, \ell ]\!]$. Then
 $$
 \Pi_{\langle\Id_{\cF_j}\rangle} \Delta_k(s) = \cO(k^{-q_j}).
 $$
 If we set $\delb_k=\pi^*(\delb_0+\gamma_k)$ and $\partial_k=\pi^*(\partial_0+\gamma_k^*) $, we also have
 $$
 \Pi_{\langle\Id_{\cF_j}\rangle} \Lambda_{\om_k}\left((\partial_k s-\bar \partial_k s)\wedge (\partial_k s-\bar \partial_k s) \right)=\cO(k^{-q_j}).
 $$
\end{lemma}

\begin{proof}
We first deal with the projection of $\Delta_k(s)$.
 Consider the various terms in the expansion of the Laplacian given by Equation (\ref{eq:fullexpansionlaplacian}). By construction, $\Delta_{\Gr(E),k}$ is self-adjoint and $\Id_{\cF_j}$ is in its cokernel. Thus $\Pi_{\langle\Id_{\cF_j}\rangle}\Delta_{\Gr(E),k}s=0$.  The terms coming from first order partial derivatives of $[\gamma_k,s]$ or $[\gamma_k^*,s]$ are off-diagonal, hence orthogonal to $\Id_{\cF_j}$ as well. So the highest order terms we are after will come from terms ``quadratic in $\gamma$'' of the form $[\gamma_{k}, [\gamma_{k}^*, s ] ]$. We will use the notation $\Lambda_k$ for $\Lambda_{\om_k}$, and argue depending on two cases:
 
\textbf{Case 1:  $i\leq j$.}

 The orthogonal projection of $\Lambda_k[\gamma_{k}, [\gamma_{k}^*, s ] ]$ onto $\langle\Id_{\cF_j}\rangle$ is given by taking the trace of its $\cF_j$ component. As $\mathfrak{sl}(\cF_j)$ is an ideal for the bracket, the projection of $\Lambda_k[\Pi_{\End(\cF_j)}\gamma_{k}, [\Pi_{\End(\cF_j)}\gamma_{k}^*, s ] ] $ to $\langle\Id_{\cF_j}\rangle$ will vanish. Hence, the only term that will contribute to $\Pi_{\langle\Id_{\cF_j}\rangle} \Lambda_k[\gamma_{k}, [\gamma_{k}^*, s ] ]$ will be of the form $\Lambda_k ((\gamma_k)_{ip}\wedge(\gamma_k^*)_{pi}) s $, with $p>j$. By Lemma \ref{lem:ordergammas}, these terms are $\cO(k^{-q_j})$ and the result follows.
 
 \textbf{Case 2 :  $i> j$.}
 
 As we project onto $\langle\Id_{\cF_j}\rangle$, assuming $j<i$, the terms of the form $(\gamma_k)_{ip}\wedge(\gamma_k^*)_{pi} s$ or $(\gamma_k^*)_{ip}\wedge(\gamma_k)_{pi} s$, etc will not contribute. We are left with the terms of the form $\Lambda_k((\gamma_k^*)_{pi}\wedge( s (\gamma_k)_{ip}))$ or $\Lambda_k((\gamma_k)_{pi}\wedge (s (\gamma_k^*)_{ip}))$, with $p\leq j$. The first terms vanish as $p\leq j < i$ implies $(\gamma_k)_{ip}=0$. The second are $\cO(k^{-q_j})$ because $p\leq j < i$ and using Lemma \ref{lem:ordergammas}.  The proof for the estimate of $\Pi_{\langle\Id_{\cF_j}\rangle} \Delta_k(s)$ is complete. 
 
 For the other estimate, we have
 $$
  \Lambda_{k}\left((\partial_k s-\bar \partial_k s)\wedge (\partial_k s-\bar \partial_k s) \right)  =  -\Lambda_k(\partial_k s\wedge \delb_k s)-\Lambda_k(\delb_k s \wedge \partial_k s)
 $$
 and we can expand (omitting pullbacks to ease notations)
 $$
 \begin{array}{ccc}
   \Lambda_k(\partial_k s\wedge \delb_k s) & = & \Lambda_k(\partial_0 s\wedge \bar \partial_0 s) \\
    & + & \Lambda_k  (\partial_0 s \wedge [\gamma_k,s]) +\Lambda_k([\gamma_k^*, s] \wedge \delb_0 s) \\
    & + & \Lambda_k([\gamma_k^*,s]\wedge[\gamma_k,s]).
 \end{array}
 $$
 Each of the three terms from the right hand side of this last equality can be dealt with separately. For the first, if $j<i$, as $A_0$ is a product connection, we deduce $\partial_0 s\wedge \bar \partial_0 s\in\Om^{1,1}(X,\End(\cG_i))$ and thus
 $$
 \Pi_{\langle\Id_{\cF_j}\rangle}\Lambda_k(\partial_0 s\wedge \bar \partial_0 s)=0.
 $$
 If $i\leq j$, the full expansion for the curvature $F_{\exp(s)\cdot A_0}$ of the Chern connection of $\exp(s)\cdot \pi^*\delb_0$ is given by  (see e.g. \cite[Section 1]{BuSchu}) :
 $$
 F_{\exp(s)\cdot A_0}=F_{A_0} + (\bar \partial_0\partial_0-\partial_0\bar \partial_0)s+\left((\partial_0 s-\bar \partial_0 s)\wedge (\partial_0 s-\bar \partial_0 s) \right).
 $$
 By Chern--Weil theory, the $\Id_{\cF_j}$-components of the contractions of $F_{\exp(s)\cdot A_0}$ and $F_{A_0}$ agree. As $\Id_{\cF_j}$ lies in the kernel of $\Delta_{\Gr(E),k}=\Lambda_k(\delb_0\partial_0-\partial_0\delb_0)$, we conclude that
 $$
  \Pi_{\langle\Id_{\cF_j}\rangle}(\Lambda_k(\partial_0 s\wedge \delb_0 s)-\Lambda_k(\delb_0 s \wedge \partial_0 s))=0.
 $$
 For the second term, using that $s$ is diagonal and $(\gamma_k)_{ii}=0$, we obtain 
 $$
 \Pi_{\langle\Id_{\cF_j}\rangle}\Lambda_k  (\partial_0 s \wedge [\gamma_k,s])=\Pi_{\langle\Id_{\cF_j}\rangle}\Lambda_k([\gamma_k^*, s] \wedge \delb_0 s) =0.
 $$
 Finally, the last term $\Pi_{\langle\Id_{\cF_j}\rangle} \Lambda_k([\gamma_k^*,s]\wedge[\gamma_k,s])$ is quadratic in $\gamma$, and can be dealt with as for the similar quadratic terms in $\Delta_k(s)$.
\end{proof}
We prove next a similar result for off-diagonal gauge transformations. As in the two components case, these gauge transformations will first be introduced to remove error terms of the form $\Lambda_kd_{A_0} a_k$ in (\ref{eq:expansioncurvaturegeneral}). Let then $-\nu(i,p)$ be the order of the $(i,p)$ component of $d_{A_0}a_k$ (that is $\nu(i,p)=m_i+\ldots+m_{p-1}$ by Lemma \ref{lem:ordergammas}). We will then use the action of $\Delta_k$ on off-diagonal sections of the form $k^{-\nu(i,p)-\frac{1}{2}}s$ (the extra $-\frac{1}{2}$ will come from $(iv)$ Lemma \ref{lem:tracesandda} ): 
\begin{lemma}
 \label{lem:linearisationoffdiagonals}
 Let $s\in \Hom(\cG_p,\cG_i)$, with $i<p$, and $j\in[\![ 1, \ell ]\!]$. Then
 $$
 \Pi_{\langle\Id_{\cF_j}\rangle} \Delta_k(k^{-\nu(i,p)-\frac{1}{2}}s) =\cO(k^{-q_j-\frac{1}{2}})
 $$
 and
  $$
 \Pi_{\langle\Id_{\cF_j}\rangle} \Lambda_{\om_k}\left(k^{-2\nu(i,p)-1}\left((\partial_k s-\bar \partial_k s)\wedge (\partial_k s-\bar \partial_k s) \right)\right)=\cO(k^{-q_j-\frac{1}{2}}).
 $$
 A similar statement holds when $p< i$. 
\end{lemma}
\begin{proof}
We give the proof assuming $i<p$, as the case $p<i$ is similar. We start estimating $\Pi_{\langle\Id_{\cF_j}\rangle} \Delta_k(k^{-\nu(i,p)-\frac{1}{2}}s)$. The proof follows the one of Lemma \ref{lem:linearisationdiagonals}, but this time the highest order contributions come from terms ``linear in $\gamma$'' of the form $[\gamma_k, s]$ (or $ [\gamma_k^*, s]$). To simplify exposition we will neglect the action of $\partial_{A_{0}^{\End E}}$ or $\delb_{A_{0}^{\End E}}$ in the argument as these operators are of order zero and preserve the matrix block decomposition.

\textbf{Case 1 :  $i\leq j$.}

  Arguing as in the proof of Lemma \ref{lem:linearisationdiagonals}, Case $1$, the only non-zero contributions that we will obtain when projecting onto $\langle\Id_{\cF_j}\rangle$ are coming from terms of the form $ (\gamma_k^*)_{pi} s$ with $p>j$, which are of order 
 $k^{-\nu(i,p)}$. Tracing with $\om_k$ will add an extra $k^{-1}$ contribution. Then, as $i\leq j <p$ and $q_j=2m_j+1$, we have
 $$
 (k^{-1}(\gamma_k^*)_{pi} )(k^{-\nu(i,p)-\frac{1}{2}}s)=\cO(k^{-q_j-\frac{1}{2}}),
 $$
 and the result follows.
 
 \textbf{Case 2 :  $i> j$.}
 
 As $i<p$, $(\gamma_k)_{pi}=0$. Then, as $j<i<p$, the projection onto $\langle\Id_{\cF_j}\rangle$ of terms coming from $(\gamma_k^*)_{pi} s$ or $s (\gamma_k^*)_{pi}$ vanish. Hence, there are no contributions from linear terms in $\gamma_k$. Then, arguing as in the proof of Lemma \ref{lem:linearisationdiagonals}, we see that the contributions from quadratic terms in $\gamma_k$ are $\cO(k^{-q_j-\frac{1}{2}})$ and the proof of the first estimate is complete.
 
 The second estimate can be obtained with similar considerations, noting that this time the highest order contribution comes from terms linear in $\gamma$ as $$\Pi_{\langle\Id_{\cF_j}\rangle} \Lambda_k(\partial_0 s\wedge \bar \partial_0 s)=0$$ from the fact that $s\in\Hom(\cG_p,\cG_i) $.
\end{proof}
\begin{remark}
 \label{rem:pertubationODE}
 The Lemmas \ref{lem:linearisationdiagonals} and \ref{lem:linearisationoffdiagonals} ensure that we will not affect the $\langle\Id_{\cF_j}\rangle$ component of the contracted curvature to order $q_j$ while killing off errors orthogonal to $\g$ (note that we use sections $s\in \End(\cG_i)$ that are $\cO(k^{-\frac{1}{2}})$). Indeed, the full expansion for the curvature  of $\exp(s)\cdot A_k$ reads  (see e.g. \cite[Section 1]{BuSchu}) :
 $$
 F_{\exp(s)\cdot A_k}=F_{A_k} + (\bar \partial_k\partial_k-\partial_k\bar \partial_k)s+\left((\partial_k s-\bar \partial_k s)\wedge (\partial_k s-\bar \partial_k s) \right).
 $$
 Note that the linear part of $\Lambda_{\om_k}F_{\exp(s)\cdot A_k}$ is $\Delta_k$.
 Then, from Lemmas \ref{lem:linearisationdiagonals} and \ref{lem:linearisationoffdiagonals}, for appropriate $\nu$ (that is $\nu\geq \frac{1}{2}$ if $s$ is diagonal, and $\nu\geq \nu(i,p)+\frac{1}{2}$ if $s$ is off-diagonal in position $(p,i)$), we will have 
 $$
 \Pi_{\langle\Id_{\cF_j}\rangle}\Lambda_{\om_k}( F_{\exp(k^{-\nu}s)\cdot A_k})=\Pi_{\langle\Id_{\cF_j}\rangle} (\Lambda_{\om_k} F_{A_k})+\cO(k^{-q_j-\frac{1}{2}}).
 $$
\end{remark}
Finally, we compute the Laplacian on sections of the form $\Id_i:=\Id_{\cG_i}$. They will be used to remove errors in $\g=\aut(\Gr(E))$. We introduce, for $(p,l)\in [\![1,\ell]\!]^2$:
 $$
 \Id_{pl}=(\frac{1}{\rk\,\cG_p}\Id_{\cG_p}-\frac{1}{\rk\,\cG_{l}}\Id_{\cG_{l}}).
 $$
\begin{lemma}
 \label{lem:linearisationIdis}
  For all $j\in [\![ 1, \ell-1]\!]$, there is a negative constant $a_{j,j+1}$ such that
\begin{equation}
 \label{eq:deltaactionIdisonautg}
\Pi_{\g}\Delta_k(\Id_{\cF_j}) =  a_{j,j+1}k^{-q_j}\Id_{j,j+1}+\cO(k^{-q_j-1}).
\end{equation}
\end{lemma}
The constants $a_{j,j+1}$, independent of $k$, will actually depend on $\underline{\lambda}$, but will play a role only after fixing $\underline\lambda$, so this dependence will not affect the argument.
\begin{proof}
As in the proof of Lemma \ref{lem:linearisationdiagonals}, the only contributions will come from quadratic terms in $\gamma_k$. To ease notations, we will omit subscript $k$ for the moment.
As we need more than simply the orders of the terms $\Pi_\g\Delta(\Id_i)$,
 we start by recalling the precise formula for the differential of the term 
 $$
 (a\wedge a)^{1,1}=((\gamma-\gamma^*)\wedge(\gamma-\gamma^*))^{1,1}=-(\gamma\wedge \gamma^*+\gamma^*\wedge \gamma)$$
 with respect to an infinitesimal action of some gauge transformation $\exp(s)$. This gives (recall we consider a right action) :
 $$
 -([\gamma,s]\wedge \gamma^* + \gamma\wedge [\gamma,s]^*+ [\gamma,s]^*\wedge \gamma + \gamma^*\wedge [\gamma,s]).
 $$
 Using $[\gamma,s]=\gamma s - s\gamma$ and $[\gamma,s]^*=[s^*, \gamma^*]$, this equals:
 $$
 \begin{array}{c}
 s (\gamma \wedge \gamma^*) -\gamma\wedge( s \gamma^*) + (\gamma\wedge \gamma^*)  s^* - \gamma \wedge( s^* \gamma^*)\\
 + \gamma^*\wedge (s^*  \gamma) - s^* (\gamma^* \wedge \gamma) + \gamma^*\wedge( s  \gamma) -( \gamma^* \wedge \gamma) s.
 \end{array}
 $$
 Assuming $s$ to be an hermitian endomorphism we obtain :
 $$
  s (\gamma \wedge \gamma^* )-2\gamma\wedge( s \gamma^*) + (\gamma\wedge \gamma^*) s +2 \gamma^*\wedge( s  \gamma )- s( \gamma^* \wedge \gamma)  - (\gamma^* \wedge \gamma) s.
 $$
 From now on we set $s=\Id_i$ for $i\in[\![1, \ell]\!]$. We first compute the $\End(\cG_1)\oplus \ldots \oplus \End(\cG_\ell)$ component of the contraction of last formula. After re-ordering terms, and using $\gamma_{ij}=0$ for $j\leq i$ we obtain:
$$
2\Lambda (\gamma \wedge \gamma^*)_{ii} -2\Lambda(\gamma^* \wedge \gamma)_{ii}-2\sum_{p<i}\Lambda (\gamma\wedge \gamma^*)_{pp}  +2 \sum_{i<p}\Lambda(\gamma^* \wedge \gamma)_{pp}.
$$
 We now compute the $\g$-component, which we denote $(\Delta(\Id_i))_\g$. Denoting by $\tr_{\cG_q}$ the trace of the $\langle\Id_{\cG_q}\rangle$ component, we obtain :
 $$
 \begin{array}{ccc}
 (\Delta(\Id_i))_\g & = & 2\left(\displaystyle \sum_{i<p} \tr_{\cG_i}\Lambda( \gamma_{ip}\wedge\gamma_{pi}^*) \frac{\Id_{\cG_i}}{\rank(\cG_i)} + \tr_{\cG_p}\Lambda (\gamma_{pi}^*\wedge\gamma_{ip}) \frac{\Id_{\cG_p}}{\rank(\cG_p)} \right) \\
  & & \\
  & & - 2\left(\displaystyle \sum_{p<i} \tr_{\cG_i}\Lambda (\gamma_{ip}^*\wedge\gamma_{pi}) \frac{\Id_{\cG_i}}{\rank(\cG_i)} + \tr_{\cG_p}\Lambda (\gamma_{pi}\wedge\gamma_{ip}^*) \frac{\Id_{\cG_p}}{\rank(\cG_p)} \right) \\
  & & \\
  & = &  2(\displaystyle \sum_{i<p} A_{ip}\Id_{ip})  - 2(\displaystyle \sum_{p<i} A_{pi} \Id_{pi} )
 \end{array}
 $$
 where we set $A_{ip}=\tr_{\cG_i}\Lambda( \gamma_{ip}\wedge\gamma_{pi}^*)=-\tr_{\cG_p}\Lambda (\gamma_{pi}^*\wedge\gamma_{ip})$. Note that $A_{ip}$ depends on $k$ and $\underline\lambda$. Using now Lemma \ref{lem:ordergammas}, and the fact that $\gamma_{i,i+1}\neq 0$ (recall Lemma \ref{lem:gammaijnonzero} and Lemma \ref{lem:tracesandda}), we obtain by integration over $X$ :
 \begin{equation}
  \label{eq:pigdeltaIdiprecise}
 \begin{array}{ccc}
 \Pi_{\g} \Delta_k(\Id_i) & = & a_{i,i+1} k^{-q_i}\Id_{i,i+1} -  a_{i-1,i} k^{-q_{i-1}}\Id_{i-1,i} \\
  &  & \\
                 & + & \displaystyle \sum_{i+2\leq p } a_{i,p} k^{-q_i-(q_{i+1}-1)-\ldots-(q_{p-1}-1)} \Id_{ip} \\
                 & & \\
                  & - & \displaystyle  \sum_{p\leq i-2 } a_{p,i} k^{-q_{i-1}-(q_{i-2}-1)-\ldots-(q_{p}-1)} \Id_{pi} +\cO(k^{-\nu})
 \end{array}
 \end{equation}
 for constants $a_{i,i+1}<0$ and $a_{i-1,i}<0$ (setting $a_{0,1}=a_{\ell,\ell+1}=0$ by convention), and constants $a_{p,l}\in \R_-$ for $\vert p-l\vert\geq 2$, and where $\cO(k^{-\nu})$ stands for higher order terms, where the order depends on the position in matrix block decomposition. The result follows by summing over $i$ from $1$ to $j$.
\end{proof}

\begin{remark}
\label{rem:samelemmasperturbedconnections}
Using the same argument as in Proposition \ref{prop:initialexpansion}, we have that the conclusions of Lemma \ref{lem:linearisationdiagonals}, Lemma \ref{lem:linearisationoffdiagonals} and Lemma \ref{lem:linearisationIdis} also hold when applying the Laplacian of a complex structure $f_k \cdot \pi^*(\delb_0+\gamma_k)$ provided $f_k = \Id_E +s_k$ for some $s_k \in \Gamma \left( X, \End \pi^* E \right)$  whose base component $s_{B,k}$ satisfies $s_{B,k} = O(k^{-\frac{1}{2}})$ and whose vertical component $s_{F, k}$ satisfies $s_{F,k} = O(k^{-\frac{3}{2}})$.
\end{remark}

\subsubsection{The inductive argument}
\label{sec:inductiveapproximate}
We now proceed to the construction of the approximate solutions. We begin by solving the equation up to the maximal discrepancy order of $E$. In order to do so, we first need to fix the constants $\underline{\lambda}$. This is where the choice of $\underline{m}=(m_1,\ldots,m_{\ell-1})$ in (\ref{eq:initialgaugetransformation}) to define ${\bf g}_{\underline{\lambda}, \underline{m}}$ turns useful, and also where we use asymptotic stability with respect to subbundles coming from the Jordan--H\"older filtration.
\begin{proposition}
 \label{prop:fixingthe lambdas}
 There exists $\underline{\lambda}=(\lambda_1,\ldots,\lambda_{\ell-1})$ independent of $k$ such that for all $j\in [\![ 1, \ell-1 ]\!]$,
\begin{align*}
\Pi_{\langle \Id_{\cF_j}\rangle }  \Lambda_{\om_k} \left( i F_{A_k} \right) = \mathcal{C}\,k^{-n}\mu_k (E)\Id_{\cF_j} + O(k^{-q_j-\frac{1}{2}}),
\end{align*}
where
$$
\cC=\frac{2\pi(m+n)}{\Vol(F)\Vol(B)}.
$$
\end{proposition}
\begin{proof}
Recall the definition of $\mu_k$ at the beginning of Section \ref{sec:adiabatic slope stability}, and the definition of the $\nu_i$'s in Equation (\ref{eq:nuslope}).
Let $j\in[\![1, \ell-1 ]\!]$.
 From Equation (\ref{eq:expansioncurvaturegeneral}), and by $(i)$ of Lemma \ref{lem:tracesandda},
 $$
\Pi_{\langle \Id_{\cF_j}\rangle } \Lambda_{\om_k}(iF_{A_k})=\Pi_{\langle \Id_{\cF_j}\rangle }\Lambda_{\om_k}(iF_0)+\Pi_{\langle \Id_{\cF_j}\rangle }\Lambda_{\om_k}(ia_k\wedge a_k).
 $$
 We compute
 $$
 \Pi_{\langle \Id_{\cF_j}\rangle }\Lambda_{\om_k}(iF_0)=\left(\frac{1}{\Vol(B)\Vol(F)}\int_X\tr_{\cF_j}(\Lambda_{\om_k}(iF_0))\;\om_X^m\wedge(\pi^*\om_B)^n\right)\frac{\Id_{\cF_j}}{\rk\;\cF_j}.
 $$
 To higher order in the $k$-expansion, the term
 $$
 \int_X\tr_{\cF_j}(\Lambda_{\om_k}(iF_0))\;\om_X^m\wedge(\pi^*\om_B)^n
 $$
 equals
 $$
 \begin{array}{ccc}
 \displaystyle k^{-n}\int_X\tr_{\cF_j}(\Lambda_{\om_k}(iF_0))\;(\om_k)^{m+n} & = & \displaystyle (m+n)k^{-n}\int_X\tr_{\cF_j}(iF_0)\wedge\om_k^{m+n-1} \\

 & =  & 2\pi (m+n)k^{-n}\mu_k(\cF_j)\rk(\cF_j),
\end{array}
$$
where we used Chern--Weil theory in the last equality.
 A direct computation, using Lemma \ref{lem:ordergammas}, gives 
 $$
 \tr_{\cF_j}(\Lambda_{\om_k}(ia_k\wedge a_k)) = -\vert \lambda_j\vert^2\tr_{\cG_j}\left( \Lambda_{\om_B} (i\gamma_{j,j+1}\wedge \gamma^*_{j+1,j})\right)k^{-q_j} + \cO(k^{-q_j-\frac{1}{2}}).
 $$
 Finally, by asymptotic stability with respect to the $\cF_j$'s, for $p<q_j$ we have $\nu_p(E)=\nu_p(\cF_j)$.
Then, the equation we want to solve is :
\begin{equation}
 \label{eq:solvingthelambdaj}
 \begin{array}{c}
 \cC (k^{-n}\mu_k(E) -  k^{-n}\mu_k(\cF_j))  =  \\
 -k^{-q_j}\displaystyle\frac{\vert \lambda_j\vert^2}{\rank(\cF_j)\Vol(B)} \int_B \tr_{\cG_j}\left( \Lambda_{\om_B} (i\gamma_{j,j+1}\wedge \gamma^*_{j+1,j})\right)\,\om_B^n
                + \cO(k^{-q_j-\frac{1}{2}}).
 \end{array}
\end{equation}
As 
$$k^{-n}(\mu_k(E) -  \mu_k(\cF_j))=k^{-q_j}(\nu_{q_j}(E)-\nu_{q_j}(\cF_j))+\cO(k^{-q_j-\frac{1}{2}}),$$
the above equation reduces to solve
\begin{equation}
 \label{eq:solvinglambdajdetail}
\cC(\nu_{q_j}(E)-\nu_{q_j}(\cF_j))=-\displaystyle\frac{\vert \lambda_j\vert^2}{\rank(\cF_j)\Vol(B)} \int_B \tr_{\cG_j}\left( \Lambda_{\om_B} (i\gamma_{j,j+1}\wedge \gamma^*_{j+1,j})\right)\,\om_B^n.
\end{equation}
Again by asymptotic stability, we have $\nu_{q_j}(E)>\nu_{q_j}(\cF_j)$.  Using $(iii)$ of Lemma \ref{lem:tracesandda}, we see that the signs of Equation (\ref{eq:solvinglambdajdetail}) are compatible at order $q_j$, so that this is solvable in $\lambda_j$. The result follows.
\end{proof}

We then choose $\underline{\lambda}$ as in Proposition \ref{prop:fixingthe lambdas}. We now perturb so it is not only the projection that satisfies the required equation. First of all, simply from the mapping properties of $\Delta_{\Gr(E),k}$, and by Proposition \ref{prop:linleadingorder}, we can perturb $A_k$ as in Section \ref{sec:ssapproximate} so that all the terms orthogonal to $\g=\oplus_i \langle \Id_{\cG_i} \rangle$ to any desired order actually vanish. Moreover, we can do this without changing the projections of the contracted curvature to $\langle \Id_{\cF_j}\rangle$ until, and including, order $q_j$, for $1\leq j \leq \ell-1$. This can be seen as follows. From Equation \ref{eq:expansioncurvaturegeneral}, there are three types of errors to be corrected:
\begin{itemize}
 \item[$(1)$] errors coming from $\Lambda_{\om_k}iF_{0}$,
 \item[$(2)$] errors coming from $\Lambda_{\om_k}d_{A_0}(a_k)$,
 \item[$(3)$] errors coming from $\Lambda_{\om_k}(ia_k\wedge a_k)$.
\end{itemize}
The type $(1)$ errors are diagonal and so can be dealt with on each $\End(\cG_i)$ independently. They are at least $\cO(k^{-\frac{1}{2}})$, so using Lemma \ref{lem:linearisationdiagonals},  removing them won't affect the $\langle \Id_{\cF_j}\rangle$ projection to order $q_j$, see Remarks \ref{rem:pertubationODE} and \ref{rem:samelemmasperturbedconnections}. These perturbations might introduce new off-diagonal errors, but from Equation (\ref{eq:fullexpansionlaplacian}), we see that these errors will be of higher order than type $(2)$ errors above, and so will be dealt with as the type $(2)$ ones.
By Lemma \ref{lem:tracesandda}, type $(2)$ errors are off-diagonal and of order $k^{-\nu(i,p)-\frac{3}{2}}$ on position $(i,p)$. By Lemma \ref{lem:linearisationoffdiagonals}, we can remove them without affecting the $\langle \Id_{\cF_j} \rangle$ projections to the desired orders, still using Remarks \ref{rem:pertubationODE} and \ref{rem:samelemmasperturbedconnections}. By Equation \ref{eq:fullexpansionlaplacian}, we see that these perturbations will introduce new diagonal errors but to higher orders, and those can be removed using Lemma \ref{lem:linearisationdiagonals} again.
 Finally, the type $(3)$ errors can be diagonal and off-diagonal, but always of higher order than the previous ones, and can be removed in the same way.  
 
 We will then assume we have done this first set of perturbations up to and including order $q$, the maximum discrepancy order of $E$, and keep the notation $A_k$ for those new perturbed connections. Now, by our choice of $\underline \lambda$ as in Proposition \ref{prop:fixingthe lambdas}, we therefore have for all $j\in [\![1, \ell-1]\!]$ :
 \begin{equation}
 \label{eq:goodprojectionFj}
\Pi_{ \End(\cF_j)}  \Lambda_{\om_k} \left( i F_{A_k} \right) =\cC k^{-n}\mu_k (E)\Id_{\cF_j} + O(k^{-q_j-\frac{1}{2}}).
\end{equation}

The next step is to improve each of the above equalities to order $q$, by induction on $j$.
We expand the first of these as 
\begin{align*}
\Pi_{ \End (\cF_1)}  \Lambda_k \left( i F_{A_k} \right) = & \,\cC\,k^{-n}\mu_k (E)\Id_{\cF_1}  + k^{-q_1 - \frac{1}{2}} c \cdot \Id_{\cF_1}  \\
&+ k^{-q_1 - \frac{1}{2}} \sigma  + O(k^{-q_1 - 1}),
\end{align*} 
where $\sigma \in \Gamma (\End \cF_1)$ is orthogonal to $\Id_{\cF_1}$.
Now, using Lemma \ref{lem:linearisationIdis}, we have
$$
\Pi_{\g}\Delta_k(k^{-\frac{1}{2}}\Id_{\cF_1}) =  a_{1,2}k^{-q_1-\frac{1}{2}}\Id_{1,2}+\cO(k^{-q_1-\frac{3}{2}})
$$
and thus there is a non-zero constant $a_{1,2}'$ such that
$$
\Pi_{\langle \Id_{\cF_1} \rangle}\Delta_k(k^{-\frac{1}{2}}\Id_{\cF_1}) =  a'_{1,2}k^{-q_1-\frac{1}{2}}\Id_{\cF_1}+\cO(k^{-q_1-\frac{3}{2}}).
$$
Hence, if we set
$$
\hat A^1_k = \exp(k^{-\frac{1}{2}}r_1 \cdot \Id_{\cF_1} ) \cdot A_{k}
$$
for a suitable constant $r_1$, then
\begin{align*}
\Pi_{ \End (\cF_1)}  \Lambda_k \left( i \hat F^1_{k} \right) = &\, \cC\,k^{-n}\mu_k (E)\Id_{\cF_1}   + k^{-q_1 - \frac{1}{2}} \sigma  + O(k^{-q_1 - 1}),
\end{align*} 
where $\hat F^1_{k} $ is the curvature of $\hat A^1_k$. Further acting by an endomorphism of $\cF_1$ that is orthogonal to $\Id_{\cF_1}$, we can also remove the $\sigma$ term, so that the curvature is 
\begin{align*}
\Pi_{ \End (\cF_1)} \Lambda_k \left( i \hat F^1_{k} \right) =& \,\cC\, k^{-n}\mu_k (E)\Id_{\cF_1}  + O(k^{-q_1 - 1}).
\end{align*}
Continuing like this, we eventually produce a connection $A^1_k$ with curvature $F^1_k$ such that
\begin{equation}
\label{eq:goodprojectionF1}
\Pi_{ \End (\cF_1)}  \Lambda_k \left( i F^1_{k} \right) = \,\cC\,k^{-n}\mu_k (E)\Id_{\cF_1}+ O(k^{-q -\frac{1}{2}}),
\end{equation}
where $q = \max \{ q_1,\ldots, q_{\ell-1} \}$ is the maximal discrepancy order of $E$. 

By Lemma \ref{lem:linearisationdiagonals}, these perturbations do not affect the $\langle \Id_{\cF_j}\rangle$ components of the contracted curvature to order $q_j$, for $j\geq 2$. Note also that we might introduce new diagonal or off-diagonal errors in the process, but those can be removed along the way as previously discussed for type $(1)$ and type $(2)$ errors. Therefore, we can assume that $A^1_k$ satisfies Equations (\ref{eq:goodprojectionFj}) for $j\geq 2$.

Next, we move on to perturb the $\Gamma (\End \cF_2)$ piece. By construction, we still have
 \begin{align*}
\Pi_{ \End(\cF_2)}  \Lambda_k \left( i F^1_{k} \right) = \cC\,k^{-n}\mu_k (E)\Id_{\cF_2} + O(k^{-q_2-\frac{1}{2}}).
\end{align*}
This means that we can write
\begin{align*}
\Pi_{\End(\cF_2)}  \Lambda_k \left( i F^1_{k} \right) =&\,\cC\, k^{-n}\mu_k (E)\Id_{\cF_2} + k^{-q_2 - \frac{1}{2}} \left( c_1 \cdot \Id_{\cG_1} + c_2 \cdot \Id_{\cG_2} \right) \\
&+ k^{-q_2 - \frac{1}{2}} \sigma + O(k^{-q_2 -1}),
\end{align*}
where now $\sigma \in \Gamma (\End \cF_2)$ is orthogonal to $\Id_{\cG_i}$ for $i=1,2,$ and $c_1,c_2$ are constants. 
Now, if $q_2 = q$, then we have solved the equation on the $\cF_2$-component up to and including the maximum discrepancy order, which is our goal at this stage. On the other hand, if $q_2<q$, then $c_1 = 0$, since by (\ref{eq:goodprojectionF1}) the $\End (\cG_1)$-component of $\Lambda_k \left( i F^1_{k} \right)$ is $\cC\,k^{-n}\mu_k (E)\Id_{\cG_1}$ up to and including order $k^{-q}$ (recall $\cF_1=\cG_1$). Thus the expansion is
\begin{align*}
\Pi_{\End(\cF_2)}  \Lambda_k \left( i F^1_{k} \right) =&\,\cC\, k^{-n}\mu_k (E)\Id_{\cF_2} + k^{-q_2 - \frac{1}{2}}  c_2 \cdot \Id_{\cG_2}  \\
&+ k^{-q_2 - \frac{1}{2}} \sigma + O(k^{-q_2 -1}).
\end{align*}
We can then proceed as in the $\cF_1$-case: we first perturb using a constant multiple of $\Id_{\cF_2}$, which can then remove the term $k^{-q_2 - \frac{1}{2}}  c_2 \cdot \Id_{\cG_2}$, since by Equation (\ref{eq:deltaactionIdisonautg}):
$$
\Pi_{\g}\Delta_k(k^{-\frac{1}{2}}\Id_{\cF_2}) =  a_{2,3}k^{-q_2-\frac{1}{2}}\Id_{2,3}+\cO(k^{-q_2-\frac{3}{2}}).
$$
After this, we then remove the $k^{-q_2 - \frac{1}{2}} \sigma$-term using just the properties of the linearisation of the initial connection on $\Gr(E)$, which captures the leading order contribution of the linearisation at our current connection too, by Proposition \ref{prop:linleadingorder}. We then proceed until we reach the order $k^{-q}$, which we can do since the contribution in the $\Id_{\cG_1}$ component is $\cC\,k^{-n}\mu_k (E)\Id_{\cG_1}$ up to and including order $k^{-q}$. The upshot is that, up to removing new errors in $\g^\perp$ introduced along the way, we get a connection $A^2_k$ on $\pi^*E$ with curvature $F^2_k$ satisfying
\begin{align*}
\Pi_{\End(\cF_2)}  \Lambda_k \left( i F^2_{k} \right) =&\,\cC\, k^{-n}\mu_k (E)\Id_{\cF_2} + O(k^{-q -\frac{1}{2}})
\end{align*}
and such that (\ref{eq:goodprojectionFj}) still holds for $j\geq 3$.

We then proceed in a similar way to build inductively connections $A^j_k$ on $\pi^*E$ with curvatures $F^j_k$ such that 
\begin{align*}
\Pi_{\End(\cF_j)}  \Lambda_k \left( i F^j_{k} \right) =&\,\cC\, k^{-n}\mu_k (E)\Id_{\cF_j} + O(k^{-q -\frac{1}{2}}).
\end{align*}
At each step the connection $A^{j+1}_k$ is obtained from $A^j_k$ by removing the errors in the previously obtained expansion
\begin{align*}
\Pi_{\End(\cF_{j+1})}  \Lambda_k \left( i F^j_{k} \right) =&\,\cC\, k^{-n}\mu_k (E)\Id_{\cF_{j+1}} + k^{-q_{j+1} - \frac{1}{2}}  c_{j+1} \cdot \Id_{\cG_{j+1}}  \\
&+ k^{-q_{j+1} - \frac{1}{2}} \sigma + O(k^{-q_{j+1} -1}).
\end{align*}
The connection $A_k^{\ell -1}$ satisfies
\begin{align*}
\Pi_{\End(\cF_{\ell-1})}  \Lambda_k \left( i F^{\ell-1}_{k} \right) =&\,\cC\, k^{-n}\mu_k (E)\Id_{\cF_{\ell-1}} + O(k^{-q -\frac{1}{2}}),
\end{align*}
but also by Chern--Weil theory:
\begin{align*}
\Pi_{\oplus_{i=1}^\ell \langle \Id_{\cG_i} \rangle}  \Lambda_k \left( i F^{\ell-1}_{k} \right) =&\,\cC\, k^{-n}\mu_k (E)\Id_{E} + O(k^{-q -\frac{1}{2}}).
\end{align*}
Thus all remaining errors in $\End(E)$ to solve the HYM equation up to and including order $k^{-q}$ are orthogonal to $\oplus_{i=1}^\ell \langle \Id_{\cG_i} \rangle$. Hence we can proceed as in the previous steps to deduce that there is a connection $A^\ell_k$ on $\pi^*E$ with curvature $F^\ell_k$ such that
\begin{align*}
\Lambda_k \left( i F^\ell_{k} \right) =& \cC\, k^{-n}\mu_k (E)\Id_{E}+ O(k^{-q -\frac{1}{2}}).
\end{align*}
Thus we have solved the HYM equation up to and including order $k^{-q}$.

We now proceed to go beyond the maximal discrepancy order of $E$, to produce connections that solve the HYM equation to any desired order. We have that
\begin{align*}
 \Lambda_k \left( i F^\ell_{k} \right) =& \,\cC\, k^{-n}\mu_k (E)\Id_{E} \\
&+ k^{-q -\frac{1}{2}} \left( \frac{c_1}{\rk \cG_1} \Id_{\cG_1} + \ldots + \frac{c_\ell}{\rk \cG_\ell}\Id_{\cG_\ell} \right) \\
&+ k^{-q -\frac{1}{2}} \sigma + O(k^{-q -1}),
\end{align*}
for constants $c_1,\ldots,c_\ell$ and a $\sigma$ that is orthogonal to $\oplus_{i=1}^\ell \langle \Id_{\cG_i} \rangle$. Since we know by Chern--Weil theory that the projection of $ \Lambda_k \left( i F^\ell_{k} \right)$ to $\Id_E$ is $\cC\, k^{-n}\mu_k (E)\Id_{E}$, we even have
$$
c_1 + \ldots + c_\ell = 0.
$$
This means that we can rewrite the above as 
\begin{align*}
 \Lambda_k \left( i F^\ell_{k} \right) =& \,\cC\, k^{-n}\mu_k (E)\Id_{E}+ k^{-q -\frac{1}{2}} \left( d_1 \Id_{1,2} +\ldots+ d_{\ell-1} \Id_{\ell-1,\ell} \right) \\
&+ k^{-q -\frac{1}{2}} \sigma + O(k^{-q -1}),
\end{align*}
for constants $(d_i)_{1\leq i \leq \ell-1}$. Perturbing first by $$
d_1' k^{q_1-q-\frac{1}{2}} \Id_{\cF_1}+\ldots + d_{\ell-1}' k^{q_{\ell-1}-q-\frac{1}{2}} \Id_{\cF_{\ell-1}}
$$
for some suitable constants $(d_i')_{1\leq i\leq \ell-1}$, using Lemma \ref{lem:linearisationIdis}, we can remove the term 
$$
 k^{-q -\frac{1}{2}}\left( d_1 \Id_{1,2} +\ldots+ d_{\ell-1} \Id_{\ell-1,\ell} \right).
 $$
If we then perturb by an element orthogonal to $\oplus_{i=1}^\ell \langle \Id_{\cG_i} \rangle$, we can also remove the $\sigma$-term. After correcting the extra errors introduced in the process, we thus obtain a connection $\tilde A_k$ with curvature $\tilde F_k$ satisfying 
\begin{align*}
 \Lambda_k \left( i \tilde F_{k} \right) =& \,\cC\, k^{-n}\mu_k (E)\Id_{E}+ O(k^{-q -1}).
\end{align*}
Proceeding in exactly the same way, we obtain connections on $E$ solving the HYM equation to any desired order in $k$.
\begin{remark}
While the above statements are all pointwise estimates, it follows exactly in the same manner as in Lemma \ref{lem:kexpstable} that these pointwise statements actually give the corresponding statements in any desired Sobolev space.
\end{remark}

\subsection{Perturbation argument}
\label{sec:linop}
We end the proof of Theorem \ref{thm:semistablemain} in this section, by perturbing our approximate solution built in Section \ref{sec:inductiveapproximate} to a genuine solution, following the method of Section \ref{sec:nonlinear}. In order to apply the quantitative implicit function theorem, we will need an estimate on the Laplacian of the approximate solutions.
\subsubsection{Estimating $\Delta_k$}
We first derive an estimate for the Laplacian $\Delta_k$ of the connection $A_k$ associated to $(\pi^*E, \pi^*(\delb_0+\gamma_k), h)$. We will then argue that this estimate also holds for the perturbed connections built in Section \ref{sec:inductiveapproximate}. Our main goal in this section is to prove:
\begin{proposition}
\label{prop:lindiagsbound}                                                    
There exists $C> 0$ such that for all $\sigma\in \Gamma(\End E)$ orthogonal to $\Id_E$,
\begin{align}
\label{eqn:laplacianbound}
\| \Delta_k (\sigma) \|_{L^2_{d}(\omega_k)} \geq C k^{-q} \| \sigma \|_{L^2_{d+2}(\omega_k)}.
\end{align}
\end{proposition}
The proof of Proposition \ref{prop:lindiagsbound} will rely on the expansion (\ref{eq:fullexpansionlaplacian}):
\begin{equation*}
\begin{array}{ccc}
 \Delta_k & = & \Delta_{\Gr(E),k}\\
          & + &i k^{-1} \:\Lambda_{\om_B} \left(\delb_{A_{0}^{\End E}} ([\gamma_{k}^*,\cdot ])-[\gamma_{k}^*,\delb_{A_{0}^{\End E}}  ]\right)\\
          & + & i k^{-1} \: \Lambda_{\om_B} \left(\partial_{A_{0}^{\End E}}([\gamma_{k},\cdot])-[\gamma_{k},\partial_{A_{0}^{\End E}} ] \right) \\
& + & ik^{-1} \Lambda_{\omega_B} ( [\gamma_{k}, [\gamma_{k}^*, \cdot ] ] - [\gamma_{k}^*, [\gamma_{k}, \cdot ] ] ) \\
& + &  \cO(k^{-\nu} ),
 \end{array}
\end{equation*}
together with the estimates in Lemma \ref{lem:ordergammas}.
We first show that orthogonal to $\g$, we preserve the bounds we have for $\Delta_{\Gr(E),k}$ when going to the extension.
\begin{lemma}
\label{lem:oldboundsextend}
There exists $C> 0$ such that for all sufficiently large $k$ and for all $\sigma\in \Gamma(\End E)$ orthogonal to $\g$,
\begin{align*}
\| \Delta_k (\sigma) \|_{L^2_{d}(\omega_k)} \geq C k^{-1} \| \sigma \|_{L^2_{d+2}(\omega_k)}.
\end{align*}
\end{lemma}
\begin{proof}
We first show that we have the bound 
\begin{align}
\label{eq:gradedobjbound}
\| \Delta_{\Gr(E),k} (\sigma) \|_{L^2_{d}(\omega_k)}  \geq C k^{-1} \| \sigma \|_{L^2_{d+2}(\omega_k)},
\end{align}
for all $\sigma \in \Gamma(\End E) \cap \left( \oplus_i \langle \Id_{\cG_i} \rangle \right)^{\perp}$.
On the diagonal elements, this follows from the individual bounds on the $\cG_i$, which are obtained from Proposition \ref{prop:inversebound}. On the off-diagonal elements, it uses that the estimate is equivalent to a Poincar\'e inequality for the induced connection on the $\Hom(\cG_i, \cG_j)$. Recall that $\g=\oplus_i \langle \Id_{\cG_i} \rangle$ since the $\cG_i$ are all non-isomorphic, and that by Remark \ref{rem:pullbacknonisom}, we also have that $\Lie(\Aut(\pi^*\Gr(E)))=\oplus_i \langle \Id_{\pi^* \cG_i} \rangle$. Then for all $\sigma \in \Gamma(\End E) \cap \g^{\perp}$, we have the bound (\ref{eq:gradedobjbound})
or equivalently by self-adjointness and the fact that the Laplacian only has non-positive eigenvalues,
$$
\vert \langle \Delta_{\Gr(E),k} (\sigma), \sigma \rangle \vert \geq C k^{- 1} \| \sigma \|^2.
$$

First suppose $\sigma = \sigma_i \in \Gamma(\End \cG_i)$. From Equation \eqref{eq:fullexpansionlaplacian}, we deduce that 
$$
\Pi_{\Gamma(\End \cG_i)} \Delta_{k} (\sigma_i) = \Delta_{\cG_i,k} (\sigma_i) + O(k^{-q_i} \| \sigma_i \|),
$$
since the only potential additional diagonal contribution comparing $\Delta_{k}$ to the product Laplacian comes from the term 
$$
k^{-1} \Lambda_{\omega_B} ( [\gamma_{k}, [\gamma_{k}^*, \cdot ] ] - [\gamma_{k}^*, [\gamma_{k}, \cdot ] ] )
$$
in the expansion. Using Lemma \ref{lem:ordergammas}, this term is $O(k^{-q_i})$. Since $q_i\geq 2$, we therefore get  from the bound on the graded object that
\begin{align*}
\vert \langle \Delta_{k} (\sigma_i), \sigma_i \rangle \vert  =& \vert \langle \Delta_{\cG_i, k} (\sigma_i), \sigma_i \rangle \vert  + O(k^{-q_i}) \\
\geq& C k^{- 1} \| \sigma_i \|^2
\end{align*}
too, for a potentially different constant $C$. This settles the case when $\sigma$ lies in $\Gamma(\End \cG_i)$ for some $i$.

The case when $\sigma$ is off-diagonal, say $\sigma \in \Gamma(\cG_j^* \otimes \cG_i)$ for some $i,j$ is the same. Again, we have 
$$
\| \Delta_{\Gr (E),k} (\sigma) \| \geq C k^{-1} \| \sigma \|,
$$
from the estimate \eqref{eq:gradedobjbound}.  From Equation \eqref{eq:fullexpansionlaplacian}, we see that the only additional contribution that lies in $\Gamma(\cG_j^* \otimes \cG_i)$ comes from 
$$
k^{-1} \Lambda_{\omega_B} ( [\gamma_{k}, [\gamma_{k}^*, \cdot ] ] - [\gamma_{k}^*, [\gamma_{k}, \cdot ] ] ).
$$
Again, this decays to a higher order than $k^{-1}$, giving us the required bound.

At this point, we have shown that the estimate holds for $\sigma \in \Gamma(\cG_j^* \otimes \cG_i)$ for any $i,j$. In general, $\sigma$ is a sum of such terms, and we need to ensure that the estimate holds also for such sums. To ease notation, we will consider only the case when 
$$
\sigma = \sigma_i + \sigma_j
$$
for $i<j$, where $\sigma_i \in \Gamma(\End \cG_i)$ and $\sigma_j \in \Gamma (\End \cG_j)$. The case when one or both are off-diagonal, or when there are more than two such terms, is similar.

We need to estimate 
\begin{align*}
- \langle \Delta_{k} (\sigma), \sigma \rangle  =& - \langle \Delta_{k} (\sigma_i), \sigma_i \rangle - \langle \Delta_{k} (\sigma_j), \sigma_j \rangle \\
&- \langle \Delta_{k} (\sigma_i), \sigma_j \rangle - \langle \Delta_{k} (\sigma_j), \sigma_i \rangle 
\end{align*}
from below. From the estimates we have already considered, we need to understand
$$
- \langle \Delta_{k} (\sigma_i) , \sigma_j \rangle
$$
and
$$
- \langle \Delta_{k} (\sigma_j) , \sigma_i \rangle.
$$
We do this through the expansion of $\Delta_{k}$ given in Equation \eqref{eq:fullexpansionlaplacian}. In computing $\Delta_{k} (\sigma_i)$, the leading order term that can give a potential $\End(\cG_j)$ component comes from
$$
k^{-1} \Lambda_{\omega_B} ( [\gamma_k, [\gamma_k^*, \sigma_i ] ] - [\gamma_k^*, [\gamma_k, \sigma_i ] ] ),
$$
since $\Delta_{\Gr(E),k} (\sigma_i)$ is a section of $\End (\cG_i)$ and the two remaining terms are off-diagonal. The relevant term is then
$$
k^{-1} \Lambda_{\omega_B} ( (\gamma_k)^*_{ji} \wedge \sigma_i \wedge (\gamma_k)_{ij}).
$$
Recalling the expansion in Lemma \ref{lem:ordergammas} of $\gamma_k$, we see that there is a positive constant $C$ only depending on $i,j$, not the choice of $\sigma_i, \sigma_j$, such that 
$$
\vert \langle \Delta_{k} (\sigma_i) , \sigma_j \rangle \vert \leq  C k^{-2} \| \sigma_i \| \| \sigma_j \|,
$$
since each $q_l \geq 2.$ From Young's inequality with exponent 2, we therefore get that 
$$
\vert \langle \Delta_{k} (\sigma_i) , \sigma_j \rangle \vert \leq  C k^{-2} (\| \sigma_i \|^2 +  \| \sigma_j \|^2)
$$
and thus
\begin{align*}
- \langle \Delta_{k} (\sigma_i) , \sigma_j \rangle \geq - C k^{-2} (\| \sigma_i \|^2 +  \| \sigma_j \|^2).
\end{align*}
While this is a negative lower bound, it is order $k^{-2}$, and so can be compensated for when $k\gg0$ by the $O(k^{-1})$ bound for $-\langle \Delta_{k} (\sigma_i), \sigma_i \rangle$ and $- \langle \Delta_{k} (\sigma_j), \sigma_j \rangle$. The case of $\langle \Delta_{k} (\sigma_j) , \sigma_i \rangle $ gives a similar bound with the same method. Thus there is a potentially different constant $C>0$ such that
\begin{align*}
-\langle \Delta_{k} (\sigma), \sigma \rangle \geq& C k^{-1} \| \sigma \|^2,
\end{align*}
which is what we wanted to show.
\end{proof}
We now begin to extend this bound to every endomorphism orthogonal to $\Id_E$.
\begin{lemma}
\label{lem:diagonalautoms}
There exists $C> 0$ such that for all $\sigma \in \g\cap\langle \Id_E \rangle^\perp$,
\begin{align*}
\| \Delta_k (\sigma) \|_{L^2_{d}(\omega_k)} \geq C k^{-q} \| \sigma \|_{L^2_{d+2}(\omega_k)}.
\end{align*}
\end{lemma}
\begin{proof}
By self-adjointness of the Laplace operator and the fact that it has only non-positive eigenvalues, it suffices to establish the estimate
$$
- \langle \Delta_k (\sigma) , \sigma \rangle \geq C k^{-q} \| \sigma \|^2
$$
for all 
$$
\sigma = \sum_{i=1}^{\ell} \frac{c_i}{\rk (\cG_i)} \Id_{\cG_i}
$$
with $\sum_i c_i = 0$. In the proof, $C$ will stand for a positive constant, whose value might vary along the argument.
To simplify notations, set $r_i=\rk(\cG_i)$ and $\tilde c_i=\frac{c_i}{r_i}$. Let then $\sigma = \sum_{i=1}^{\ell} \tilde c_i \Id_{\cG_i}$ with $\sum_i r_i \tilde c_i =0$.
To conclude the proof it will be enough to obtain an estimate of the form:
$$
- \langle \Delta_k (\sigma) , \sigma \rangle \geq Ck^{-q}\displaystyle\sum_{i=1}^\ell \tilde c_i^2.
$$
Using formula (\ref{eq:pigdeltaIdiprecise}) in the proof of Lemma \ref{lem:linearisationIdis}, we obtain :
$$
\begin{array}{ccc}
\langle \Delta_k (\sigma) , \sigma \rangle & = &  \displaystyle\sum_{i=1}^\ell \tilde c_i \sum_{j=1}^\ell \tilde c_j \langle \Pi_\g \Delta_k \Id_i, \Id_j \rangle \\
 & & \\
 & = & \displaystyle \sum_{i=1}^\ell \tilde c_i \left( \tilde c_i (a_{i,i+1} k^{-q_i} + a_{i-1,i}k^{-q_{i-1}})\right)\\
  &  &  \\
  & & \displaystyle -\sum_{i=1}^\ell \tilde c_i \left(\tilde c_{i-1} a_{i-1,i}k^{-q_{i-1}} + \tilde c_{i+1} a_{i,i+1}k^{-q_i}\right) -\sigma_h 
\end{array}
$$
where $\sigma_h$ stands for higher order terms:
$$
\sigma_h =\displaystyle \sum_{i=1}^\ell \tilde c_i \left(\sum_{j\leq i-2} \tilde c_j a_{ji}k^{-q_j-\ldots-(q_{i-1}-1)}+\sum_{j\geq i+2} \tilde c_j a_{ij}k^{-q_i-\ldots-(q_{j-1}-1)} \right).
$$
Rearranging terms according to their order in $k$ this gives:
\begin{equation}
 \label{eq:Deltaksigmainlieg}
\begin{array}{ccc}
\langle \Delta_k (\sigma) , \sigma \rangle & = & \displaystyle\sum_{i=1}^{\ell-1} a_{i,i+1} k^{-q_i} (\tilde c_i -\tilde c_{i+1})^2 -\sigma_h.
\end{array}
\end{equation}
At this point it is useful to remember that for all $i,j\in[\![ 1, \ell ]\!]$, $a_{ij}\leq 0$, with $a_{i,i+1}<0$ for $i\leq \ell-1$. 

We deal with the terms in $\sigma_h$ as in the proof of Lemma \ref{lem:oldboundsextend} by completing the squares. Indeed, consider for example $j\leq i-2$:
$$
\begin{array}{ccc}
 \tilde c_i\tilde c_j a_{ji}k^{-q_j-\ldots -(q_{i-1}-1)} & = & \displaystyle\frac{a_{ji}}{2}k^{-q_j-\ldots -(q_{i-1}-1)}\left((\tilde c_i + \tilde c_j)^2 -\tilde c_i^2 -\tilde c_j^2 \right)\\
  & & \\
 & \geq & -Ck^{-q_j-\ldots -(q_{i-1}-1)}(\tilde c_i^2 +\tilde c_j^2)\\
 & & \\
 & \geq & -C( k^{-q_j-1}\tilde c_j^2 + k^{-q_{i-1}-1}\tilde c_i^2)
\end{array}
$$
for a constant $C>0$. As in the expression of $\langle \Delta_k (\sigma) , \sigma \rangle$, the term $\tilde c_i^2$ appears at orders $k^{-q_i}$ and $k^{-q_{i-1}}$ and so we can compensate the potentially negative contributions that will be coming from $\sigma_h$ at strictly higher orders. 

It remains to estimate
$$
 - \displaystyle\sum_{i=1}^{\ell-1} a_{i,i+1} k^{-q_i} (\tilde c_i -\tilde c_{i+1})^2.
$$
We thus need to prove that there is $C>0$ such that
$$
- \displaystyle\sum_{i=1}^{\ell-1} a_{i,i+1} k^{-q_i} (\tilde c_i -\tilde c_{i+1})^2\geq Ck^{-q}\sum_{i=1}^\ell \tilde c_i^2.
$$
We clearly have $C>0$ satisfying
$$
- \displaystyle\sum_{i=1}^{\ell-1} a_{i,i+1} k^{-q_i} (\tilde c_i -\tilde c_{i+1})^2\geq Ck^{-q}Q(\tilde c_i),
$$
where $Q$ is the real quadratic form in $\ell$-variables defined by
$$
Q(x)=\displaystyle\sum_{i=1}^{\ell-1} (x_i -x_{i+1})^2.
$$
The kernel of $Q$ is given by the vector space spanned by ${\bf 1}=(1,\ldots, 1)$. Denote by $\phi : \R^\ell\to \R$ the map $\phi(x)=\sum_i r_i x_i$. Then, the subspace spanned by ${\bf 1}$ and $\ker \phi$ are complementary. Thus the restriction of $Q$ to $\ker \phi$ is a positive definite quadratic form, hence its associated bilinear form has only strictly positive eigenvalues, and there is a constant $C>0$ such that for all $x\in \ker \phi$ 
$$
Q(x)\geq C \sum_{i=1}^\ell x_i^2.
$$
As $(\tilde c_i)\in \ker \phi$, the result follows.
\end{proof}

We are now ready to prove the main result of this subsection, which gives a lower bound for the first eigenvalue of the linearised operator, in the family of complex structures we are considering.
\begin{proof}[Proof of Proposition \ref{prop:lindiagsbound}]
 From Lemmas \ref{lem:oldboundsextend} and \ref{lem:diagonalautoms}, we know the result on $\g^{\perp}$ and $\g \cap \langle \Id_E \rangle^{\perp}$ individually. What is left is therefore to make sure that no mixed terms interfere with the estimate. We only outline the argument as it is similar to the corresponding part of the proof of Lemma \ref{lem:oldboundsextend}. Suppose $\sigma_1\in\g^\perp$ and that $\sigma_2$ is a diagonal automorphism of $\oplus_i \cG_i$, orthogonal to $\Id_E$, satisfying the bounds 
\begin{equation}
 \label{eq:lowerboundsigma1}
\|  \Delta_{k} \sigma_1 \| \geq C k^{-1} \|  \sigma_1 \|
\end{equation}
and
\begin{equation}
 \label{eq:lowerboundsigma2}
\|  \Delta_{k} \sigma_2 \| \geq C k^{-q} \|  \sigma_2 \|.
\end{equation}
Then, to extend the estimate to $\sigma_1+\sigma_2$, by self-adjointness and non-positiveness of the eigenvalues of the Laplacian, we only need to consider the term $-\langle  \Delta_{k} \sigma_2 , \sigma_1 \rangle$. Write $\sigma_2 = \sum_i \tilde c_i \Id_{\cG_i}$ with $\sum_i \rk(\cG_i) \tilde c_i=0$ and $\sigma_1=\sum_{j,l}\sigma_{jl}$ with $\sigma_{jl}\in\Gamma(\cG_l^*\otimes\cG_j)\cap\g^\perp$. We thus need to control 
$$
-\langle \Delta_k (c_i \Id_{\cG_i}), \sigma_{jl}\rangle
$$
in  terms of $\vert \tilde c_i \vert^2$ and $\vert\vert\sigma_{jl}\vert\vert^2$ for all $i,j,l$. We will give the argument for the ``worst case scenario'' when $j=i$ and $l=i+1$, the other cases being similar. 

By Equation \eqref{eq:fullexpansionlaplacian} giving the expansion of $\Delta_k$, we see that the higher order contribution of 
$$
\Pi_{\Gamma(\cG_{i+1}^*\otimes\cG_i)}\Delta_k(\tilde c_i\Id_{\cG_i})
$$
will come from the terms in
$$ 
\Pi_{\Gamma(\cG_{i+1}^*\otimes\cG_i)}\left(k^{-1} \Lambda_{\om_B}\left(\: \delb_{A_{0}^{\End E}}([\gamma_{k}^*,\tilde c_i \Id_{\cG_i} ]) + \partial_{A_{0}^{\End E}}([\gamma_{k},\tilde c_i\Id_{\cG_i}])\right)\right).
$$
By Lemmas \ref{lem:tracesandda} and \ref{lem:ordergammas}, this term is $\cO(k^{-m_i-1-\frac{1}{2}})$. Thus, there is a constant $C>0$ independent of $\sigma_1$ and $\sigma_2$ such that
$$
\begin{array}{ccc}
- \langle \Delta_k (c_i \Id_{\cG_i}), \sigma_{i,i+1}\rangle  & \geq & - C  k^{-m_i-1-\frac{1}{2}} \vert \tilde c_i \vert \cdot \vert\vert \sigma_{i,i+1} \vert\vert\\
 & & \\
 &\geq &  -C   \vert  k^{-m_i-\frac{1}{2}-\frac{1}{4}} \tilde c_i \vert \cdot \vert\vert k^{-\frac{1}{2}-\frac{1}{4}}\sigma_{i,i+1} \vert\vert \\
 & & \\
 & \geq &  -  C ( k^{-q_i -\frac{1}{2}} \vert \tilde c_i \vert^2 + k^{-1-\frac{1}{2}} \vert\vert \sigma_{i,i+1} \vert\vert^2)
\end{array}
$$
where last inequality is obtained by using Young's inequality as in the proof of Lemma \ref{lem:oldboundsextend} and using that $q_i=2m_i+1$. From Lemma \ref{lem:oldboundsextend}, the contribution of $ \vert\vert\sigma_{i,i+1}\vert\vert^2$ in the lower bound (\ref{eq:lowerboundsigma1}) is $\cO(k^{-1})$. On the other hand, from Equation (\ref{eq:Deltaksigmainlieg}) in the proof of Lemma \ref{lem:diagonalautoms}, the contribution of $\vert \tilde c_i \vert^2$ is $\cO(k^{-\min(q_{i-1},q_i)})$. Hence, the positive contributions coming from the term $\langle \Delta_k (\tilde c_i \Id_{\cG_i}), \sigma_{i,i+1}\rangle $ will only affect the negative ones coming from the terms $\langle \Delta_k \tilde c_i \Id_{\cG_i}, \tilde c_i \Id_{\cG_i} \rangle$ and $\langle \Delta_k \sigma_{i,i+1}, \sigma_{i,i+1} \rangle$ at strictly higher orders and thus won't spoil the estimate.  The argument is similar for the other components, ultimately giving the desired inequality.
\end{proof}

The above bounds hold for our initial approximate solution to the HYM equation on $\pi^* E$. In Section \ref{sec:inductiveapproximate}, we perturbed this connection and we need the same result also for these perturbed connections.
\begin{proposition}\label{prop:lindiagsbound2}
Let $\tilde A_k$ be a connection constructed from $A_k$ as in Section \ref{sec:inductiveapproximate}. Then there is a constant $C>0$, depending on $\tilde A_k$, such that for all $\sigma\in \Gamma(\End E)$ orthogonal to $\Id_E$,
\begin{align}
\label{eqn:laplacebound2}
\| \Delta_k (\sigma) \|_{L^2_{d}(\omega_k)} \geq C k^{-q} \| \sigma \|_{L^2_{d+2}(\omega_k)}.
\end{align}
\end{proposition}
\begin{proof}
Each such $\tilde A_k$ will have an associated expansion
$$
\tilde A_k^{0,1} = \delb_0 + \tilde \gamma_k
$$
analogously to $\gamma_k={\bf g}_{\underline{\lambda}, \underline{m}}\cdot \gamma$ for $A_k^{0,1}$. The point is then that under the changes made from $A_k$ in producing $\tilde A_k$, the relevant leading order contributions in $\gamma_k$ and $\tilde \gamma_k$ agree. The subleading order terms are different, but these only affect the constant $C$, not the rate. Thus, using exactly the same method of proof as for $A_k$, there is for each such $\tilde A_k$ a constant $C$ such that the bound \eqref{eqn:laplacebound2} holds.
\end{proof}

\subsubsection{Conclusion}
\label{sec:conclusionproofss}
The proof of Theorem \ref{thm:semistablemain} now follows by the same method as in the stable case in Section \ref{sec:nonlinear}. The only difference is that since the bound in Proposition \ref{prop:lindiagsbound2} has a higher order of $k$, we need a higher order approximate solution to apply the quantitative implicit function Theorem \ref{thm:quantimpl}. Following the same argument as in the proof of Theorem \ref{thm:stablecase}, we see that the required approximate solutions are of order greater than $2q+1.$ But from Section \ref{sec:inductiveapproximate}, we know that we can reach any order. This completes the proof of Theorem \ref{thm:semistablemain}.

\subsection{Some consequences}
\label{sec:consequences}
In this section, we gather some consequences of our main result and its proof. The first is the proof of Corollary \ref{cor:semistablecor}. We will need for this an auxiliary result.  Assume we have an extension 
$$
0 \to E_1 \to E \to E_2 \to 0,
$$
where $E_1=\cF_j$ for some $1\leq j\leq \ell-1$. Denote by $\delb_{i}$ the Dolbeault operator of $E_i$, for $i\in\lbrace 1,  2 \rbrace$, and set the diagonal operator $\delb_{1+2}=\delb_{1}+\delb_{2}$ on $E_1\oplus E_2$. Note in particular that there is $\gamma_i\in\Om^{0,1}(B,\End(E_i))$ such that $\delb_i=(\delb_0)_{\vert E_i}+\gamma_i$ and $\delb_0^*\gamma_i=0$ for $i\in\lbrace 1,2\rbrace$. 
From \cite[Proposition 4.5]{BuSchu}, we extract the following result that will enable us to compare symmetries for $E_1\oplus E_2$ to symmetries for $\Gr(E)$.
\begin{proposition}
\label{prop:comparisonsymmetries}
 There exists $\epsilon >0$ depending on $d_{A_0}$ such that if $\vert\vert \gamma_i\vert\vert_{\infty} < \epsilon$ for $i\in \lbrace 1, 2\rbrace$, then any element $\sigma\in H^0(B,\End(E_1\oplus E_2))$ satisfies $\delb_0 \sigma=0$.
\end{proposition}
Note that in the above statement the norms $\vert\vert \gamma_i\vert\vert_{\infty}$ are determined by $\om_B$ and $h$, and hence are fixed throughout the argument that follows.
\begin{proof}
 The result follows from \cite[Lemma 4.1, Corollary 4.2 and Proposition 4.5]{BuSchu}. We recall the proof for a more self contained exposition. Let $\sigma$ be a holomorphic section in $H^0(B,\End(E_1\oplus E_2))$. It satisfies the equation
 \begin{equation}
  \label{eq:symmetryE1E2}
 \delb_0 \sigma + [\gamma_1+\gamma_2,\sigma]=0.
 \end{equation}
 Decompose $\sigma=\sigma_0+\sigma_1$ according to the Dolbeault orthogonal decomposition, with $\delb_0\sigma_0=0$ and $\sigma_1\in \ker(\delb_0)^\perp$. By Proposition \ref{prop:holomorphicisparallel}, $\partial_0 \sigma_0=0$. Then we apply $\delb_0^*=-i\Lambda_{\om_B}\partial_0$ to \eqref{eq:symmetryE1E2} and, using $\delb_0^*(\gamma_1+\gamma_2)=0$ we obtain
 $$
 \delb_0^*\delb_0 \sigma_1+ \delb_0^*[\gamma_1+\gamma_2,\sigma_1]=0.
 $$
 Pairing with $\sigma_1$, we deduce
 $$
 \begin{array}{ccc}
 \vert\vert \delb_0 \sigma_1 \vert\vert_{L^2(\om_B)}^2 & = & -\langle [\gamma_1+\gamma_2,\sigma_1], \delb_0 \sigma_1 \rangle\\
                                        &  \leq  & c \vert\vert \gamma_1 + \gamma_2 \vert\vert_\infty \vert\vert\sigma_1 \vert\vert_{L^2(\om_B)} \vert\vert\delb_0 \sigma_1 \vert\vert_{L^2(\om_B)}
 \end{array}
 $$
 for some constant $c>0$. Using $\sigma_1\in \ker(\delb_0)^\perp$ together with H\"ormander's estimate, there is a constant $c'>0$ independent of $\sigma_1$ such that 
 $$
 \vert\vert\sigma_1 \vert\vert_{L^2(\om_B)} \leq c'\vert\vert\delb_0 \sigma_1 \vert\vert_{L^2(\om_B)}
 $$
 and thus another $c''>0$ satisfying
 $$
 \vert\vert \delb_0 \sigma_1 \vert\vert_{L^2(\om_B)}\leq c'' \vert\vert \gamma_1 + \gamma_2 \vert\vert_\infty \vert\vert\delb_0 \sigma_1 \vert\vert_{L^2(\om_B)} .
 $$
 Then, for $\vert\vert \gamma_1 + \gamma_2 \vert\vert_\infty$ small enough, $\sigma_1=0$, $\sigma=\sigma_0$ and the result follows.
\end{proof}
\begin{corollary}
 \label{cor:symmetrypreserveGI}
 For $\epsilon >0$ as in Proposition \ref{prop:comparisonsymmetries}, assuming that $\vert\vert \gamma_i\vert\vert_{\infty} < \epsilon$ for $i\in \lbrace 1, 2\rbrace$, we have $\aut(E_1\oplus E_2)=\C\Id_{E_1}\oplus\C\Id_{E_2}$.
\end{corollary}
\begin{proof}
First, a reformulation of Proposition \ref{prop:comparisonsymmetries} gives 
 $ \aut(E_1\oplus E_2)\subset \g$. Then, if $\sigma\in \aut(E_1)$,  $\sigma\in \g$.  By the description of $\g$ in Lemma \ref{lem:generalcokernel}, and recalling that $E_1=\cF_j$, we see that $\Aut(E_1)\subset \C^*\Id_{\cG_1}\times \ldots\times \C^*\Id_{\cG_j}$. Given the action of $\Aut(E_1)$ on $\gamma_1$ (as in Equation \eqref{eq:Autactionongamma}) and the fact that $\gamma_{i,i+1}\neq 0$ for $1\leq i \leq j-1$, we deduce that $\Aut(E_1)=\C^* \Id_{E_1}$. The same argument on $E_2$ gives the result.
\end{proof}
From Proposition \ref{prop:comparisonsymmetries} we can prove the following technical lemma.
\begin{lemma}
\label{lem:decompstab}
Let $E$ be a semistable bundle on $B$ which satisfies $(H1)-(H3)$ and is asymptotically stable with respect to subbundles induced from a Jordan--H\"older filtration. Fix $i_1\in[\![ 1, l-1 ]\!]$ such that
\begin{enumerate}
 \item[(i)] For all $j\in [\![ 1, l-1 ]\!]$, $\mu_\infty(\cF_j)\leq \mu_\infty(\cF_{i_1})$
 \item[(ii)] The rank $r_1$ of $\cF_{i_1}$ satisfies
 $$
 r_1=\min\lbrace \rank(\cF_j),\:j\in [\![ 1, l-1 ]\!]\:\mathrm{ satisfies } \: (i)\rbrace.
 $$
\end{enumerate}
That is, $S:=\cF_{i_1}$ has a maximal adiabatic slope amongst strict subbundles $\cF_j\subset E$ and minimal rank amongst those bundles with the same adiabatic slope. Set $Q:=E/S$. Then $S$ and $Q$ are simple, semistable, satisfy $(H1)-(H3)$, and are asymptotically stable with respect to subbundles induced from a Jordan--H\"older filtration.
\end{lemma}
\begin{proof}
 To see that $S$ is semistable, suppose that  $\cF\subset S$. Then $\cF\subset E$ and thus $\mu_L(\cF)\leq \mu_L(E)=\mu_L(S)$. Moreover, the graded object of $S$ is $\bigoplus_{i=1}^{i_1}\cG_i$ and is locally free. Since any subbundle of $S$ obtained from its induced Jordan--H\"older filtration is a subbundle of $E$ induced from the given Jordan--H\"older filtration of $E$, and $S$ was chosen to have maximal asymptotic slope and minimal rank among such subbundles, $S$ is asymptotically stable with respect to subbundles from its Jordan--H\"older filtration. Finally, hypothesis $(H1)$ and $(H2)$ follow from the same assumptions on $E$ and from $S=\cF_{i_1}$.

Next, we consider $Q$. Using similar arguments as for $S$, $Q$ is semistable and satisfies $(H1)-(H3)$. We next show that $Q$ is stable with respect to subbundles coming from its Jordan--H\"older filtration. 

Stability with respect to these subbundles is equivalent to the fact that for all $\cF_j$, with $S \subsetneq \cF_j$, one has $\mu_\infty(Q) <\mu_\infty(E/\cF_j)$. Let $\cF_j\subset E$, with $S \subsetneq \cF_j$. By point $(i)$, $\mu_\infty(\cF_j)\leq\mu_\infty(S)$. Since $\mu_\infty(S)<\mu_\infty(E)$ it follows by Corollary \ref{cor:subadditivity of slopes} that
$$
\mu_\infty(S )<\mu_\infty(E)<\mu_\infty(Q)
$$
and
$$
\mu_\infty(\cF_j/S )\leq \mu_\infty(\cF_j) \leq \mu_\infty(S).
$$
Combining these inequalities we obtain
$$
\mu_\infty(\cF_j/S)<\mu_\infty(Q)
$$
and by Corollary \ref{cor:subadditivity of slopes} again:
$$
\mu_\infty(Q)<\mu_\infty(E/\cF_j)
$$
where we used $Q=E/S$, and that $E/\cF_j \cong \frac{E/S}{\cF_j/S}.$
\end{proof}
We can now prove Corollary \ref{cor:semistablecor}.
\begin{proof}[Proof of Corollary \ref{cor:semistablecor}]
 Let $E$ be a semistable vector bundle on $(B,L)$ satisfying $(H1)-(H3)$. Assume that for all $i\in[\![1,\ell-1]\!]$, one has $\mu_\infty(\cF_i)\leq \mu_\infty(E)$ with at least one equality. We want to show that $\pi^*E$ is strictly semistable on $X$ with respect to adiabatic polarisations. As in the proof of Lemma \ref{lem:decompstab}, consider $i_1$ such that the subbundle $G_1:=\cF_{i_1}\subset E$ satisfies $\mu_\infty(\cF_{i_1})=\mu_\infty(E)$ and has minimal rank amongst the bundles $\cF_i$ satisfying $\mu_\infty(\cF_i)=\mu_\infty(E)$. Then, following the same argument as in the proof of Lemma \ref{lem:decompstab}, we deduce that:
 \begin{itemize}
 \item[$(i)$] We have $\mu_\infty(G_1)=\mu_\infty(E)=\mu_\infty(E/G_1)$.
  \item[$(ii)$] $G_1$ is simple, semistable, satisfies $(H1)-(H3)$, and is asymptotically stable with respect to subbundles coming from a Jordan--H\"older filtration.
  \item[$(iii)$] For all $i$ with $G_1\subsetneq \cF_i$, $\mu_\infty(E/G_1)\leq \mu_\infty(E/\cF_i)$.
 \end{itemize}
By Theorem \ref{thm:semistablemain}, $\pi^*G_1$ is stable for adiabatic polarisations. If the inequalities in $(iii)$ are all strict, we stop at this stage, and set $F_1:=E/G_1$. Then $\pi^*F_1$ is stable with same asymptotic slope as $\pi^*G_1$. Then, as discussed before, $\pi^*E$ is a small deformation of $\pi^*G_1\oplus \pi^*F_1$, which is polystable with respect to $L_k$, for $k \gg 1$. But semistability is an open condition for flat families, and thus $\pi^*E$ is semistable (see e.g. \cite[Proposition 2.3.1]{HuLe}).

If some of the inequalities in $(iii)$ are actually equalities, iterating the argument on $F_1$ (or rather its dual), we can decompose $F_1$ as an extension
\begin{equation*}
 \label{eq:extensionsemistableargument}
 0 \to G_2\to F_1 \to F_1/G_2 \to 0
\end{equation*}
 with $G_2, F_1$ and $F_1/G_2$ satisfying $(i), (ii)$ and $(iii)$ as above. In particular, $\pi^*G_2$ is stable for adiabatic polarisations and $\mu_\infty(G_2)=\mu_\infty(F_1)=\mu_\infty(G_1)$. By induction, we see that $E$ is obtained by a sequence of extensions of this form, and thus $\pi^*E$ is a small deformation of a polystable bundle with respect to $L_k$ with $k\gg 1$. Then, $\pi^*E$ is semistable for adiabatic polarisations.
 To conclude, from $\mu_\infty(G_1)=\mu_\infty(E)$, we see that $\pi^*E$ is strictly semistable.
\end{proof}
It is interesting to notice that in the course of the proof of Corollary \ref{cor:semistablecor}, we described how to recover a Jordan--H\"older filtration for $\pi^*E$ out of a Jordan--H\"older filtration for $E$. In particular:
\begin{corollary}
 \label{cor:JHfiltrationpullback}
 Let $E$ be a semistable vector bundle on $(B,L)$ satisfying $(H1)-(H3)$. Assume that for all $i\in[\![1,\ell-1]\!]$, one has $\mu_\infty(\cF_i)\leq \mu_\infty(E)$. Then, the stable components of $\Gr(\pi^*E)$ are direct sums of pullbacks of stable components of $\Gr(E)$. More precisely, a Jordan--H\"older filtration for $\pi^*E$ is given by $$
0\subset \pi^*\cF_{i_1}\subset \ldots \subset \pi^*\cF_{i_\ell}=\pi^*E
$$
where the $i_j$'s are precisely the indices with $\mu_\infty(\cF_{i_j})=\mu_\infty(E)$.
\end{corollary}

%%%%
%Section 6
%%%%

\section{Applications}
\label{sec:examples}
In this section we investigate our results in various situations. As before, $(B,L)$ stands for a compact polarised K\"ahler manifold of dimension $n$, $\pi: X \to B$ is a holomorphic submersion with connected fibres with relatively ample polarisation $H$. We let $E$ be a simple holomorphic vector bundle of rank $r$ over $B$.
\subsection{Some trivial cases}
\label{sec:trivialcases}
We present here three cases where (semi)stability is automatically preserved by pullback for adiabatic classes:
\begin{corollary}
 \label{cor:trivialcases}
 Assume that at least one of the three following conditions holds:
 \begin{enumerate}
  \item[(i)] $B$ is a Riemann surface,
  \item[(ii)] $X\to B$ is trivial, and $H$ is the pullback of a polarisation on $F$,
  \item[(iii)] the K\"ahler cone of $B$ is one dimensional.
   \end{enumerate}
 Then the pullback of a slope stable (resp. strictly semistable, resp. unstable) vector bundle is again slope stable (resp. strictly semistable, resp. unstable) with respect to adiabatic classes, assuming the graded object to be locally free in the semistable case.
\end{corollary}
Note that any torsion-free sheaf on a Riemann surface is locally-free, so that the assumption on the graded object is automatically satisfied in case $(i)$ above.
\begin{proof}
Let $\cE$ be any coherent torsion-free sheaf on $B$. If $(i)$ or $(ii)$ is satisfied, then, for all $k\gg 0$, we have
$$
\mu_k(\pi^*\cE)=k^{n-1}\binom{m+n-1}{n-1}\mu_L(\cE)\cdot c_1(H)^m.
$$
Assuming $(iii)$,  there is a unique real constant $\alpha_\cE$ such that the torsion-free part of $c_1(\cE)$ equals $\alpha_\cE c_1(L)$. Thus 
$
\mu_L(\cE)=\frac{\alpha_\cE}{\rank(\cE)}c_1(L)^n
$
and then
$$
\mu_k(\pi^*\cE)=\frac{\mu_L(\cE)}{c_1(L)^n}\,c_1(L)\cdot(kc_1(L)+c_1(H))^{m+n-1}.
$$
In any case, there is a positive constant $c_k$ only depending on $(X,H)\to (B,L)$ and $k$ such that for any torsion-free coherent sheaf $\cE$ on $B$, we have $\mu_k(\cE)=c_k \mu_L(\cE)$. We deduce that the pullback of a Jordan--H\"older filtration for $\cE$ is a Jordan--H\"older filtration for $\pi^*\cE$.  Thus, from Theorem \ref{thm:stablecase}, Proposition \ref{prop:unstable case} and Proposition \ref{prop:semistable semistable case}, the result follows.
\end{proof}

\subsection{Fibrations equal to the base}
\label{sec:XisBcase}
Another case that is worth investigating is when $X=B$, or when the fibre is reduced to a point. In that case, $H$ is nothing but another line bundle (not necessarily positive) on $X$. Then, slope stability with respect to the adiabatic classes $L_k$, $k\gg 0$, is equivalent to slope stability with respect to $L_\ep=L+\ep H$ for $\ep=k^{-1}\ll 1$. Thus, we are considering small perturbations of the polarisation defining the stability notion.

This should be compared to Thaddeus's \cite{Thad} and Dolgachev and Hu's results \cite{DolHu}, who studied the variations of GIT quotients of smooth projective varieties when the linearisation of the action changes. The variations of moduli spaces of stable vector bundles induced by different polarisations have been studied in the 90's, mostly on projective surfaces, and in relation to the computation of Donaldson's polynomials (see \cite[Chapter 4, Section C]{HuLe} and the references therein). One of the main discovered features is that for a smooth projective surface $X$, and for a fixed topological type $\tau$ of vector bundles on $X$, the K\"ahler cone can be partitioned into chambers, and the moduli spaces of stable bundles on $X$ of type $\tau$ are isomorphic when the polarisation stays in a chamber, while one can relate moduli via birational transformations similar to flips when the polarisation crosses a wall between two chambers.

Our results provide, locally, but in higher dimension, further evidence for a decomposition of the cone of polarisations into chambers. First, by Theorem \ref{thm:stablecase} and Proposition \ref{prop:unstable case}, we obtain:
\begin{corollary}
 \label{cor:XisBcasestable}
 Let $E$ be a $L$-slope stable (resp. unstable) vector bundle on $X$. Then, for any $H\in \mathrm{Pic}(X)$, there is $k_0\in\N$ such that for $k\geq k_0$, $E$ is $L+k^{-1}H$-slope stable (resp. unstable).
\end{corollary}

The case when we start from a strictly semistable vector bundle is more interesting. We focus first on an illustration on $K3$ surfaces to relate our results to already observed phenomena. Let $E$ be a strictly semistable vector bundle on a polarised $K3$ surface $(X,L)$. We assume that $E$ satisfies $(H1)-(H3)$ and denote by $(\cF_i)_{1\leq i\leq \ell}$ the subbundles coming from the Jordan--H\"older filtration of $E$. The only intersection number (\ref{eq:nuslope}) relevant in our results is then, for $\cF\subset E$,
$$
\nu_2(\cF)=\frac{c_1(\cF)\cdot c_1(H)}{\rank(\cF)}.
$$
Our result \ref{thm:semistablemain} implies:
\begin{corollary}
 \label{cor:semistableK3}
 Assume that for all $i\in[\![1,\ell-1 ]\!]$,
 $$
 \nu_2(\cF_i)<\nu_2(E).
 $$
 Then, for $k\gg 0$, $E$ is slope stable with respect to $L+k^{-1}H$.
\end{corollary}
This result shows that the class
$$
\rank(\cF)c_1(E)-\rank(E) c_1(\cF)
$$
plays a crucial role in understanding whether a perturbation of the polarisation by $H$ will provide stability or instability for $E$. This should be compared to \cite[Theorem 4.C.3]{HuLe}, where this quantity is used to provide conditions on the rank and Chern classes of torsion-free sheaves that imply non-existence of strictly semistable sheaves for a given polarisation.

Returning to the general case, for a given polarised K\"ahler variety $(X,L)$ of dimension $n$ and a strictly slope semistable vector bundle $E$, still assuming $(H1)-(H3)$, we obtain the following picture. The intersection numbers $\nu_i$ (see Section \ref{sec:adiabaticslopes}) in this setting are given by:
\begin{equation}
 \label{eq:XisBnui}
 \nu_i(E)=\binom{n-1}{i-1} \frac{c_1(E)\cdot c_1(H)^{i-1}\cdot c_1(L)^{n-i}}{\rank(E)}.
\end{equation}
The zero loci of the functions
\begin{equation}
 \label{eq:chamberfunctions}
 \begin{array}{ccc}
\mathrm{Pic}(X) &  \to & \Q \\
H & \mapsto & ( \frac{c_1(E)}{\rank(E)}-\frac{c_1(\cF_j)}{\rank(\cF_j)})\cdot c_1(H)^{i-1}\cdot c_1(L)^{n-i}
\end{array}
\end{equation}
for  $(j,i)\in[\![1,\ell-1 ]\!]\times  [\![ 2, n]\!]$, cut out chambers in $\mathrm{Pic}(X)$ (or rather the Neron-Severi group) describing which small perturbations of the polarisation will send $E$ to the set of slope stable or slope unstable vector bundles. Note that even if it is not clear yet to the authors whether the results in Theorems \ref{thm:stablecase} and \ref{thm:semistablemain} can be extended uniformly to the set of all semistable vector bundles, it seems very likely that they should hold at least locally in the set of semistable vector bundles, hence supporting this local evidence for a chamber decomposition of the space of polarisations.

\subsection{Projectivisations of vector bundles}
We finish with the study of a situation where the holomorphic submersion with connected fibres is non trivial over a base of dimension greater than $2$. Assume from now that $X=\P(V)$ for a given vector bundle $V\to B$ of rank $m+1$, and that $H=\cO_X(1)$ is the Serre line bundle. Then, recall that if $\xi=c_1(H)$, the cohomology ring $H^2(X,\Z)$ is the ring over $H^2(B,\Z)$ with generator $\xi$ subject to the relation:
$$
\xi^{m+1}=c_1(V)\xi^m - c_2(V)\xi^{m-1}+\ldots + (-1)^m c_{m+1}(V).
$$
If $L$ is the polarisation on $B$, then the intersection numbers $\nu_i$ of the adiabatic slopes \ref{eq:nuslope} are given by:
\begin{lemma}
 \label{lem:nuiPVcase}
 For $\cE$ a torsion free coherent sheaf on $B$, and for $i\in\lbrace 2, \ldots , n\rbrace$, we have
\begin{equation}
 \label{eq:nuiPVcase}
 \nu_i(\cE)=(-1)^{i}\binom{n+m-1}{i+m-1}\: \mathrm{vol}(X_b) \: \frac{c_1(\cE)\cdot c_1(L)^{n-i}\cdot c_{i-1}(V)}{\rank(\cE)},
\end{equation}
where $\mathrm{vol}(X_b)=\int_{X_b} \xi^m$ is the volume of any fibre with respect to $H$.
\end{lemma}
Using this explicit description of the expansion of the adiabatic slopes together with Theorem \ref{thm:semistablemain}, the knowledge of the cohomology ring $H^*(B,\Z)$ is enough to understand whether a strictly semistable vector bundle on $B$ will lift to a stable vector bundle on $X$ for adiabatic classes.

\bibliography{pullbacks}
\bibliographystyle{amsplain}

\end{document}